\documentclass[a4paper, british]{amsart}

\usepackage{libertine}
\usepackage[libertine]{newtxmath}

\usepackage{babel}
\usepackage{enumitem}
\usepackage{hyperref}
\usepackage[utf8]{inputenc}
\usepackage{newunicodechar}
\usepackage{mathtools}
\usepackage{varioref}
\usepackage[arrow,curve,matrix]{xy}

\usepackage{colortbl}
\usepackage{graphicx}
\usepackage{tikz}

\usepackage{mathrsfs}

\definecolor{linkred}{rgb}{0.7,0.2,0.2}
\definecolor{linkblue}{rgb}{0,0.2,0.6}

\numberwithin{figure}{section}

\usepackage[hyperpageref]{backref}

\setdescription{labelindent=\parindent, leftmargin=2\parindent}
\setitemize[1]{labelindent=\parindent, leftmargin=2\parindent}
\setenumerate[1]{labelindent=0cm, leftmargin=*, widest=iiii}

\newunicodechar{א}{\ensuremath{\aleph}}
\newunicodechar{α}{\ensuremath{\alpha}}
\newunicodechar{β}{\ensuremath{\beta}}
\newunicodechar{χ}{\ensuremath{\chi}}
\newunicodechar{δ}{\ensuremath{\delta}}
\newunicodechar{ε}{\ensuremath{\varepsilon}}
\newunicodechar{Δ}{\ensuremath{\Delta}}
\newunicodechar{η}{\ensuremath{\eta}}
\newunicodechar{γ}{\ensuremath{\gamma}}
\newunicodechar{Γ}{\ensuremath{\Gamma}}
\newunicodechar{ι}{\ensuremath{\iota}}
\newunicodechar{κ}{\ensuremath{\kappa}}
\newunicodechar{λ}{\ensuremath{\lambda}}
\newunicodechar{Λ}{\ensuremath{\Lambda}}
\newunicodechar{ν}{\ensuremath{\nu}}
\newunicodechar{μ}{\ensuremath{\mu}}
\newunicodechar{ω}{\ensuremath{\omega}}
\newunicodechar{Ω}{\ensuremath{\Omega}}
\newunicodechar{π}{\ensuremath{\pi}}
\newunicodechar{Π}{\ensuremath{\Pi}}
\newunicodechar{φ}{\ensuremath{\phi}}
\newunicodechar{Φ}{\ensuremath{\Phi}}
\newunicodechar{ψ}{\ensuremath{\psi}}
\newunicodechar{Ψ}{\ensuremath{\Psi}}
\newunicodechar{ρ}{\ensuremath{\rho}}
\newunicodechar{σ}{\ensuremath{\sigma}}
\newunicodechar{Σ}{\ensuremath{\Sigma}}
\newunicodechar{τ}{\ensuremath{\tau}}
\newunicodechar{θ}{\ensuremath{\theta}}
\newunicodechar{Θ}{\ensuremath{\Theta}}
\newunicodechar{ξ}{\ensuremath{\xi}}
\newunicodechar{Ξ}{\ensuremath{\Xi}}
\newunicodechar{ζ}{\ensuremath{\zeta}}

\newunicodechar{ℓ}{\ensuremath{\ell}}
\newunicodechar{ï}{\"{\i}}

\newunicodechar{𝔸}{\ensuremath{\bA}}
\newunicodechar{𝔹}{\ensuremath{\bB}}
\newunicodechar{ℂ}{\ensuremath{\bC}}
\newunicodechar{𝔻}{\ensuremath{\bD}}
\newunicodechar{𝔼}{\ensuremath{\bE}}
\newunicodechar{𝔽}{\ensuremath{\bF}}
\newunicodechar{𝔾}{\ensuremath{\bG}}
\newunicodechar{ℕ}{\ensuremath{\bN}}
\newunicodechar{ℙ}{\ensuremath{\bP}}
\newunicodechar{ℚ}{\ensuremath{\bQ}}
\newunicodechar{ℝ}{\ensuremath{\bR}}
\newunicodechar{𝕏}{\ensuremath{\bX}}
\newunicodechar{ℤ}{\ensuremath{\bZ}}
\newunicodechar{𝒜}{\ensuremath{\sA}}
\newunicodechar{ℬ}{\ensuremath{\sB}}
\newunicodechar{𝒞}{\ensuremath{\sC}}
\newunicodechar{𝒟}{\ensuremath{\sD}}
\newunicodechar{ℰ}{\ensuremath{\sE}}
\newunicodechar{ℱ}{\ensuremath{\sF}}
\newunicodechar{𝒢}{\ensuremath{\sG}}
\newunicodechar{ℋ}{\ensuremath{\sH}}
\newunicodechar{𝒥}{\ensuremath{\sJ}}
\newunicodechar{ℒ}{\ensuremath{\sL}}
\newunicodechar{ℳ}{\ensuremath{\sM}}
\newunicodechar{𝒪}{\ensuremath{\sO}}
\newunicodechar{𝒬}{\ensuremath{\sQ}}
\newunicodechar{𝒮}{\ensuremath{\sS}}
\newunicodechar{𝒯}{\ensuremath{\sT}}
\newunicodechar{𝒲}{\ensuremath{\sW}}

\newunicodechar{∂}{\ensuremath{\partial}}
\newunicodechar{∇}{\ensuremath{\nabla}}

\newunicodechar{↺}{\ensuremath{\circlearrowleft}}
\newunicodechar{∞}{\ensuremath{\infty}}
\newunicodechar{⊕}{\ensuremath{\oplus}}
\newunicodechar{⊗}{\ensuremath{\otimes}}
\newunicodechar{•}{\ensuremath{\bullet}}
\newunicodechar{Λ}{\ensuremath{\wedge}}
\newunicodechar{↪}{\ensuremath{\into}}
\newunicodechar{→}{\ensuremath{\to}}
\newunicodechar{↦}{\ensuremath{\mapsto}}
\newunicodechar{⨯}{\ensuremath{\times}}
\newunicodechar{∪}{\ensuremath{\cup}}
\newunicodechar{∩}{\ensuremath{\cap}}
\newunicodechar{⊋}{\ensuremath{\supsetneq}}
\newunicodechar{⊇}{\ensuremath{\supseteq}}
\newunicodechar{⊃}{\ensuremath{\supset}}
\newunicodechar{⊊}{\ensuremath{\subsetneq}}
\newunicodechar{⊆}{\ensuremath{\subseteq}}
\newunicodechar{⊂}{\ensuremath{\subset}}
\newunicodechar{⊄}{\ensuremath{\not \subset}}
\newunicodechar{≥}{\ensuremath{\geq}}
\newunicodechar{≠}{\ensuremath{\neq}}
\newunicodechar{≫}{\ensuremath{\gg}}
\newunicodechar{≪}{\ensuremath{\ll}}

\newunicodechar{≤}{\ensuremath{\leq}}
\newunicodechar{∈}{\ensuremath{\in}}
\newunicodechar{∉}{\ensuremath{\not \in}}
\newunicodechar{∖}{\ensuremath{\setminus}}
\newunicodechar{◦}{\ensuremath{\circ}}
\newunicodechar{°}{\ensuremath{^\circ}}
\newunicodechar{…}{\ifmmode\mathellipsis\else\textellipsis\fi}
\newunicodechar{·}{\ensuremath{\cdot}}
\newunicodechar{⋯}{\ensuremath{\cdots}}
\newunicodechar{∅}{\ensuremath{\emptyset}}
\newunicodechar{⇒}{\ensuremath{\Rightarrow}}

\newunicodechar{⁰}{\ensuremath{^0}}
\newunicodechar{¹}{\ensuremath{^1}}
\newunicodechar{²}{\ensuremath{^2}}
\newunicodechar{³}{\ensuremath{^3}}
\newunicodechar{⁴}{\ensuremath{^4}}
\newunicodechar{⁵}{\ensuremath{^5}}
\newunicodechar{⁶}{\ensuremath{^6}}
\newunicodechar{⁷}{\ensuremath{^7}}
\newunicodechar{⁸}{\ensuremath{^8}}
\newunicodechar{⁹}{\ensuremath{^9}}
\newunicodechar{ⁱ}{\ensuremath{^i}}
\newunicodechar{⁺}{\ensuremath{^+}}

\newunicodechar{⌈}{\ensuremath{\lceil}}
\newunicodechar{⌉}{\ensuremath{\rceil}}
\newunicodechar{⌊}{\ensuremath{\lfloor}}
\newunicodechar{⌋}{\ensuremath{\rfloor}}

\newunicodechar{≅}{\ensuremath{\cong}}
\newunicodechar{⇔}{\ensuremath{\Leftrightarrow}}
\newunicodechar{∃}{\ensuremath{\exists}}
\newunicodechar{±}{\ensuremath{\pm}}

\DeclareFontFamily{OMS}{rsfs}{\skewchar\font'60}
\DeclareFontShape{OMS}{rsfs}{m}{n}{<-5>rsfs5 <5-7>rsfs7 <7->rsfs10 }{}
\DeclareSymbolFont{rsfs}{OMS}{rsfs}{m}{n}
\DeclareSymbolFontAlphabet{\scr}{rsfs}
\DeclareSymbolFontAlphabet{\scr}{rsfs}

\DeclareFontFamily{U}{mathx}{\hyphenchar\font45}
\DeclareFontShape{U}{mathx}{m}{n}{
      <5> <6> <7> <8> <9> <10>
      <10.95> <12> <14.4> <17.28> <20.74> <24.88>
      mathx10
      }{}
\DeclareSymbolFont{mathx}{U}{mathx}{m}{n}
\DeclareFontSubstitution{U}{mathx}{m}{n}
\DeclareMathAccent{\wcheck}{0}{mathx}{"71}

\DeclareMathOperator{\Aut}{Aut}

\DeclareMathOperator{\Id}{Id}

\DeclareMathOperator{\img}{img}
\DeclareMathOperator{\Pic}{Pic}

\DeclareMathOperator{\red}{red}
\DeclareMathOperator{\reg}{reg}

\DeclareMathOperator{\sing}{sing}

\DeclareMathOperator{\Sym}{Sym}
\DeclareMathOperator{\supp}{supp}

\newcommand{\sA}{\scr{A}}
\newcommand{\sB}{\scr{B}}
\newcommand{\sC}{\scr{C}}
\newcommand{\sD}{\scr{D}}
\newcommand{\sE}{\scr{E}}
\newcommand{\sF}{\scr{F}}
\newcommand{\sG}{\scr{G}}
\newcommand{\sH}{\scr{H}}

\newcommand{\sJ}{\scr{J}}

\newcommand{\sL}{\scr{L}}
\newcommand{\sM}{\scr{M}}

\newcommand{\sO}{\scr{O}}

\newcommand{\sQ}{\scr{Q}}

\newcommand{\sS}{\scr{S}}
\newcommand{\sT}{\scr{T}}

\newcommand{\sW}{\scr{W}}

\newcommand{\cC}{\mathcal C}

\newcommand{\bA}{\mathbb{A}}
\newcommand{\bB}{\mathbb{B}}
\newcommand{\bC}{\mathbb{C}}
\newcommand{\bD}{\mathbb{D}}
\newcommand{\bE}{\mathbb{E}}
\newcommand{\bF}{\mathbb{F}}
\newcommand{\bG}{\mathbb{G}}

\newcommand{\bN}{\mathbb{N}}

\newcommand{\bP}{\mathbb{P}}
\newcommand{\bQ}{\mathbb{Q}}
\newcommand{\bR}{\mathbb{R}}

\newcommand{\bX}{\mathbb{X}}

\newcommand{\bZ}{\mathbb{Z}}

\theoremstyle{plain}
\newtheorem{thm}{Theorem}[section]

\newtheorem{cor}[thm]{Corollary}
\newtheorem{defn}[thm]{Definition}
\newtheorem{fact}[thm]{Fact}
\newtheorem{lem}[thm]{Lemma}

\newtheorem{prop}[thm]{Proposition}

\theoremstyle{remark}

\newtheorem{asswlog}[thm]{Assumption w.l.o.g.}
\newtheorem{claim}[thm]{Claim}
\newtheorem{c-n-d}[thm]{Claim and Definition}

\newtheorem{construction}[thm]{Construction}

\newtheorem{example}[thm]{Example}
\newtheorem{explanation}[thm]{Explanation}
\newtheorem{notation}[thm]{Notation}
\newtheorem{obs}[thm]{Observation}
\newtheorem{rem}[thm]{Remark}

\newtheorem*{rem-nonumber}{Remark}
\newtheorem{setting}[thm]{Setting}
\newtheorem{warning}[thm]{Warning}

\numberwithin{equation}{thm}

\setlist[enumerate]{label=(\thethm.\arabic*), before={\setcounter{enumi}{\value{equation}}}, after={\setcounter{equation}{\value{enumi}}}}

\newcommand{\into}{\hookrightarrow}

\newcommand{\wtilde}{\widetilde}
\newcommand{\what}{\widehat}

\newcommand\CounterStep{\addtocounter{thm}{1}\setcounter{equation}{0}}

\newcommand{\factor}[2]{\left. \raise 2pt\hbox{$#1$} \right/\hskip -2pt\raise -2pt\hbox{$#2$}}
\newcommand{\Publication}[1]{}

\newcommand{\subversionInfo}{}
\newcommand{\svnid}[1]{}
\newcommand{\approvals}[2][Approval]{}
\renewcommand{\phi}{\varphi}
\graphicspath{{../}}

\usepackage{cancel}
\usepackage{tikz-cd}
\tikzset{commutative diagrams/arrow style=tikz}
\author{Stefan Kebekus}
\address{Stefan Kebekus, Mathematisches Institut, Albert-Ludwigs-Universität Freiburg, Ernst-Zermelo-Straße 1, 79104 Freiburg im Breisgau, Germany}
\email{\href{mailto:stefan.kebekus@math.uni-freiburg.de}{stefan.kebekus@math.uni-freiburg.de}}
\urladdr{\url{https://cplx.vm.uni-freiburg.de}}

\author{Erwan Rousseau}
\address{Erwan Rousseau, Univ Brest, CNRS UMR 6205, Laboratoire de Mathematiques de Bretagne Atlantique, F-29200 Brest, France}
\email{\href{mailto:erwan.rousseau@univ-brest.fr}{erwan.rousseau@univ-brest.fr}}
\urladdr{\href{http://eroussea.perso.math.cnrs.fr}{http://eroussea.perso.math.cnrs.fr}}

\thanks{This work started during the visit of Erwan Rousseau to the Freiburg
Institute for Advanced Studies, supported by the European Unions Horizon 2020
research and innovation program under the Marie Sklodowska-Curie grant agreement
No 75434.  Rousseau thanks the Institute for providing an excellent working
environment.}

\keywords{$\cC$-pairs, Albanese variety}

\subjclass[2020]{32C99, 32H99, 32A22}

\title{The Albanese of a $\cC$-pair}
\date{\today}

\makeindex

\makeatletter
\hypersetup{
  pdfauthor={\authors},
  pdftitle={\@title},
  pdfsubject={\@subjclass},
  pdfkeywords={\@keywords},
  pdfstartview={Fit},
  pdfpagelayout={TwoColumnRight},
  pdfpagemode={UseOutlines},
  bookmarks,
  colorlinks,
  linkcolor=linkblue,
  citecolor=linkred,
  urlcolor=linkred
}
\makeatother

\DeclareMathOperator{\alb}{alb}
\DeclareMathOperator{\Alb}{Alb}
\DeclareMathOperator{\Branch}{Branch}
\DeclareMathOperator{\diff}{d}
\DeclareMathOperator{\Div}{Div}
\DeclareMathOperator{\Fix}{Fix}
\DeclareMathOperator{\mult}{mult}

\DeclareMathOperator{\walb}{walb}

\theoremstyle{plain}

\theoremstyle{remark}
\newtheorem{choice}[thm]{Choice}
\newtheorem{conj}[thm]{Conjecture}

\newtheorem{leansum}[thm]{Summary}
\newtheorem{reminder}[thm]{Reminder}

\begin{document}

\approvals[Approval for Abstract]{Erwan & yes \\ Stefan & yes}
\begin{abstract}
\selectlanguage{british}

Written with a view towards applications in hyperbolicity, rational points, and
entire curves, this paper addresses the problem of constructing Albanese maps
within Campana's theory of $\cC$-pairs (or ``geometric orbifolds'').  It
introduces $\cC$-semitoric pairs as analogues of the (semi)tori used in the
classic Albanese theory and follows Serre by defining the Albanese of a
$\cC$-pair as the universal map to a $\cC$-semitoric pairs.  The paper shows
that the Albanese exists in relevant cases, gives sharp existence criteria, and
conjectures that a ``weak Albanese'' exists unconditionally.

\end{abstract}

\maketitle
\tableofcontents

%
%
\svnid{$Id: 01-intro.tex 949 2024-11-11 12:46:09Z kebekus $}
\selectlanguage{british}

\section{Introduction}
\subversionInfo

\approvals{Erwan & yes \\ Stefan & yes}

This paper constructs Albanese maps for $\cC$-pairs, with the goal to provide
tools for the study rational points and entire curves in algebraic varieties and
complex manifolds.  To illustrate our motivation, consider the following two
classic theorems of Faltings and Bloch-Ochiai/Kawamata.

\begin{thm}[Faltings' theorem on density of rational points, \cite{MR1109353}]\label{thm:1-1}%
  Let $X$ be a projective manifold defined over a number field $k$.  If the
  dimension of its Albanese variety satisfies $\dim \Alb(X) > \dim X$, then its
  rational points are not potentially dense.  In other words: If $k ⊆ k'$ is any
  finite field extension, then $k'$-rational points are not Zariski dense in
  $X$.  \qed
\end{thm}

\begin{thm}[\protect{Bloch-Ochiai's theorem on entire curves, \cite[Thm.~2]{Kawa80}}]\label{thm:1-2}%
  Let $X$ be a complex projective manifold such that its Albanese variety
  satisfies $\dim \Alb(X) > \dim X$, then entire curves on $X$ are not Zariski
  dense.  \qed
\end{thm}

Aiming to generalize these results, Campana has formulated a series of
far-reaching conjecture that generalize Lang's conjectures and relate potential
density of rational points and existence of entire curves to the notion of
``specialness'' of his theory of $\cC$-pairs.

\begin{conj}[\protect{Specialness and density of rational points, \cite[Conj.~13.21]{MR2831280}}]\label{conj:1-3}%
  Let $X$ be a projective manifold defined over a number field $k$.  Then, $X$
  is special if and only if its rational points are potentially dense.
\end{conj}

\begin{conj}[\protect{Specialness and $\cC$-entire curves, \cite[Conj.~13.17]{MR2831280}}]\label{conj:1-4}%
  Let $X$ be a complex projective manifold.  Then, $X$ is special if and only if
  $X$ admits a Zariski dense entire curve.
\end{conj}

Given that the Albanese appears prominently in Theorems~\ref{thm:1-1} and
\ref{thm:1-2}, we expect that an ``Albanese for $\cC$-pairs'' might play an
important role in future progress towards Conjectures~\ref{conj:1-3} and
\ref{conj:1-4}.  Section~\vref{sec:1-1-4} announces first results in this
direction.

\subsection{Main results}
\approvals{Erwan & yes \\ Stefan & yes}

The Albanese of a projective manifold $X$ is characterized by universal
properties that can be formulated in a number of ways, relating to the geometry
or topology of $X$.  Our presentation follows Serre's classic paper
\cite{SCC_1958-1959__4__A10_0}\footnote{See also the presentation in
\cite[Appendix~A]{MR2372739}.}, where the Albanese of a projective manifold $X$
is an Abelian variety $\Alb(X)$ together with a morphism $\alb : X → \Alb(X)$
such that any other morphism from $X$ to an Abelian variety factors via $\alb$.
More generally, we recall in Section~\ref{sec:4} that the Albanese of a
logarithmic pair $(X,D)$ is a semitoric variety $A° ⊂ A$, together with a
quasi-algebraic morphism $\alb : X ∖ D → A°$ such that any other quasi-algebraic
morphism from $X ∖ D$ to a semitoric variety factors via $\alb$.

\subsubsection{$\cC$-semitoric pairs}
\approvals{Erwan & yes \\ Stefan & yes}

For $\cC$-pairs $(X,D)$, we argue that the natural analogues of compact tori and
semitoric varieties are ``$\cC$-semitoric pairs'', that is, quotients of tori
and semitoric varieties, with their natural structure as a quotient $\cC$-pair.
Section~\ref{sec:9} introduces $\cC$-semitoric pairs and discusses their main
properties.  The following non-trivial result suggests that $\cC$-semitoric
pairs are a geometrically meaningful concept.

\begin{thm}[Precise statement in Theorem~\ref{thm:9-4}]
  Quasi-algebraic $\cC$-morphisms between $\cC$-semitoric pairs come from group
  morphisms.  \qed
\end{thm}

Following Serre, we define the Albanese of a $\cC$-pair $(X,D)$ as a universal,
quasi-algebraic $\cC$-morphism from $(X,D)$ to a $\cC$-semitoric pair.

\subsubsection{The Albanese irregularity}
\approvals{Erwan & yes \\ Stefan & yes}

Given a $\cC$-pair $(X,D)$, it turns out that the existence of an Albanese is
tied to an invariant of independent interest, the ``Albanese irregularity''
\[
  q⁺_{\Alb}(X,D) ∈ ℕ ∪ \{ ∞ \}.
\]
The Albanese irregularity is bounded from above by the augmented irregularity
$q⁺(X,D)$, introduced in \cite[Sect.~6.1]{orbiAlb1}, which measures the
dimension of the space of adapted differentials on suitable high covers.  The
Albanese irregularity differs from the augmented irregularity in that it
considers only those adapted differentials that are induced by morphisms to
semitoric varieties.  Sections~\ref{sec:5}--\ref{sec:8} define and discuss the
Albanese irregularity and the associated ``Albanese of a cover'' in great
detail.  As one of our major results, we will prove near the end of this paper
that special pairs have bounded Albanese irregularity.

\begin{thm}[Precise statement in Corollary~\ref{cor:8-2} and Remark~\ref{rem:8-4}]
  If $(X,D)$ is special in the sense of Campana, then $q⁺_{\Alb}(X,D) ≤ \dim X$.
  \qed
\end{thm}

In spite of the notion's obvious importance, we do not fully understand the
geometric meaning of the (potentially strict) inequality $q⁺_{\Alb}(X,D) ≤
q⁺(X,D)$.  Section~\ref{sec:11-2} gathers a number of open questions.

\subsubsection{The Albanese of a $\cC$-pair}
\approvals{Erwan & yes \\ Stefan & yes}

With all preparations in place, the main result of our paper is now formulated
as follows.

\begin{thm}[Precise statement in Theorem~\ref{thm:10-2} and Proposition~\ref{prop:10-5}]
  Let $(X,D)$ be a nc $\cC$-pair, where $X$ is a compact Kähler manifold.  Then,
  the following statements are equivalent.
  \begin{enumerate}
    \item An Albanese of the $\cC$-pair $(X,D)$ exists.

    \item The Albanese irregularity is finite, $q⁺_{\Alb}(X,D) < ∞$.  \qed
  \end{enumerate}
\end{thm}

We speculate that if $q⁺_{\Alb}(X,D) = ∞$, it might still make sense to define
an Albanese, either with a weaker universal property, or in the broader setup of
ind-varieties.  We refer to Section~\ref{sec:11-1} for a discussion.

\subsubsection{Preview: Pairs with high irregularity}
\approvals{Erwan & yes \\ Stefan & yes}
\label{sec:1-1-4}

In the forthcoming paper \cite{orbiAlb3}, we develop the beginnings of a
Nevanlinna theory for $\cC$-pairs, with the goal to study hyperbolicity
properties of pairs with high irregularity.  A first application generalizes the
classic Bloch-Ochiai Theorem~\ref{thm:1-2} to the setting of $\cC$-pairs: If
$q⁺_{\Alb}(X, D) > \dim X$, then every $\cC$-entire curve $(ℂ,0) → (X,D)$ is
algebraically degenerate.  This theorem explicitly includes the case where
$q⁺_{\Alb}(X, D) = ∞$.  It establishes Conjecture~\ref{conj:1-4} for some
non-special varieties.

\subsection{Acknowledgements}
\approvals{Erwan & yes \\ Stefan & yes}

We would like to thank Oliver Bräunling, Lukas Braun, Michel Brion, Frédéric
Campana, Johan Commelin, Andreas Demleitner, Constantin Podelski and Wolfgang
Soergel for long discussions.  Pedro Núñez pointed us to several mistakes in
early versions of the paper.  Jörg Winkelmann patiently answered our questions
throughout the work on this project.

The work on this paper was completed while Stefan Kebekus visited Zsolt
Patakfalvi at the EPFL and Erwan Rousseau at the
\foreignlanguage{french}{Université de Bretagne Occidentale}.  He would like to
thank Patakfalvi and Rousseau for hospitality and for many discussions.


\phantomsection\addcontentsline{toc}{part}{Preparation}

%
%
\svnid{$Id: 02-notation.tex 905 2024-09-27 08:02:49Z kebekus $}
\selectlanguage{british}

\section{Notation and standard facts}
\subversionInfo

\subsection{Global conventions}
\approvals{Erwan & yes \\ Stefan & yes}

This paper works in the category of complex analytic spaces, though all the
material in this paper will work in the complex-algebraic setting, often with
less involved definitions and proofs.  With very few exceptions, we follow the
notation of the standard reference texts \cite{CAS, DemaillyBook, MR3156076}.
An \emph{analytic variety} is a reduced, irreducible complex space.  For
clarity, we refer to holomorphic maps between analytic varieties as
\emph{morphisms} and reserve the word \emph{map} for meromorphic mappings.

We use the language of $\cC$-pairs, as surveyed in \cite{orbiAlb1}, and freely
refer to definitions and results from \cite{orbiAlb1} throughout the present
text.  The reader might wish to keep a hardcopy within reach.

\subsection{Quasi-algebraic morphisms}
\approvals{Erwan & yes \\ Stefan & yes}

Let $X$ and $Y$ be normal analytic varieties.  In contrast to the algebraic
setting, it is generally \emph{not} possible to extend a morphism between
Zariski open subsets to a meromorphic map between $X$ and $Y$: the exponential
map does not extend to a meromorphic map $ℙ¹ \dasharrow ℙ¹$.  Morphisms that do
extend meromorphically will be of special interest.  Following \cite{MR3156076},
we refer to them as \emph{quasi-algebraic}.

\begin{defn}[Quasi-algebraic morphism]\label{def:2-1}%
  Let $(X,D_X)$ and $(Y,D_Y)$ be pairs where $X$ and $Y$ are compact.  A
  morphism between the open parts, $X° → Y°$, is called \emph{quasi-algebraic
  with respect to the compactifications $X$ and $Y$} if it extends to a
  meromorphic map $X \dasharrow Y$.
\end{defn}

\begin{notation}[Quasi-algebraic morphisms to $ℂ$ and $ℂ^*$]
  Recall that $ℂ$ and $ℂ^*$ admit a unique normal compactification to $ℙ¹$.  If
  $(X,D_X)$ is a pair where $X$ is compact, it is therefore meaningful to say
  that a morphism $X° → ℂ$ or $X° → ℂ^*$ is quasi-algebraic.  Analogously, it
  makes sense to say that a function in $𝒪_X(X°)$ or in $𝒪^*_X(X°)$ is
  quasi-algebraic.
\end{notation}

\begin{defn}[Family of quasi-algebraic morphisms]\label{def:2-3}%
  In the setting of Definition~\ref{def:2-1}, let $Z$ be any normal analytic
  variety.  A \emph{family of quasi-algebraic morphisms over $Z$} is a morphism
  $X° ⨯ Z → Y°$ that extends to a meromorphic map $X⨯ Z \dasharrow Y$.
\end{defn}

For lack of an adequate reference, we include proofs of the following elementary
facts.

\begin{lem}[Elementary properties]\label{lem:2-4} %
  Let $(X,D_X)$, $(Y,D_Y)$ and $(Z,D_Z)$ be pairs, where $X$, $Y$ and $Z$ are
  compact.  Assume that a sequence of morphism is given,
  \[
    \begin{tikzcd}[column sep=2cm]
      X° \arrow[r, "α°"'] \arrow[rr, "γ°", bend left=15] & Y° \arrow[r, "β°"'] & Z°,
    \end{tikzcd}
  \]
  where $α°$ is quasi-algebraic.  Then, the following holds.
  \begin{enumerate}
    \item\label{il:2-4-1} If $β°$ is quasi-algebraic, then $γ°$ is
    quasi-algebraic.

    \item\label{il:2-4-2} If $α°$ is dominant and $γ°$ is quasi-algebraic, then
    $β°$ is quasi-algebraic.
  \end{enumerate}
\end{lem}
\begin{proof}
  Only \ref{il:2-4-2} will be shown.  Replacing $X$ and $Y$ by suitable
  bimeromorphic models, we may assume that there exists a commutative diagram as
  follows,
  \[
    \begin{tikzcd}[column sep=2cm]
      X \arrow[r, two heads, "α\text{, surjective}"'] \arrow[rr, "γ", bend left=15] & Y \arrow[r, dotted, "∃?\:β"'] & Z \\
      X° \arrow[rr, "γ°"', bend right=15] \arrow[r, "α°\text{, dominant}"] \arrow[u, hook] & Y° \arrow[r, "β°"] \arrow[u, hook] & Z°.  \arrow[u, hook]
    \end{tikzcd}
  \]
  The image $Γ ⊂ Y⨯Z$ of the product morphism $α⨯γ : X → Y⨯Z$ is analytic by the
  proper mapping theorem.  Commutativity of the diagram guarantees that $Γ$ is
  bimeromorphic to $Y$, and hence the graph of the desired meromorphic map $β :
  Y \dasharrow Z$.
\end{proof}

Quasi-algebraic morphisms to $ℂ^*$ enjoy the following strong rigidity property.

\begin{lem}[Families of quasi-algebraic morphisms to $ℂ^*$]\label{lem:2-5}%
  Let $(X,D_X)$ be a pair where $X$ is compact, let $Z$ be any normal analytic
  variety and let $φ° : X°⨯Z → ℂ^*$ be a family of quasi-algebraic morphisms
  over $Z$.  Then, there exist functions $f° ∈ 𝒪^*_X(X°)$ and $g ∈ 𝒪^*_Z(Z)$
  such that the equality $φ°(x,z) = f°(x)·g(z)$ holds for every $(x,z) ∈ X°⨯ Z$.
\end{lem}
\begin{proof}
  Extend $φ°$ to a meromorphic map $φ : X⨯Z → ℙ¹$ and view $φ$ as a meromorphic
  function.  Choosing a point $z_0 ∈ Z$, we would like to compare $φ$ to the
  meromorphic function $F(x,z) := φ(x,z_0)$.  For that, consider the associated
  principal divisors, $\operatorname{div} φ$ and $\operatorname{div} F$ in
  $\Div(X⨯ Z)$.  Both divisors are supported on $(X∖X°)⨯Z$ and are hence of
  product form.  Their restrictions to $X⨯\{z_0\}$ agree.  It follows that the
  two divisors are equal, so that $G := φ/F$ is a holomorphic function on $X ⨯
  Z$ without zeros or poles.  The function $G$ is constant on the (compact!)
  fibres of the projection map $X⨯Z → Z$ and hence descends to a function $g ∈
  𝒪^*_Z(Z)$ on the normal space $Z$.  To conclude, set $f°(x) := φ°(x,z_0)$.
\end{proof}

%
%
\svnid{$Id: 03-semitori.tex 942 2024-10-10 16:13:56Z kebekus $}
\selectlanguage{british}

\section{Semitoric varieties, quasi-algebraic morphisms and groups}
\subversionInfo
\approvals{Erwan & yes \\ Stefan & yes}
\label{sec:3}

The Albanese of a compact Kähler manifold is a compact complex torus.  We will
recall in Section~\ref{sec:4} that the Albanese of a logarithmic pair is a more
complicated object: a semitorus together with a preferred bimeromorphic
equivalence class of a compactification.  For the reader's convenience, we
recall the relevant notions and prove a number of elementary statements that are
not readily found in the literature.

We follow conventions and the language of the textbook \cite{MR3156076} and
refer the reader to \cite[Sect.~4 and 5]{MR3156076} for details, proofs and
references to the original literature.

\begin{defn}[\protect{Semitorus, presentation, \cite[Def.~5.1.5 and Sect.~5.1.5]{MR3156076}}]\label{def:3-1}%
  A \emph{semitorus} is a connected commutative complex Lie group $A°$ that
  admits a surjective Lie group morphism $π° : A° \twoheadrightarrow T$, where
  $T$ is a compact complex torus and $\ker π° ≅ (ℂ^*)^{⨯•}$.  Lie group
  morphisms of this form are called \emph{presentations} of the semitorus $A°$.
\end{defn}

\begin{rem}
  Semitori also appear under the name \emph{quasi-tori} in the literature,
  \cite[p.~119]{KobayashiGrundlehren}.  Presentations are not unique.  A given
  semitorus might allow two different presentations whose associated compact
  complex tori are hugely different.
\end{rem}

\subsection{Semitoric varieties}
\approvals{Erwan & yes \\ Stefan & yes}

Semitoric varieties are the analytic analogues of Abelian varieties, complex
tori and toric varieties.  The following definition is taken almost verbatim
from \cite{MR3156076}.

\begin{defn}[\protect{Semitoric variety, \cite[Def.~5.3.3]{MR3156076}}]\label{def:3-3}%
  A \emph{semitoric variety} is a compact analytic variety $A$ together with a
  holomorphic action $א° ↺ A$ of a semitorus $א°$ such that the following holds.
  \begin{enumerate}
    \item\label{il:3-3-1} There is a dense open orbit $A° ⊂ A$ on which $א°$
    acts freely.

    \item\label{il:3-3-2} There exists a presentation $π° : א°
    \twoheadrightarrow T$ with the following properties.
    \begin{itemize}
      \item Using \ref{il:3-3-1} to identify $A°$ with $א°$, the morphism $π°$
      extends to an $א°$-equivariant morphism $π : A \twoheadrightarrow T$.

      \item For every point $t ∈ T$ the fibre $A_t = π^{-1}(t)$ is isomorphic to
      a smooth toric variety.  In other words, $A_t$ admits the structure of a
      smooth algebraic variety such that the action of $\ker π°$ on $A_t$ is
      algebraic.
    \end{itemize}
  \end{enumerate}
\end{defn}

\begin{explanation}[First item of \ref{il:3-3-2}]
  Extendability of $π°$ in the first item of \ref{il:3-3-2} is independent of
  the identification.  The extension is unique if it exists.
\end{explanation}

\begin{explanation}[Second item of \ref{il:3-3-2}]
  Spelled out in detail, the second item of \ref{il:3-3-2} requires that $A_t$
  admits the structure of a smooth algebraic variety such that the action of
  $\ker π°$ on $A_t$ is algebraic.  Item~\ref{il:3-3-2} immediately implies that
  $A$ is smooth and $A ∖ A°$ is an snc divisor.
\end{explanation}

\begin{warning}
  An identification of $A°$ with a Lie group is not part of the data that
  defines a semitoric variety.
\end{warning}

\begin{notation}[Semitoric compactifications]%
  For brevity of notation, we will frequently write ``Let $A°$ be a semitorus,
  and let $A° ⊂ A$ be a semitoric compactification…'' to say that $A$ is a
  compactification where the action on $A°$ on itself by left multiplication
  extends to $A$ in a way that makes $A$ with the action $A° ↺ A$ a semitoric
  variety.
\end{notation}

\begin{notation}[Semitoric varieties]%
  For brevity of notation, we will frequently write ``Let $A° ⊂ A$ be a
  semitoric variety…'' to say that we consider a compact analytic variety $A$
  and a dense open subset $A° ⊂ A$ where $A°$ is biholomorphic to a semitorus
  $א°$ that acts holomorphically on $A$ such that
  \begin{itemize}
    \item the subset $A° ⊂ A$ is an orbit on which $א°$ acts freely, and

    \item the analytic variety $A$ together with the action of $א°$ is a
    semitoric variety in the sense of Definition~\ref{def:3-3}.
 \end{itemize}
\end{notation}

\begin{notation}[Semitoric varieties as logarithmic pairs]
  Given a semitoric variety $A° ⊂ A$, we will often consider the associated
  logarithmic pair $(A, Δ)$ and write $Ω^p_A (\log Δ)$, with the implicit
  understanding that $Δ := A ∖ A°$ is the difference divisor.  If there is more
  than one semitoric variety involved in the discussion, we write $(A, Δ_A)$ for
  clarity.
\end{notation}

\begin{notation}[Quasi-algebraic morphisms]
  Given two semitoric varieties, $A° ⊂ A$ and $B° ⊂ B$, we follow
  Definition~\ref{def:2-1} and say that a morphism $A° → B°$ is quasi-algebraic
  if it extends to a meromorphic map $A \dasharrow B$.  Along similar lines, if
  $(X,D)$ is any pair where $X$ is compact, it makes sense to say that morphisms
  between the open parts, $A° → X°$ and $X° → A°$, are quasi-algebraic.
\end{notation}

\subsection{Elementary properties}
\approvals{Erwan & yes \\ Stefan & yes}
\label{sec:3-2}

For later reference, we state several facts about quasi-algebraic morphisms
between semitoric varieties.  The proofs are tedious but mostly elementary, and
left to the reader.  In fancy words, Facts~\ref{fact:3-11}--\ref{fact:3-13} can
be seen to give an equivalence of categories between presentations of semitori
and bimeromorphic equivalence classes of semitoric compactifications.

\begin{fact}[Uniqueness of presentation]\label{fact:3-11}%
  The presentation of $\:א°$ in Definition~\ref{def:3-3} is unique.  More
  precisely, there exists a unique presentation $π° : א° = A° \twoheadrightarrow
  T$ that extends to an $א°$-equivariant fibre bundle $π : A \twoheadrightarrow
  T$.  \qed
\end{fact}

\begin{fact}[\protect{Existence for given presentation, \cite[Thm.~5.1.35]{MR3156076}}]\label{fact:3-12}%
  Let $A°$ be a semitorus and let $π° : A° \twoheadrightarrow T$ be a
  presentation.  If $F$ is any smooth toric variety compactifying $F° :=
  (π°)^{-1}(0_T)$, then there exists a semitoric variety $A° ⊂ A$ with
  associated morphism $π : A → T$ where $π^{-1} (0_T)$ is isomorphic to $F$ as
  an $F°$-space.  \qed
\end{fact}

A group morphism between the open parts of semitoric varieties is
quasi-algebraic if and only if it respects the associated presentations.  In
particular, we find that the bimeromorphic equivalence class of a semitoric
compactification is uniquely determined by the presentation.

\begin{fact}[Quasi-algebraic group morphisms and presentations]\label{fact:3-13}%
  Let $A°$ and $B°$ be semitori, and let $A° ⊂ A$ and $B° ⊂ B$ be semitoric
  compactifications, with associated morphisms $π_A : A \twoheadrightarrow T_A$
  and $π_B: B \twoheadrightarrow T_B$.  If $σ° : A° → B°$ is any holomorphic
  group morphism, then the following two statements are equivalent.
  \begin{enumerate}
    \item\label{il:3-13-1} The morphism $σ°$ is quasi-algebraic.

    \item\label{il:3-13-2} There exists a holomorphic group morphism $τ: T_A →
    T_B$, where $τ◦π_A = π_B◦σ°$.  \qed
  \end{enumerate}
\end{fact}
\begin{proof}[Proof of the implication \ref{il:3-13-1} $⇒$ \ref{il:3-13-2}]
  Lemma~\ref{lem:2-4} guarantees that the composed map $π°_B◦σ°$ is
  quasi-algebraic.  Since the compact complex torus $B$ does not contain
  rational curves, we find that the quasi-algebraic morphism $π°_B◦σ°$ factors
  via the $(ℂ^*)^{⨯•}$-fibre bundle $π°_A$.  We obtain a morphism $f_T : T_A →
  T_B$ and a commutative diagram as follows,
  \begin{equation}\label{eq:3-13-3}
    \begin{tikzcd}[column sep=2cm, row sep=large]
      A° \ar[d, two heads, "π°_A"'] \ar[r, "σ°"] & B° \arrow[d, two heads, "π°_B"] \\
      T_A \ar[r, "τ"'] & T_B.
    \end{tikzcd}
  \end{equation}
  The morphism $τ$ maps $0_{T_A}$ to $0_{T_B}$ and is hence a group morphism,
  \cite[Def.~5.1.36]{MR3156076}.
\end{proof}

If the equivalent conditions of Fact~\ref{fact:3-13} hold, there is a little
more that we can say: $σ°$ extends to a morphism between $A$ and $B$ if and only
its restriction to the central fibre extends to a morphism.

\begin{fact}[Morphisms and bimeromorphic maps]\label{fact:3-14}%
  In the setting of Fact~\ref{fact:3-13}, assume that $σ°$ is quasi-algebraic,
  with associated meromorphic map $σ : A \dasharrow B$.  Then, the following two
  statements are equivalent.
  \begin{enumerate}
    \item The meromorphic map $σ$ is a morphism.

    \item The meromorphic map $σ|_{π_A^{-1}(0_{T_A})} : π_A^{-1}(0_{T_A})
    \dasharrow π_B^{-1}(0_{T_B})$ is a morphism.  \qed
  \end{enumerate}
\end{fact}

On semitoric varieties, a differential form is logarithmic if and only if it is
invariant.

\begin{prop}[Invariant differentials and logarithmic differentials]\label{prop:3-15} %
  In the setting of Definition~\ref{def:3-3}, the following statements hold for
  every number $p ∈ ℕ$.
  \begin{enumerate}
    \item\label{il:3-15-1} The locally free sheaf $Ω^p_A(\log Δ)$ is free.

    \item\label{il:3-15-2} Every $A°$-invariant differential form $τ° ∈ H⁰\bigl(
    A°,\, Ω^p_{A°}\bigr)$ extends to a logarithmic form $τ ∈ H⁰\bigl(A,\,
    Ω^p_A(\log Δ) \bigr)$.

    \item\label{il:3-15-3} Every logarithmic form in $H⁰\bigl(A,\, Ω^p_A(\log Δ)
    \bigr)$ is $A°$-invariant.
  \end{enumerate}
\end{prop}
\begin{proof}
  Item~\ref{il:3-15-1} is \cite[Cor.~5.4.5]{MR3156076}.  For
  Item~\ref{il:3-15-2}, observe that every $A°$-invariant differential form $τ°
  ∈ H⁰\bigl( A°,\, Ω^p_{A°}\bigr)$ can be written as a sum of wedge products of
  1-differentials.  To prove Item~\ref{il:3-15-2}, it will therefore suffice to
  consider the case $p=1$.  The group $A°$ acts on itself by left
  multiplication.  By assumption, this actions extends to an action of $A°$ on
  $A$ that stabilizes $Δ$.  The $A°$-invariant vector fields on $A°$ that are
  induced by this action will therefore extend to sections of $𝒯_A(-\log Δ)$.
  Using \ref{il:3-15-1}, the case $p=1$ of Item~\ref{il:3-15-2} now follows by
  taking duals.

  Item~\ref{il:3-15-3} follows from \ref{il:3-15-1} and \ref{il:3-15-2}, given
  that the dimensions of the spaces $H⁰\bigl( A°,\, Ω^p_{A°} \bigr)^{A°}$ and
  $H⁰\bigl(A, Ω^p_A(\log Δ) \bigr)$ agree.
\end{proof}

\begin{rem}[Pull-back of logarithmic differentials I]\label{rem:3-16}%
  Given a semitoric variety $A° ⊂ A$ and a nc log pair $(X,D)$, we are often
  interested in quasi-algebraic morphisms $a° : X° → A°$.  Given that $X$ and
  $A$ are smooth and that $a$ is holomorphic away from a small subset of $X$,
  there exists a pull-back morphism for logarithmic differentials
  \[
    \diff a : H⁰\bigl( A, Ω¹_A(\log Δ) \bigr) → H⁰\bigl( X,\,
    Ω¹_X(\log D)\bigr)
  \]
  that restricts on $X°$ to the standard pull-back $\diff a°$.
\end{rem}

\begin{rem}[Pull-back of logarithmic differentials II]\label{rem:3-17}%
  Generalizing Remark~\ref{rem:3-16}, given a semitoric variety $A° ⊂ A$, a log
  pair $(X,D)$ that is not necessarily nc, and a quasi-algebraic morphism $a° :
  X° → A°$, there exists a pull-back morphism for logarithmic differentials
  \[
    \diff a : H⁰\bigl( A, Ω¹_A(\log Δ) \bigr) → H⁰\bigl( X,\,
    Ω^{[1]}_X(\log D)\bigr)
  \]
  that restricts on $X°_{\reg}$ to the standard pull-back $\diff a°$.
\end{rem}

\subsection{Quasi-algebraic morphisms}
\approvals{Erwan & yes \\ Stefan & yes}

In contrast to the algebraic setting, a morphism between semitori need not be a
group morphism, even if it respects the neutral elements of the group structure.
For an example, consider the morphism $ℂ^* → ℂ^*$, $t ↦ \exp(t-1)$.  The
situation improves for quasi-algebraic morphisms of semitoric varieties.

\begin{prop}[Quasi-algebraic morphisms and group morphisms]\label{prop:3-18} %
  Let $A°$ and $B°$ be semitori, and let $A° ⊂ A$ and $B° ⊂ B$ be semitoric
  compactifications.  Let $f° : A° → B°$ be any quasi-algebraic morphism of
  analytic varieties.  If $f°(0_{A°}) = 0_{B°}$, then $f°$ is a morphism of
  complex Lie groups.
\end{prop}
\begin{proof}
  In order to prepare for the proof, consider the associated presentations $π°_A
  : A° \twoheadrightarrow T_A$ and $π°_B : B° \twoheadrightarrow T_B$.
  Fact~\ref{fact:3-13} equips us with a group morphism $f_T : T_A → T_B$ forming
  a commutative diagram as follows,
  \begin{equation}\label{eq:3-18-1}
    \begin{tikzcd}[column sep=2cm, row sep=large]
      A° \ar[d, two heads, "π°_A"'] \ar[r, "f°"] & B° \arrow[d, two heads, "π°_B"] \\
      T_A \ar[r, "f_T"'] & T_B.
    \end{tikzcd}
  \end{equation}

  We would like to show that $f°$ is a group morphism.  For this, consider the
  auxiliary morphism
  \[
    ξ° : A°⨯A° → B°, \quad (x,y) ↦ f°(x)+f°(y)-f°(x+y).
  \]
  To conclude, we need to show that $ξ° \equiv 0_{B°}$ or equivalently that $ξ°$
  is constant.  The assumption that $f°$ is quasi-algebraic and
  \cite[Prop.~5.3.5]{MR3156076} together guarantee that $ξ°$ extends to a
  meromorphic map $ξ : A⨯A \dasharrow B$ and is hence quasi-algebraic.  The
  following property follows from the assumption that $f°(0_{A°}) = 0_{B°}$.
  \begin{equation}\label{eq:3-18-2}
    \forall a ∈ A° : ξ°(a, 0_{A°}) = ξ°(0_{A°},a) = 0_{B°}
  \end{equation}
  There is more that we can say.  If $(x,y) ∈ A°⨯A°$ is any pair of points, then
  \begin{align*}
    (π°_B◦ξ°)(x,y) & = π°_B(f°(x) + f°(y) - f°(x+y)\bigr) & & \text{definition} \\
                   & = (π°_B◦f°)(x) + (π°_B◦f°)(y) - (π°_B◦f°)(x+y) & & π°_B \text{ a grp.~morph.} \\
                   & = (f_T◦π°_A)(x) + (f_T◦π°_A)(y) - (f_T◦π°_A)(x+y) & & \text{Diagram~\eqref{eq:3-18-1}} \\
                   & = 0_{T_B} & & f_T◦π°_A \text{ a grp.~morph.}
  \end{align*}

  In summary, we find that $ξ°$ takes its image in $(π°_B)^{-1}(0_{T_B})$.
  Fixing one identification $(π°_B)^{-1}(0_{T_B}) ≅ (ℂ^*)^{⨯ •}$,
  Lemma~\ref{lem:2-5} allows writing $ξ°$ in product form.  More precisely,
  there exist functions $a_•, b_• ∈ 𝒪^*_A(A°)$ such that
  \[
    ξ°(x,y) = \bigl(a_1(x)·b_1(y), …, a_n(x)·b_n(y)\bigr), \quad \text{for every } (x,y) ∈ A°⨯ A°.
  \]
  Equation~\eqref{eq:3-18-2} will then imply that $ξ°$ is constant.
\end{proof}

\begin{cor}[Quasi-algebraic morphisms between open parts of semitoric varieties]\label{cor:3-19} %
  Let $A° ⊂ A$ and $B° ⊂ B$ be semitoric varieties and let $f° : A° → B°$ be a
  quasi-algebraic morphism of analytic varieties.  Then, the following holds.
  \begin{enumerate}
    \item The fibres of $f°$ are of pure dimension.

    \item Any two non-empty fibres of $f°$ are of the same dimension.

    \item If $f°$ is quasi-finite, then it is finite.

    \item If $f°$ is finite and surjective, then it is étale.  \qed
  \end{enumerate}
\end{cor}

\begin{cor}[Quasi-algebraic automorphisms]\label{cor:3-20}%
  Let $A° ⊂ A$ be a semitoric variety.  Once we choose an element $0_{A°} ∈ A°$
  to equip $A°$ with the structure of a holomorphic Lie group, the group of
  quasi-algebraic automorphisms of the analytic variety $A°$ decomposes as a
  semidirect product $(\text{translations}) \rtimes (\text{group morphisms})$.
  \qed
\end{cor}

\begin{cor}[Semitoric compactification with additional symmetry]\label{cor:3-21}
  Let $A° ⊂ A_1$ be a semitoric variety, and let $G ⊂ \Aut(A°)$ be a finite
  group of quasi-algebraic automorphisms.  Then, there exists a semitoric
  variety $A° ⊂ A_2$, such that the following holds.
  \begin{itemize}
    \item The analytic varieties $A_1$ and $A_2$ are bimeromorphic.

    \item The $G$-action on $A°$ extends equivariantly to $A_2$.
  \end{itemize}
\end{cor}
\begin{proof}
  As before, write $π° : A° \twoheadrightarrow T$ for the unique presentation
  that extends to an $A°$-equivariant morphism $π : A_1 \twoheadrightarrow T$.
  Corollary~\ref{cor:3-20} allows assuming without loss of generality that $G$
  is a finite group of quasi-algebraic group morphisms.  Fact~\ref{fact:3-13}
  will then guarantee that $G$ acts by group morphisms on $T$ in a way that
  makes the morphism $π°$ equivariant.  In particular, the $G$-action fixes the
  point $0_T$ and stabilizes the fibre $F° := (π°)^{-1}(0_T) ≅ (ℂ^*)^{⨯ •}$.
  Toric geometry will then allow choosing\footnote{Since $G$ is finite, every
  fan can be refined to become stable under the action of $G$ on $N_ℝ$.} a
  $G$-equivariant toric compactification $F° ⊂ F$, and Fact~\ref{fact:3-12}
  presents us with a semitoric compactification $A° ⊂ A_2$, fibred over $T$ with
  typical fibre $F$.  Fact~\ref{fact:3-13} ensures that $A_1$ and $A_2$ are
  bimeromorphic, and Fact~\ref{fact:3-14} asserts that the $G$-action on $A°$
  extends equivariantly to $A_2$.
\end{proof}

\subsection{Quasi-algebraic subgroups}
\approvals{Erwan & yes \\ Stefan & yes}

In analogy to the notion of a quasi-algebraic morphism, a quasi-algebraic
subgroup of a semitorus is a subgroup that extends to an analytic set in a
preferred compactification.  A full discussion of this notion is found in
\cite[Sect.~5.3.4]{MR3156076}.

\begin{defn}[\protect{Quasi-algebraic subgroup, \cite[Def.~5.3.14]{MR3156076}}]\label{def:3-22}%
  Let $A°$ be a semitorus, and let $A° ⊂ A$ be a semitoric compactification.  An
  analytic subgroup $H° ⊂ A°$ is called \emph{quasi-algebraic for the semitoric
  compactification $A° ⊂ A$} if the topological closure of $H°$ in $A$ is an
  analytic subset.
\end{defn}

\subsubsection{Elementary properties}
\approvals{Erwan & yes \\ Stefan & yes}

We state two facts about quasi-algebraic subgroups for later reference.  The
elementary proofs are left to the reader.

\begin{fact}[\protect{Quasi-algebraic subgroups are semitori, \cite[Prop.~5.3.13]{MR3156076}}]\label{fact:3-23}%
  In the setting of Definition~\ref{def:3-22}, quasi-algebraic subgroups are
  again semitori.  \qed
\end{fact}

\begin{warning}[Analytic subgroups need not be semitori]
  Despite claims to the contrary in the literature,
  cf.~\cite[Lem.~3.8.18]{KobayashiGrundlehren}, closed analytic subgroups of
  semitori need not be semitori in general.  See \cite[Ex.~5.1.44]{MR3156076}
  and the references there for an example.
\end{warning}

The following fact implies that the notion of ``quasi-algebraic subgroup''
depends only on the bimeromorphic equivalence class of a semitoric
compactification.

\begin{fact}[Dependence on choice of compactification]\label{fact:3-25}%
  Let $A°$ be a semitorus, and let $A° ⊂ A_1$ and $A° ⊂ A_2$ be two
  bimeromorphic semitoric compactifications.  Then, a subgroup $H° ⊂ A°$ is
  quasi-algebraic for the semitoric compactification $A° ⊂ A_1$ if and only if
  it is quasi-algebraic for the semitoric compactification $A° ⊂ A_2$.  \qed
\end{fact}

\subsubsection{Lattice structure}
\approvals{Erwan & yes \\ Stefan & yes}

As usual in algebra, quasi-algebraic subgroups form a complete lattice.  We
refrain from going into any details here and state the only fact that will be
relevant for us later.

\begin{fact}[Existence of a smallest group]\label{fact:3-26}%
  In the setting of Definition~\ref{def:3-22}, the intersection of arbitrarily
  many quasi-algebraic subgroups is quasi-algebraic.  In particular, given any
  subset $I ⊂ A°$, there exists a unique smallest quasi-algebraic subgroup that
  contains $I$.  \qed
\end{fact}

\subsubsection{Quotients}
\approvals{Erwan & yes \\ Stefan & yes}

Semitoric varieties are stable under quotients by quasi-algebraic groups, in the
following sense.

\begin{fact}[\protect{Existence of a quotients, \cite[Thm.~5.3.13]{MR3156076}}]\label{fact:3-27}%
  Let $A°$ be a semitorus, and let $A° ⊂ A$ be a semitoric compactification.  If
  $H° ⊂ A°$ is a quasi-algebraic subgroup, then the quotient $Q° := A°/H°$ is a
  semitorus and there exists a semitoric compactification $Q° ⊂ Q$ that renders
  the quotient morphism $q° : A° \twoheadrightarrow Q°$ quasi-algebraic.
  \qed
\end{fact}

\subsubsection{Examples}
\approvals{Erwan & yes \\ Stefan & yes}

Throughout this paper, quasi-algebraic subgroups appear as kernels of
quasi-algebraic group morphisms and as fixed point sets of quasi-algebraic group
action.  We recall the relevant facts.

\begin{fact}[Kernels of quasi-algebraic group morphisms]\label{fact:3-28}%
  Let $A°$ and $B°$ be a semitori, and let $A° ⊂ A$ and $B° ⊂ B$ be semitoric
  compactifications.  If $α° : A° → B°$ is any quasi-algebraic morphism of
  complex Lie groups, then $\ker(α°) ⊂ A°$ is quasi-algebraic for the semitoric
  compactification $A° ⊂ A$.  \qed
\end{fact}

\begin{prop}[Fixed points of quasi-algebraic groups actions]\label{prop:3-29}%
  Let $A°$ be a semitorus, and let $A° ⊂ A$ be a semitoric compactification.
  Let $G ⊂ \Aut(A°)$ be a finite group that acts on $A°$ by quasi-algebraic
  automorphisms.  If
  \[
    X ⊂ \{ \vec{a} ∈ A° \::\: \text{isotropy $G_{\vec{a}}$ is not trivial}\}
  \]
  is any irreducible complex subspace, then $X$ is contained in the translate of
  a proper quasi-algebraic subgroup of $A°$.
\end{prop}
\begin{proof}
  We assume that the group $G$ is non-trivial, or else there is nothing to
  prove.  Since $G$ is finite, there will be an element $g ∈ G ∖ \{e\}$ that
  fixes $X$ pointwise.  Shrinking $G$ and enlarging $X$, we may therefore assume
  without loss of generality that $G$ is cyclic, $G = \langle g \rangle$, and
  that $X$ is a component of $\Fix(G)$.

  Recall from Proposition~\ref{prop:3-18} that the action of $g$ on $A°$ is of
  the form
  \[
    g : A° → A°, \quad \vec{a} ↦ φ°(\vec{a}) - \vec{a}_0
  \]
  where $φ° : A° → A°$ is a quasi-algebraic group morphism and $\vec{a}_0 ∈ A°$
  is a constant.  It follows that $\vec{a} ∈ \Fix(g)$ if and only if $\bigl(φ° -
  \Id_{A°}\bigr)(\vec{a}) = \vec{a}_0$.  If $\vec{x} ∈ X$ is any element, this
  implies that
  \[
    \Fix(G) = \ker(φ°-\Id_{A°}) + \vec{x}.
  \]
  But by Fact~\ref{fact:3-28}, the components of $\ker(φ°-\Id_{A°})$ are
  translates of quasi-algebraic subgroups.
\end{proof}

%
%
\svnid{$Id: 04-Albanese.tex 915 2024-09-30 08:51:09Z kebekus $}
\selectlanguage{british}

\section{The Albanese of a logarithmic pair}
\label{sec:4}
\subversionInfo
\approvals{Erwan & yes\\ Stefan & yes}

To prepare for the slightly involved constructions later in this paper, we
recall a number of facts about the Albanese for logarithmic pairs.  For lack of
references, we include a full discussion the Albanese construction in the
singular Kähler case.  We refer the reader to \cite{SCC_1958-1959__4__A10_0} and
\cite[Appendix~A]{MR2372739} for general results in the algebraic setting, and
to \cite[Sect.~4.5]{MR3156076} for a construction of the Albanese for
logarithmic pairs $(X,D)$ in case where $X$ is a compact Kähler manifold and $D$
a reduced divisor that does not necessarily have snc support.

\begin{setting}\label{set:4-1}%
  Let $(X,D)$ be a log pair where $X$ is compact.  In line with
  \cite[Notation~2.14]{orbiAlb1}, denote the open part by $X° := X∖D$.
\end{setting}

\begin{defn}[The Albanese for compact log pairs]\label{def:4-2}%
  Assume Setting~\ref{set:4-1}.  An Albanese of $(X,D)$ is a semitoric variety
  $\Alb(X,D)° ⊂ \Alb(X,D)$ together with a quasi-algebraic morphism
  \[
    \alb(X,D)° : X° → \Alb(X,D)°
  \]
  that satisfies the following universal property.  If $A° ⊂ A$ is any semitoric
  variety and
  \[
    a° : X° → A°
  \]
  is any quasi-algebraic morphism, then there exists a unique morphism $b°$ that
  makes the following diagram commute,
  \begin{equation}\label{eq:4-2-1}
    \begin{tikzcd}[column sep=2cm]
      X° \arrow[r, "\alb(X{,}D)°"'] \arrow[rr, "a°", bend left=15] & \Alb(X,D)°
      \arrow[r, "∃!b°"'] & A°.
    \end{tikzcd}
  \end{equation}
  The morphism $b°$ is quasi-algebraic.
\end{defn}

\begin{rem}[Quasi-Albanese]
  The Albanese of an snc logarithmic pair also appears under the name
  ``quasi-Albanese'' in the literature, cf.~\cite{Fujino15}.
\end{rem}

\begin{notation}[Empty boundary]
  If the divisor $D$ in Definition~\ref{def:4-2} is zero, we will often drop it
  from the notation and write $\alb(X)° : X° → \Alb(X)° = \Alb(X)$ for brevity.
\end{notation}

\begin{rem}[Compactification and presentation of $\Alb(X,D)°$]\label{rem:4-5}%
  In the setting of Definition~\ref{def:4-2}, recall from Facts~\ref{fact:3-11}
  and \ref{fact:3-13} that the semitoric compactification $\Alb(X,D)° ⊂
  \Alb(X,D)$ defines a unique presentation of the semitorus $\Alb(X,D)°$.  If
  $(X,D)$ is snc, the construction presented in Section~\ref{sec:4-2} will show
  that this presentation equals the natural morphism $\Alb(X,D)°
  \twoheadrightarrow \Alb(X)$ induced by the universal property.
\end{rem}

\begin{explanation}
  The reader coming from algebraic geometry might wonder why
  Definition~\ref{def:4-2} is so complicated.  The reason is this: if $V°$ is a
  smooth, quasi-projective variety and if $V° ⊂ V_1$ and $V° ⊂ V_2$ are two
  projective compactifications, then $V_1$ and $V_2$ are birational and there
  exists a third compactification that dominates both.

  This is no longer true in complex geometry, where two compactifications need
  not necessarily be bimeromorphic, and where the bimeromorphic equivalence
  class of a particular compactification is often part of the data.  Along these
  lines, the Albanese is not just the semitorus $\Alb(X,D)°$, but the semitorus
  together with a bimeromorphic equivalence class of a compactification
  $\Alb(X,D)$.  The word ``quasi-algebraic'' that appears all over
  Definition~\ref{def:4-2} ensures that all morphisms respect the classes of the
  compactifications.
\end{explanation}

\subsection{Uniqueness}
\approvals{Erwan & yes \\ Stefan & yes}

The universal property of the Albanese guarantees that $\Alb(X,D)°$ and
$\alb(X,D)°$ are unique up to unique isomorphism.  The left-invariant
compactification $\Alb(X,D)$ is bimeromorphically unique.  Following the
classics, we abuse notation and refer to any Albanese as ``the Albanese'', with
associated semitoric \emph{Albanese variety} $\Alb(X,D)° ⊂ \Alb(X,D)$ and
\emph{Albanese morphism} $\alb(X,D)°$.  Once we fix a point $0_{\Alb(X,D)°} ∈
\Alb(X,D)°$ to equip $\Alb(X,D)°$ with the structure of a Lie group,
Fact~\vref{fact:3-25} allows talking about subgroups of $\Alb(X,D)°$ that are
quasi-algebraic for the semitoric compactification $\Alb(X,D)° ⊂ \Alb(X,D)$.

\subsection{Existence}
\label{sec:4-2}
\approvals{Erwan & yes \\ Stefan & yes}

The existence of an Albanese is well-known for snc pairs, but hardly discussed
in the literature for arbitrary Kähler pairs.  We briefly recall the arguments
in the snc setting, use resolutions of singularities to construct a candidate
for the Albanese in general and prove that this candidate satisfies the
properties spelled out in Definition~\ref{def:4-2} above.

\begin{prop}[Existence of the Albanese of a Kähler log pair]\label{prop:4-7}%
  In Setting~\ref{set:4-1}, assume that $X$ is Kähler.  Then, an Albanese of
  $(X,D)$ exists.
\end{prop}

We begin the proof by recalling the classic construction for snc pairs.  For
singular pairs, Construction~\ref{cons:4-8} will show how to build an Albanese
using a resolution of singularities.  We conclude the proof of
Proposition~\ref{prop:4-7} {} \vpageref*{pf:4-7}, showing that
Construction~\ref{cons:4-8} does indeed satisfy the necessary universal
property.

\begin{proof}[Proof of Proposition~\ref{prop:4-7} is $(X,D)$ is snc]
  If the pair $(X,D)$ of Setting~\ref{set:4-1} is snc, choose a point $x ∈ X°$
  and consider the group morphism
  \[
    i : π_1(X°,x) → H⁰\bigl( X,\, Ω¹_X(\log D) \bigr)^*
  \]
  obtained by path integration.  Set
  \[
    \Alb(X,D)° := \factor{H⁰\bigl( X,\, Ω¹_X(\log D) \bigr)^*}{\img(i)}
  \]
  and define $\alb(X,D)°$ by path integration.  Hodge theory guarantees that
  $\Alb(X,D)°$ is a semitorus.  It admits a presentation as a principal
  $(ℂ^*)^{⨯•}$-bundle over $\Alb(X)$, and hence by Fact~\vref{fact:3-12} an
  equivariant compactification $\Alb(X,D)$ as a $(ℙ¹)^{⨯•}$-bundle over
  $\Alb(X)$.  A local computation shows that $\alb(X,D)°$ is quasi-algebraic for
  this compactification.  More precisely, it extends to a meromorphic map $X
  \dasharrow \Alb(X,D)$ that is holomorphic on the big open subset $X ∖ (\supp
  D)_{\sing}$.  We refer the reader to \cite[Sect.~4]{MR3156076} for details and
  proofs.
\end{proof}

\begin{construction}[Construction of the Albanese of a log pair]\label{cons:4-8}%
  Assume the setting of Proposition~\ref{prop:4-7}.  For the reader's
  convenience, we subdivided the construction into relatively independent steps.

  \subsubsection*{Step~1 in Construction~\ref*{cons:4-8}, Resolution of singularities}
  
  Choose a log-resolution $π : \wtilde{X} \twoheadrightarrow X$, consider the
  reduced snc divisor $\wtilde{D} := \supp π^{-1}(D)$ on $\wtilde{X}$ and write
  $\wtilde{X}° := \wtilde{X} ∖ \wtilde{D}$.  The proof of
  Proposition~\ref{prop:4-7} in the snc case provides us with an Albanese of
  $(\wtilde{X}, \wtilde{D})$ that we briefly denote as
  \begin{equation}\label{eq:4-8-1}
    \begin{tikzcd}[row sep=1cm]
      \wtilde{X} \ar[d, two heads, "π\text{, log resolution}"'] \ar[r, phantom, "⊇"] & \wtilde{X}° \ar[d, two heads, "π°"'] \ar[rrr, "\wtilde{a}° \::=\: \alb(\wtilde{X}{,} \wtilde{D})°", "\text{quasi-algebraic}"'] &&& \underbrace{\Alb(\wtilde{X}, \wtilde{D})°}_{=: \wtilde{A}°} \ar[r, phantom, "⊆"] & \underbrace{\Alb(\wtilde{X}, \wtilde{D})}_{=: \wtilde{A}} \\
      X \ar[r, phantom, "⊇"] & X°.
    \end{tikzcd}
  \end{equation}

  \subsubsection*{Step~2 in Construction~\ref*{cons:4-8}, Quotients by subgroups of $\wtilde{A}°$}

  Choose an element $0_{\wtilde{A}°} ∈ \wtilde{A}°$ in order to equip
  $\wtilde{A}°$ with the structure of a Lie group.  If $H° ⊆ \wtilde{A}°$ is any
  quasi-algebraic subgroup, recall from Fact~\ref{fact:3-27} that the quotient
  \[
    A°_{H°} := \factor{\wtilde{A}°}{H°}
  \]
  is a semitorus and there exists a semitoric compactification $A°_{H°} ⊂
  A_{H°}$ that renders the quotient morphism $q°_{H°} : \wtilde{A}°
  \twoheadrightarrow A°_{H°}$ quasi-algebraic.  If the composed map
  \[
    q°_{H°} ◦ \wtilde{a}° : \wtilde{X}° → A_H°
  \]
  is constant on $π°$-fibers, then it factors via $π°$, and we obtain an
  extension of Diagram~\eqref{eq:4-8-1} as follows,
  \begin{equation}\label{eq:4-8-2}
    \begin{tikzcd}[row sep=1cm]
      \wtilde{X} \ar[d, two heads, "π\text{, log resolution}"'] \ar[r, phantom, "⊇"] & \wtilde{X}° \ar[d, two heads, "π°"'] \ar[rrrr, "\wtilde{a}°\text{, quasi-algebraic}"] &&&& \wtilde{A}° \ar[d, two heads, "q°_{H°}"] \ar[r, phantom, "⊆"] & \wtilde{A} \ar[d, dashed, two heads, "q_{H°}"] \\
      X \ar[r, phantom, "⊇"] & X° \ar[rrrr, "a°_{H°}"'] &&&& A°_{H°} \ar[r, phantom, "⊆"] & A_{H°}
    \end{tikzcd}
  \end{equation}
  The quotient carries a natural structure of a Lie group that makes $q°_{H°}$ a
  group morphism.  Lemma~\ref{lem:2-4} guarantees that $a°_{H°}$ is again
  quasi-algebraic.

  \subsubsection*{Step~3 in Construction~\ref*{cons:4-8}, Identifying a suitable subgroup of $A°$}

  Aiming to construct an Albanese for $(X,D)$ using the construction of Step~2,
  we need to find a quasi-algebraic subgroup $H° ⊆ A°$ to which Step~2 can be
  applied.  To this end, consider the set of all subgroups that satisfy the
  assumptions of Step~2,
  \[
    \mathcal H° := \{ B° ⊆ \wtilde{A}° \text{ quasi-algebraic} \::\: q°_{B°} ◦ \wtilde{a}° \text{ is constant on $π°$-fibers}\}.
  \]
  Take $H°$ as the infimum of $\mathcal H°$ in the complete lattice of all
  quasi-algebraic subgroups $\wtilde{A}°$.  In other words, define
  \[
    H° := \bigcap_{B° ∈ \mathcal H°} B°
  \]
  and recall from Fact~\ref{fact:3-26} that $H°$ is indeed a quasi-algebraic
  subgroup.  With this choice, observe that $q°_{H°} ◦ \wtilde{a}°$ is again
  constant on $π°$-fibers, so that $H°$ is in fact the minimal element of
  $\mathcal H°$.  Step~2 equips us with a semitoric compactification $A°_{H°} ⊂
  A_{H°}$ and a diagram of Form~\eqref{eq:4-8-2}.  Write $\Alb(X,D)° ⊂
  \Alb(X,D)$ for $A°_{H°} ⊆ A_{H°}$ and denote the quasi-algebraic morphism
  $a°_{H°}$ by
  \[
    \alb(X,D)° : X° → \Alb(X,D)°.
  \]
  Construction~\ref{cons:4-8} ends here.
\end{construction}

\begin{proof}[Proof of Proposition~\ref{prop:4-7}]\label{pf:4-7}
  \CounterStep{}It remains to show that the varieties and morphism of
  Construction~\ref{cons:4-8} satisfy the conditions spelled out in
  Definition~\ref{def:4-2} above.  To this end, assume that $A° ⊂ A$ is a
  semitoric variety and $a° : X° → A°$ is a quasi-algebraic morphism.
  Item~\ref{il:2-4-1} of Lemma~\ref{lem:2-4} guarantees that $a°◦π° :
  \wtilde{X}° → A°$ is quasi-algebraic.  The universal property of the Albanese
  $\Alb(\wtilde{X}, \wtilde{D})°$ of the snc pair $(X,D)$ thus gives us a unique
  quasi-algebraic morphism $\wtilde{b}°$ of Lie groups that makes the following
  diagram commute,
  \[
    \begin{tikzcd}[row sep=1cm, column sep=1.8cm]
      \wtilde{X}° \ar[rr, "\alb(\wtilde{X}{,} \wtilde{D})°"] \ar[d, two heads, "π°"'] && \Alb(\wtilde{X}, \wtilde{D})° \ar[d, two heads, "q°_{H°}"] \ar[rd, bend left=5, "∃!\wtilde{b}°"] \\
      X° \ar[rr, "\alb(X{,} D)°"] \ar[rrr, bend right=10, "a°"'] && \Alb(X,D)° \ar[r, dotted, "\text{want: }b°"] & A°.
    \end{tikzcd}
  \]
  Consider the element $\wtilde{b}°(0_{\Alb(\wtilde{X}, \wtilde{D})°}) ∈ A°$ to
  equip $A°$ with the structure of a Lie group that makes $\wtilde{b}°$ a group
  morphism.  Since the composed map
  \[
    \wtilde{b}° ◦ \alb(\wtilde{X}{,} \wtilde{D})° = \alb(X, D)° ◦ π°
  \]
  is constant on $π°$-fibres, the choice of $H°$ in Step~3 of
  Construction~\ref{cons:4-8} immediately guarantees that
  \[
    \ker q°_{H°} = H° ⊆ \ker \wtilde{b}°.
  \]
  It follows that there is a unique Lie group morphism $b° : \Alb(X,D)° → A°$
  that makes the diagram commute.  Item~\ref{il:2-4-2} of Lemma~\ref{lem:2-4}
  guarantees that $b°$ is quasi-algebraic, as desired.  The fact that
  $\wtilde{b}°$ is unique as a morphism of varieties implies that $b°$ is unique
  as a morphism of varieties.
\end{proof}

\subsection{Additional properties}
\approvals{Erwan & yes \\ Stefan & yes}

The Albanese has numerous properties that we will use in the sequel.  While all
of those necessarily follow from the universal property that determines the
Albanese uniquely, we find it often easier to use the concrete construction of
the Albanese in \ref{cons:4-8}, which quickly reduces us to the snc setting
where all results are known and readily citable.

\begin{prop}[Image of $\alb$ generates $\Alb$]\label{prop:4-10} %
  In Setting~\ref{set:4-1}, assume that $X$ is Kähler.  Let $x ∈ X°$ be any
  point and use
  \[
    0_{\Alb(X,D)°(x)} := \alb(X,D)°(x) ∈ \Alb(X,D)°
  \]
  to equip $\Alb(X,D)°$ with the structure of a Lie group.  Then, the image of
  $\alb(X,D)°$ generates $\Alb(X,D)°$ as an Abelian group.
\end{prop}
\begin{proof}
  If $(X,D)$ is snc, this is \cite[Prop.~4.5.11]{MR3156076}.  In general,
  consider Diagram~\eqref{eq:4-8-2} of Construction~\ref{cons:4-8}, use that
  $\img \alb(\wtilde{X},\wtilde{D})°$ generates $\Alb(\wtilde{X},\wtilde{D})°$
  and that the quotient map
  \[
    q°_{H°} : \Alb(\wtilde{X},\wtilde{D})° \twoheadrightarrow \Alb(X,D)°
  \]
  is surjective.
\end{proof}

\begin{prop}[Group actions]\label{prop:4-11}%
 In Setting~\ref{set:4-1}, assume that $X$ is Kähler.  Given a finite subgroup
  $G$ of $\Aut(X,D)$, there exists an Albanese $\Alb(X,D)° ⊂ \Alb(X,D)$ where
  $G$ acts on the pair $(\Alb(X,D), Δ_{\Alb(X,D)°})$ in a way that makes the
  morphisms
  \[
    \begin{tikzcd}[column sep=1.5cm]
      X° \ar[r, "\alb(X{,}D)°"] & \Alb(X,D)° \ar[r, hook] & \Alb(X,D)
    \end{tikzcd}
  \]
  equivariant.
\end{prop}

\begin{rem}[Equivariant presentation of the Albanese]
 In the setting of Proposition~\ref{prop:4-11}, the group $G$ will also act on
  the pair $(X,0)$ and hence on the Albanese $\Alb(X)$.  Continuing
  Remark~\ref{rem:4-5}, we leave it to the reader to check that the presentation
  morphism
  \[
    \begin{tikzcd}[column sep=1.5cm]
      \Alb(X,D) \ar[r] & \Alb(X)
    \end{tikzcd}
  \]
  is likewise $G$-equivariant.
\end{rem}

\begin{proof}[Proof of Proposition~\ref{prop:4-11}]
  The group action on $\Alb(X,D)°$ are of course induced by the universal
  property.  In fact, given any automorphism $g ∈ \Aut(X,D)$, consider the
  diagram
  \[
    \begin{tikzcd}[column sep=2cm]
      X° \ar[d, "g"'] \ar[r, "\alb(X{,}D)°"] & \Alb(X,D)° \arrow[d, "∃!σ(g)"] \\
      X° \ar[r, "\alb(X{,}D)°"'] & \Alb(X,D)°,
    \end{tikzcd}
  \]
  where $σ(g)$ is the quasi-algebraic morphism of semitori given by the
  universal property.  An elementary computation shows that the morphism
  \[
    \Aut(X,D) → \Aut\bigl(\Alb(X,D)°\bigr), \quad g ↦ σ(g)
  \]
  is indeed a group morphism that makes the morphism to $\Alb(X)$ equivariant.
  Corollary~\vref{cor:3-21} allows finding a $G$-equivariant, semitoric
  compactification.
\end{proof}

\subsubsection{Resolution of singularities}
\approvals{Erwan & yes \\ Stefan & yes}

Construction~\ref{cons:4-8} makes it easy to compare the Albanese of a pair with
the Albanese of a resolution of singularities.  To begin, we observe that a
surjection of pairs induces a surjection between the Albanese varieties.

\begin{obs}[Surjective morphisms]\label{obs:4-13}%
  Let $(X,D_X)$ and $(Y,D_Y)$ be two log pairs, where $X$ and $Y$ are compact
  Kähler spaces.  Given a quasi-algebraic surjection $φ° : X° \twoheadrightarrow
  Y°$, the universal property of the Albanese yields a diagram of the form
  \[
    \begin{tikzcd}[row sep=1cm, column sep=1.8cm]
      X° \ar[rr, "\alb(X{,} D_X)°"] \ar[d, two heads, "φ°"'] && \Alb(X, D_X)° \ar[d, "\alb(φ°)"] \\
      Y° \ar[rr, "\alb(Y{,} D_Y)°"'] && \Alb(Y,D_Y)°.
    \end{tikzcd}
  \]
  Choosing a point $x ∈ X°$ and using
  \begin{align*}
    0_{\Alb(X, D_X)°} & := \alb(X{,} D_X)°(x) && ∈ \Alb(X, D_X)° \\
    0_{\Alb(Y, D_Y)°} & := \alb(Y{,} D_Y)°(φ° x) && ∈ \Alb(Y, D_Y)°
  \end{align*}
  to equip $\Alb(X, D_X)°$ and $\Alb(Y, D_Y)°$ with Lie group structures,
  $\alb(φ°)$ becomes a quasi-algebraic Lie group morphism.  The image of
  $\alb(φ°)$ is thus a subgroup that contains the image of $\alb(Y, D_Y)°$ and
  hence generates $\Alb(Y,D_Y)°$ as a group.  It follows that $\alb(φ°)$ is
  surjective.
\end{obs}

\begin{prop}[The Albanese and the Albanese of a log resolution]\label{prop:4-14}
  In Setting~\ref{set:4-1}, assume that $X$ is Kähler.  Let $π : \wtilde{X} → X$
  be a log resolution of the pair $(X,D)$.  Consider the reduced divisor
  $\wtilde{D} := \supp π^{-1}(D)$ and the associated diagram
  \[
    \begin{tikzcd}[row sep=1cm, column sep=1.8cm]
      \wtilde{X}° \ar[rr, "\alb(\wtilde{X}{,} \wtilde{D})°"] \ar[d, two heads, "π°"'] && \Alb(\wtilde{X}, \wtilde{D})° \ar[d, two heads, "\alb(π°)\text{, surj.~by Obs.~\ref{obs:4-13}}"] \\
      X° \ar[rr, "\alb(X{,} D)°"'] && \Alb(X,D)°.
    \end{tikzcd}
  \]
  In particular, observe that
  \begin{equation}\label{eq:4-14-1}
    \dim \Alb(X,D)° ≤ \dim \Alb(\wtilde{X}, \wtilde{D})°.
  \end{equation}
  If $X°$ has only rational singularities, then $\alb(π°)$ is isomorphic and
  Inequality~\eqref{eq:4-14-1} is an equality.
\end{prop}
\begin{proof}
  The assumption that $X°$ has only rational singularities implies that every
  form $σ ∈ H⁰\bigl( \wtilde{X}°,\, Ω¹_{\wtilde X}\bigr)$ vanishes when
  restricted to the smooth locus of any $π°$-fibre, \cite[Lem.~1.2]{MR1819886}.
  This applies in particular to differential forms coming from $\Alb(\wtilde{X},
  \wtilde{D})°$.  Since the cotangent bundle of $\Alb(\wtilde{X}, \wtilde{D})°$
  is free, we find that $\alb(\wtilde{X}, \wtilde{D})°$ maps $π°$-fibres to
  points.  The map $\alb(\wtilde{X}, \wtilde{D})°$ therefore factors via $π°$,
  and the group $H°$ of Construction~\ref{cons:4-8} is therefore trivial, $H° =
  \{0\}$.
\end{proof}

\subsubsection{Description in terms of differentials}
\approvals{Erwan & yes \\ Stefan & yes}

As in the classic case, the Albanese of a singular pair can be described in
terms of differentials, as a Lie group quotient of a dualized space of
one-forms.  The following observation makes this statement precise.

\begin{obs}[Presentation of the Albanese as a Lie group quotient]\label{obs:4-15}%
  In the setting of Proposition~\ref{prop:4-14}, choose a point $\wtilde{x} ∈
  \wtilde{X}$ and use
  \begin{align*}
    0_{\Alb(\wtilde{X}, \wtilde{D})°} & := \alb(\wtilde{X}, \wtilde{D})°(\wtilde{x}) && ∈ \Alb(\wtilde{X}, \wtilde{D})° \\
    0_{\Alb(X, D)°} & := \alb(X, D)°(π° \wtilde{x}) && ∈ \Alb(X, D)°
  \end{align*}
  to equip $\Alb(\wtilde{X}, \wtilde{D})°$ and $\Alb(X, D)°$ with Lie group
  structures that make $\alb(π°)$ a group morphism.  Since $π$ is surjective,
  the push-forward of any torsion free sheaf is torsion free, and we obtain an
  injection
  \begin{equation}\label{il:4-15-1}
    π_* Ω¹_{\wtilde{X}}(\log \wtilde{D}) ↪ Ω^{[1]}_X(\log D),
  \end{equation}
  which presents $\Alb(X, D)°$ as a Lie group quotient,
  \begin{align}
    \label{il:4-15-2} H⁰\bigl( X,\, Ω^{[1]}_X(\log D) \bigr)^* & \twoheadrightarrow H⁰\bigl( \wtilde{X},\, Ω¹_{\wtilde{X}}(\log \wtilde{D}) \bigr)^* && \text{dual of \eqref{il:4-15-1}} \\
    \label{il:4-15-3} & \twoheadrightarrow \Alb(\wtilde{X},\wtilde{D})° && \text{quotient by } π_1(\wtilde{X}°,\wtilde{x}) \\
    \label{il:4-15-4} & \twoheadrightarrow \Alb(X, D)° && \text{quotient by quasi-algebraic.}
  \end{align}
  The pull-back morphism for logarithmic differentials introduced in
  Remark~\vref{rem:3-17},
  \[
    \diff \alb(X,D) : H⁰\Bigl( Ω¹_{\Alb(X,D)}(\log Δ)
    \Bigr) → H⁰\bigl( X,\, Ω^{[1]}_X (\log D)\bigr),
  \]
  is the induced map between dual Lie algebras, hence injective.
  Observation~\ref{obs:4-15} ends here.
\end{obs}

\begin{cor}[Dimension of $\Alb$]\label{cor:4-16}%
  In Setting~\ref{set:4-1}, assume that $X$ is Kähler.  Then, the dimension of
  $\Alb(X,D)°$ satisfies the inequality
  \begin{equation}\label{eq:4-16-1}
    \dim \Alb(X,D)° ≤ h⁰\bigl( X,\, Ω^{[1]}_X(\log D) \bigr).
  \end{equation}
  If the pair $(X,D)$ is Du~Bois and if $X°$ has rational singularities, then
  \eqref{eq:4-16-1} is an equality.
\end{cor}
\begin{proof}
  The inequality follows directly from Observation~\ref{obs:4-15} above.
  Assuming that $(X,D)$ is Du~Bois and that $X°$ has rational singularities, we
  show that the composed surjection \eqref{il:4-15-2}--\eqref{il:4-15-4} has a
  discrete kernel.

  To begin, recall that since $X°$ has rational singularities,
  Proposition~\ref{prop:4-14} asserts that \eqref{il:4-15-4} is an isomorphism.
  Its kernel is hence trivial.  The kernel of \eqref{il:4-15-3} is discrete.  We
  claim that \eqref{il:4-15-2} is likewise isomorphic.  To this end, decompose
  \eqref{il:4-15-1} as
  \begin{equation}\label{eq:4-16-2}
    π_* Ω¹_{\wtilde{X}}(\log \wtilde{D})
    \overset{a}{↪} π_* Ω¹_{\wtilde{X}}(\log \wtilde{D} + \operatorname{Exc} π)
    \overset{b}{↪} Ω^{[1]}_X(\log D).
  \end{equation}
  Recall from \cite[Cor.~1.8, Rem.~1.9]{KS18} that $a$ is isomorphic because
  $X°$ has rational singularities.  Recall from \cite[Thm.~4.1]{zbMATH06405804}
  that $b$ is isomorphic because $(X,D)$ is Du~Bois.
\end{proof}

\begin{rem}[Relation to Minimal Model Theory]
  Recall the classic results that log-canonical pairs are Du~Bois and that the
  space underlying a log-terminal pair has rational singularities.
  Corollary~\ref{cor:4-16} will therefore give an equality if the pair $(X,D)$
  is dlt in the sense of Minimal Model Theory, \cite[Def.~2.37]{KM98}.
\end{rem}

\begin{rem}[Improvements]
  Corollary~\ref{cor:4-16} is probably not optimal.  Using the notion of
  ``weakly rational singularities'' introduced in \cite[Sect.~1.4]{KS18} and the
  extension results of \cite{park2024du, tighe2024holomorphic}, the assumptions
  on rational singularities might be weakened, at the expense of introducing
  technically challenging singularity classes, \cite[Thm.~5.23]{KM98} and
  \cite[Sect.~6.2]{MR3057950}.
\end{rem}

We leave the proof of the following fact to the reader.

\begin{fact}[Image of $\diff a$ and $\ker(b)$]\label{fact:4-19}%
  In the setting of Observation~\ref{obs:4-15}, assume we are given a
  factorization as in Diagram~\eqref{eq:4-2-1}.  Consider the linear subspace
  \[
    W := \img \Bigl( \diff a : H⁰\bigl( A, Ω¹_A(\log Δ_A) \bigr) →
    H⁰\bigl( X,\, Ω^{[1]}_X(\log D) \bigr) \Bigr),
  \]
  write $W^{\perp} ⊆ H⁰\bigl( X,\, Ω^{[1]}_X(\log D) \bigr)^*$ for its
  annihilator and recall from Observation~\ref{obs:4-15} above that there exists
  a natural surjection of Lie groups
  \[
    η : H⁰\bigl( X,\, Ω^{[1]}_X(\log D) \bigr)^* \twoheadrightarrow \Alb(X,D)°.
  \]
  Then, $\ker(b°) = η\bigl( W^{\perp} \bigr)$.  \qed
\end{fact}

\subsection{Examples}
\approvals{Erwan & yes \\ Stefan & yes}

The following example shows that the Inequalities~\eqref{eq:4-14-1} and
\eqref{eq:4-16-1} will generally be strict, even for pairs with no boundary and
with the simplest log-canonical singularities.

\begin{example}[Strict inequalities]\label{ex:4-20}%
  Consider closed immersions $E ⊊ ℙ² ⊊ ℙ³$ where $E$ is an elliptic curve and
  where $ℙ²$ is linearly embedded into $ℙ³$.  Let $X ⊂ ℙ³$ be the projective
  cone over $E$.  Since $X$ is rationally connected, morphisms to semitori will
  necessarily be constant.  It follows that the Albanese of $(X,0)$ will be
  trivial.  Next, let $π : \wtilde{X} → X$ be the resolution of singularities,
  obtained as the blow-up of the unique singular point in $X$.  Since
  $\wtilde{X}$ is a $ℙ¹$-bundle over $E$, its Albanese equals $E$.  The
  following diagram summarizes the situation,
  \[
    \begin{tikzcd}[row sep=1cm]
      \wtilde{X} \ar[rrr, "\alb(\wtilde{X})"] \ar[d, two heads, "π"'] &&& \Alb(\wtilde{X}) \ar[d, two heads, "\alb(π)"] \ar[r, phantom, "="] & E \\
      X \ar[rrr, "\alb(X)"'] &&& \Alb(X) \ar[r, phantom, "="] & \{0\}.
    \end{tikzcd}
  \]
  Inequality~\eqref{eq:4-14-1} is strict in this case.  The inequalities
  \[
    1 = h⁰\bigl( E,\, Ω¹_E \bigr) ≤ h⁰\bigl( \wtilde{X} ,\, Ω¹_{\wtilde{X}} \bigr) = h⁰\bigl( \wtilde{X} ,\, π_* Ω¹_{\wtilde{X}} \bigr) ≤
    h⁰\bigl( X,\, Ω^{[1]}_X \bigr)
  \]
  show that \eqref{eq:4-16-1} is likewise strict.
\end{example}

In case where the underlying space $X$ of a pair $(X,D)$ is smooth, the
following example shows that the Albanese of $(X,D)$ agrees with the Albanese of
a log resolution.  Together with \cite[Rem.~4.5.10]{MR3156076}, this implies
that our construction of Albanese agrees with that of
\cite[Sect.~4.5]{MR3156076}, even though the two constructions might initially
look different.

\begin{example}[Albanese in case where the underlying space is smooth]
  In Setting~\ref{set:4-1}, assume that $X$ is a Kähler manifold and let
  $\Alb(X,D)° ⊂ \Alb(X,D)$ be an Albanese of $(X,D)$, with Albanese morphism
  $\alb(X,D)° : X° → \Alb(X,D)°$.

  If $π : \wtilde{X} \twoheadrightarrow X$, is a log-resolution of the pair
  $(X,D)$ and $\wtilde{D} := \supp π^{-1}(D)$ the reduced preimage divisor, then
  Construction~\ref{cons:4-8} immediately shows that $\Alb(X,D)° ⊂ \Alb(X,D)$ is
  also an Albanese of $(\wtilde{X}, \wtilde{D})$, with Albanese morphism
  $\alb°(\wtilde{X},\wtilde{D}) = \alb(X,D)° ◦ π$.
\end{example}


\phantomsection\addcontentsline{toc}{part}{The Albanese of a cover}

%
%
\svnid{$Id: 05-adaptedAlbanese.tex 942 2024-10-10 16:13:56Z kebekus $}
\selectlanguage{british}

\section{The Albanese of a cover}
\label{sec:5}
\subversionInfo
\approvals{Erwan & yes \\ Stefan & yes}

Generalizing the Albanese of a logarithmic pair, we construct an Albanese
attached to every cover $\what{X} \twoheadrightarrow X$ of a given $\cC$-pair
$(X,D)$, which need not be logarithmic.  Recalling that the Albanese of a
logarithmic snc pair is a ``universal'' morphism to a semitoric variety that
induces all logarithmic differentials, we define the Albanese of a cover as a
``universal'' morphism from $\what{X}$ to a semitoric variety such that every
pull-back differential is adapted.  We consider the following setting throughout
the present section.

\begin{setting}\label{set:5-1} %
  Let $(X, D)$ be a $\cC$-pair where $X$ is compact and let $γ : \what{X}
  \twoheadrightarrow X$ be a cover of $(X,D)$.  Consider the reduced divisor
  \[
    \what{D} := \bigl(γ^* ⌊D⌋ \bigr)_{\red} ∈ \Div(\what{X})
  \]
  and write $\what{X}° := \what{X} ∖ \supp \what{D}$.
\end{setting}

We underline that Setting~\ref{set:5-1} does \emph{not} assume that $γ$ is
adapted, that $\what{X}$ is smooth, or that $γ^* D$ has nc support.  The
following definition of the Albanese will therefore use adapted \emph{reflexive}
differentials.

\begin{defn}[The Albanese of a cover of a $\cC$-pair]\label{def:5-2} %
  Assume Setting~\ref{set:5-1}.  An \emph{Albanese of $(X,D,γ)$} is a semitoric
  variety $\Alb(X,D,γ)° ⊂ \Alb(X,D,γ)$ together with a quasi-algebraic morphism
  \[
    \alb(X,D,γ)° : \what{X}° → \Alb(X,D,γ)°
  \]
  such that the following holds.
  \begin{enumerate}
    \item\label{il:5-2-1} The pull-back morphism for logarithmic differentials
    of Remark~\ref{rem:3-17},
    \[
      H⁰\Bigl( Ω¹_{\Alb(X,D,γ)}(\log Δ) \Bigr) \xrightarrow{\diff \alb(X,D,γ)} H⁰\Bigl( \what{X},\, Ω^{[1]}_{\what{X}}(\log \what{D}) \Bigr),
    \]
    takes its image in the subspace $H⁰\Bigl( \what{X},\, Ω^{[1]}_{(X,D,γ)}
    \Bigr) ⊆ H⁰\Bigl( \what{X},\, Ω^{[1]}_{\what{X}}(\log \what{D}) \Bigr)$.

    \item\label{il:5-2-2} If $A° ⊂ A$ is any semitoric variety, if $a° :
    \what{X}° → A°$ is any quasi-algebraic morphism such that the pull-back
    morphism
    \[
      \diff a : H⁰\bigl( A,\, Ω¹_A(\log Δ) \bigr) → H⁰\Bigl( \what{X},\,
      Ω^{[1]}_{\what{X}} (\log \what{D}) \Bigr)
    \]
    takes its image in $H⁰\Bigl( \what{X},\, Ω^{[1]}_{(X,D,γ)} \Bigr)$, then $a$
    factors uniquely as
    \[
      \begin{tikzcd}[column sep=2cm]
        \what{X}° \arrow[r, "\alb(X{,}D{,}γ)°"'] \arrow[rr, "a°", bend left=10] & \Alb(X,D,γ)° \arrow[r, "∃!b°"'] & A°,
      \end{tikzcd}
    \]
    where $b°$ is quasi-algebraic.
  \end{enumerate}
\end{defn}

\begin{rem}[Pull-back of $p$-differentials]\label{rem:5-3} %
  Item~\ref{il:5-2-1} of Definition~\ref{def:5-2} can be phrased in terms of
  sheaf morphisms.  Recall from \cite[Cor.~5.4.5]{MR3156076} that the locally
  free sheaf $Ω¹_{\Alb(X,D,γ)}(\log Δ)$ is free and hence globally generated.
  Item~\ref{il:5-2-1} is therefore equivalent to the following, seemingly
  stronger statement: If $p$ is any number, then the composed pull-back morphism
  \[
    (\alb(X,D,γ)°)^* \: Ω^p_{\Alb(X,D,γ)°} → Ω^{[p]}_{\what{X}°}
  \]
  takes its image in the subsheaf $Ω^{[p]}_{(X°, D°, γ)} ⊆ Ω^{[p]}_{\what{X}°}$.
\end{rem}

\subsection{Uniqueness and existence}
\approvals{Erwan & yes \\ Stefan & yes}

As before, the universal property spelled out in Item~\ref{il:5-2-2} implies
that $\Alb(X,D,γ)°$ is unique up to unique isomorphism.  The compactification
$\Alb(X,D,γ)$ is bimeromorphically unique.  As before, we abuse notation and
refer to any Albanese as ``the Albanese'', with associated semitoric
\emph{Albanese variety} $\Alb(X,D,γ)° ⊂ \Alb(X,D,γ)$ and quasi-algebraic
\emph{Albanese morphism} $\alb(X,D,γ)°$.

\begin{prop}[Existence of the Albanese of a cover]\label{prop:5-4} %
  In Setting~\ref{set:5-1}, assume that $X$ is Kähler.  Then, an Albanese of
  $(X,D,γ)$ exists.  Its dimension satisfies the inequality
  \[
    \dim \Alb(X,D,γ)° ≤ q(X,D,γ).
  \]
  If $\:\what{x} ∈ \what{X}°$ is any point and if we use
  \[
    0_{\Alb(X,D,γ)°} := \alb(X,D,γ)° (\what{x}) ∈ \Alb(X,D,γ)°
  \]
  to equip $\Alb(X,D,γ)°$ with the structure of a Lie group, then the image of
  $\alb(X,D,γ)°$ generates $\Alb(X,D,γ)°$ as an Abelian group.
\end{prop}

The proof of Proposition~\ref{prop:5-4} requires some preparation.  We give it
in Section~\ref{sec:7-2}, starting from Page~\pageref{sec:7-2} below.  Assuming
for the moment that the Albanese can be shown to exist, the subsequent
Sections~\ref{sec:5-2}--\ref{sec:5-4} gather its most important properties.

\subsection{Functoriality in sequences of covers}
\approvals{Erwan & yes \\ Stefan & yes}
\label{sec:5-2}

The following immediate consequence of the universal property will be used
later.

\begin{lem}[Functoriality of the Albanese]\label{lem:5-5} %
  Let $(X, D)$ be a $\cC$-pair where $X$ is compact Kähler.  Let
  \[
    \begin{tikzcd}
      \what{X}_1 \ar[r, two heads, "γ_1"] & \what{X}_2 \ar[r, two heads, "γ_2"] & X
    \end{tikzcd}
  \]
  be a sequence of covers.  Consider the reduced divisors
  \[
    \what{D}_2 := \left(γ^*_2 ⌊D⌋ \right)_{\red} \quad\text{and}\quad \what{D}_1
    := \bigl((γ_2◦γ_1)^* ⌊D⌋ \bigr)_{\red}
  \]
  and write $\what{X}°_• := \what{X}_• ∖ \supp \what{D}_•$.  Then, there exists
  a unique surjection $c°$ that renders the following diagram commutative,
  \begin{equation}\label{eq:5-5-1}
    \begin{tikzcd}[column sep=2.4cm, row sep=1cm]
      \what{X}°_1 \ar[d, two heads, "γ_1|_{\what{X}°_1}"'] \ar[r, "\alb(X{,}D{,}γ_2◦γ_1)°"] & \Alb(X,D,γ_2◦γ_1)° \ar[d, two heads, "∃!c°"] \\
      \what{X}°_2 \ar[r, "\alb(X{,}D{,}γ_2)°"'] \ar[d, two heads, "γ_2|_{\what{X}°_2}"'] & \Alb(X,D,γ_2)° \\
      X°.
    \end{tikzcd}
  \end{equation}
  The morphism $c°$ is quasi-algebraic.
\end{lem}
\begin{proof}
  Uniqueness and surjectivity of $c°$ (if it exists) follows from
  Proposition~\ref{prop:5-4}, which asserts that the images of $\alb_•(•)°$
  generate $\Alb_•(•)°$ as groups once suitable structures of Lie groups are
  chosen.

  Existence of $c°$ as a quasi-algebraic morphism follows from the universal
  property of the Albanese.  To be precise, recall from Property~\ref{il:5-2-1}
  that the pull-back morphism
  \[
    \alb(X,D,γ_2)^* : H⁰\bigl( Ω¹_{\Alb(X,D,γ_2)}(\log Δ)\bigr) →
    H⁰\Bigr( \what{X}_2,\, Ω^{[1]}_{\what{X}_2}(\log \what{D}_1) \Bigr)
  \]
  takes its image in $H⁰\bigr( \what{X}_2,\, Ω^{[1]}_{(X,D,γ_2)} \bigr)$.  As a
  consequence, we find that the pull-back morphism
  \[
    \bigl(\alb(X,D,γ_2) ◦ γ_1\bigr)^* : H⁰\bigl( Ω¹_{\Alb(X,D,γ_2)}(\log Δ)\bigr) →
    H⁰\Bigl( \what{X}_1,\, Ω^{[1]}_{\what{X}_1}(\log \what{D}_1) \Bigr)
  \]
  takes its image in
  \begin{equation}\label{eq:5-5-2}
    H⁰\bigr( \what{X}_1,\, γ^{[*]}_1 Ω^{[1]}_{(X,D,γ_2)} \bigr)
    ⊆ H⁰\bigr( \what{X}_1,\, Ω^{[1]}_{(X,D,γ_2◦γ_1)} \bigr),
  \end{equation}
  where the inclusion in \eqref{eq:5-5-2} is \cite[Obs.~4.14]{orbiAlb1}.  As
  pointed out above, the universal property of the Albanese $\Alb(X,D,γ_2◦γ_1)°$
  now gives a unique quasi-algebraic morphism $c°$ that makes
  Diagram~\eqref{eq:5-5-1} commute.
\end{proof}

\subsection{The Albanese of a Galois cover}
\label{sec:5-3}
\approvals{Erwan & yes \\ Stefan & yes}

Lemma~\ref{lem:5-5} applies in particular in case where $\what{X}_1 =
\what{X}_2$ are equal and where $γ_1$ is a Galois automorphism of the cover
$γ_2$.  We find that the Galois group acts on the Albanese and that the Albanese
morphism is equivariant.

\begin{obs}[Galois action on the Albanese of a cover]\label{obs:5-6} %
  In Setting~\ref{set:5-1}, assume that $X$ is Kähler and that the cover $γ$ is
  Galois with group $G$.  Recall from \cite[Obs.~4.19]{orbiAlb1} that
  $Ω^{[1]}_{(X,D,γ)}$ carries a natural $G$-linearisation that is compatible
  with the natural $\Aut(\what{X})$-linearisations of $Ω^{[1]}_{\what{X}}$.  In
  complete analogy to Proposition~\ref{prop:4-11}, it follows from
  Lemma~\ref{lem:5-5} that $G$ acts on $\Alb(X,D,γ)°$ by quasi-algebraic
  automorphisms, in a way that makes the morphism $\alb(X,D,γ)°$ equivariant.
  Corollary~\vref{cor:3-21} allows choosing a compactification
  \[
    \Alb(X,D,γ)° ⊂ \Alb(X,D,γ), %
    \quad\text{written in short as } %
    A° ⊂ A,
  \]
  such that the $G$-action on $A°$ extends to $A$, and such that $A° ⊂ A$ is an
  Albanese for $(X,D,γ)$.  \qed
\end{obs}

\begin{construction}[Morphism to Galois quotient of the Albanese of a cover]\label{cons:5-7}%
  In Observation~\ref{obs:5-6}, take quotients to find a diagram
  \begin{equation}\label{eq:5-7-1}
    \begin{tikzcd}[column sep=2cm, row sep=1cm]
      \what{X}° \arrow[r, "\alb_{\what{x}}({X,D,γ})°"] \ar[d, two heads, "γ\text{, quotient}"'] & A° \ar[d, two heads, "γ_A\text{, quotient}"]\\
      X° \arrow[r, "a°"'] & \factor{A°}{G}
    \end{tikzcd}
  \end{equation}
  where $a°$ is quasi-algebraic for the compactifications $X° ⊂ X$ and $A°/G ⊂
  A/G$.  Propositions~\ref{prop:5-4} and \ref{prop:3-29} together guarantee that
  the image of $\alb_{\what{x}}(X,D,γ)°$ is not contained in the ramification
  locus of the quotient morphism $γ_A : A° → A°/G$.  The image of $a°$ is
  therefore not contained in the branch locus.
\end{construction}

Diagram~\eqref{eq:5-7-1} is a commutative diagram of holomorphic morphisms
between normal analytic varieties.  We upgrade it to a commutative diagram of
$\cC$-morphisms.

\begin{obs}[$\cC$-Morphism to Galois quotient of the Albanese]\label{obs:5-8}%
  The variety $A°$ of Observation~\ref{obs:5-6} and Construction~\ref{cons:5-7}
  is a semitorus and therefore smooth.  The criterion for $\cC$-morphisms
  spelled out in \cite[Prop.~8.6]{orbiAlb1} therefore applies to show that
  $\alb_{\what{x}}(X,D,γ)°$ induces a morphism of $\cC$-pairs\footnote{In
  contrast, recall from \cite[Ex.~8.7 and 8.8]{orbiAlb1} that a morphism between
  singular spaces $Z_1 → Z_2$ does not always induce a $\cC$-morphism $(Z_1,0) →
  (Z_2,0)$.},
  \[
    \alb(X,D,γ)° : (\what{X}°, 0) → (A°, 0).
  \]
  Taking the categorical quotients of $\cC$-pairs, \cite[Prop.~12.7]{orbiAlb1}
  will thus yield a diagram of $\cC$-morphisms between $\cC$-pairs as follows,
  \begin{equation}\label{eq:5-8-1}
    \begin{tikzcd}[column sep=2cm, row sep=1cm]
      (\what{X}°, 0) \arrow[r, "\alb({X,D,γ})"] \ar[d, two heads, "γ\text{, quotient}"'] & (A°, 0) \ar[d, two heads, "γ_A\text{, quotient}"]\\
      (X°, D') \arrow[r, "a°"'] & (Y°, D_Y),
    \end{tikzcd}
  \end{equation}
  where
  \[
    (X°, D') := \factor{(\what{X}°, 0)}{G} %
    \quad\text{and}\quad %
    (Y°, D_Y) := \factor{(A°, 0)}{G}.  %
  \]
\end{obs}

\begin{warning}[Boundary divisors in the quotient construction]\label{warn:5-9}%
  The boundary divisor $D'$ in Observation~\ref{obs:5-8} does not need to equal
  $D°$.  In fact, recall from \cite[Obs.~12.9]{orbiAlb1} that there is only an
  inequality $D' ≥ D°$, which might be strict.  As before,
  \cite[Prop.~10.4]{orbiAlb1} allows formulating this inequality by saying that
  the identity on $X°$ induces a morphism of $\cC$-pairs,
  \[
    \Id_{X°} : (X°, D') → (X°, D°).
  \]
  Warning~\ref{warn:5-9} ends here.
\end{warning}

The following proposition, which is central to everything that follows, claims
that in spite of Warning~\ref{warn:5-9}, the morphism $a°$ of
Diagram~\eqref{eq:5-8-1} does induce a morphism of $\cC$-pairs,
\[
  \underline{a}° : (X°, D°) → (Y°, D_Y).
\]
This is expressed in technical terms by saying that the quasi-algebraic
$\cC$-morphism $a°$ of Diagram~\eqref{eq:5-8-1} factorizes via the
$\cC$-morphism $\Id_{X°}$ that we discussed in Warning~\ref{warn:5-9}.

\begin{prop}\label{prop:5-10} %
  In the setting of Observation~\ref{obs:5-6} and Warning~\ref{warn:5-9}, the
  quasi-algebraic $\cC$-morphism $a°$ of Diagram~\eqref{eq:5-8-1} factorizes
  via $\Id_{X°} : (X°, D') → (X°, D°)$.  In other words, we obtain a diagram of
  $\cC$-morphisms,
  \[
    \begin{tikzcd}[column sep=1cm, row sep=1cm]
      (\what{X}°, 0) \arrow[rr, "\alb({X,D,γ})°"] \ar[d, two heads, "γ\text{, quotient}"'] && (A°, 0) \ar[d, two heads, "γ_A\text{, quotient}"] \\
      (X°, D') \arrow[r, "\Id_{X°}"] \ar[rr, bend right=15, "a°"'] & (X°, D°) \ar[r, "\underline{a}°"] & (Y°, D_Y),
    \end{tikzcd}
  \]
  where $\img a° = \img \underline{a}°$ is not contained in the branch locus of
  the quotient morphism $γ_A$.
\end{prop}
\begin{proof}
  We aim to apply the criterion for $\cC$-morphisms spelled out in
  \cite[Prop.~9.3]{orbiAlb1} and consider the sub-diagram
  \[
    \begin{tikzcd}[column sep=1.8cm, row sep=1cm]
      \what{X}° \arrow[r, "\alb({X,D,γ})°"] \ar[d, two heads, "γ\text{, quotient}"'] & A° \ar[d, two heads, "γ_A\text{, quotient}"] \\
      X° \ar[r, "\underline{a}°"'] & Y°.
    \end{tikzcd}
  \]
  Recall \cite[Obs.~12.10]{orbiAlb1}, which asserts that $γ_A$ is strongly
  adapted for the $\cC$-pair $(Y°, D_Y)$, and that the $\cC$-cotangent sheaf is
  $Ω^{[1]}_{(Y°, D_Y, γ_A)} = Ω¹_{A°}$.  Given that $A°$ is a semitorus, we find
  that $Ω^{[1]}_{(Y°, D_Y, γ_A)}$ is locally free.  The criterion for
  $\cC$-morphisms, \cite[Prop.~9.3]{orbiAlb1} therefore applies to show that
  $\underline{a}°$ is a $\cC$-morphism as soon as we show that there exists a
  sheaf morphism
  \[
    \diff \alb(X,D,γ)° \::\: \Bigl(\alb(X,D,γ)°\Bigr)^* \: Ω^{[1]}_{(Y°, D_Y, γ_A)} → Ω^{[1]}_{(X°, D°, γ)}
  \]
  that agrees with the standard pull-back of Kähler differentials wherever this
  makes sense.  That is however precisely the statement of Remark~\ref{rem:5-3}.
  The fact that $\img a° = \img \underline{a}°$ is not contained in the branch
  locus of the quotient morphism $γ_A$ has already been remarked in
  Construction~\ref{cons:5-7}.
\end{proof}

\subsection{Functoriality in sequences of Galois covers}
\approvals{Erwan & yes \\ Stefan & yes}
\label{sec:5-4}

The following lemma combines and summarizes the results of
Sections~\ref{sec:5-2} and \ref{sec:5-3}.

\begin{lem}[Functoriality of the Albanese]\label{lem:5-11} %
  In the setting of Lemma~\ref{lem:5-5}, assume that the covering morphisms
  $γ_2◦γ_1$ and $γ_2$ are Galois, with groups $G_2$ and $G_1$ respectively.
  Then, there exists a commutative diagram
  \[
    \begin{tikzcd}[column sep=2.2cm, row sep=1cm]
      \what{X}°_1 \ar[d, two heads, "γ_1|_{\what{X}°_1}"'] \ar[r, "\alb(X{,}D{,}γ_2◦γ_1)°"] & \Alb(X,D,γ_2◦γ_1)° \ar[rd, two heads, bend left=10, "c°\text{, quot.\ of semitori}"] \ar[dd, two heads, near start, "\text{quotient}"] \\
      \what{X}°_2 \ar[rr, near end, crossing over, "\alb(X{,}D{,}γ_2)°"'] \ar[d, two heads, "γ_2|_{\what{X}°_2}"'] && \Alb(X,D,γ_2)° \ar[d, two heads, "\text{quotient}"] \\
      X° \ar[r, "\underline{\alb}(X{,}D{,}γ_2◦γ_1)°"] \ar[rr, bend right=10, "\underline{\alb}(X{,}D{,}γ_2)°"'] & \factor{\Alb(X,D,γ_2◦γ_1)°}{G_1} \ar[r, two heads, "\underline{c}°"] & \factor{\Alb(X,D,γ_2)°}{G_2}
    \end{tikzcd}
  \]
  where all morphisms are quasi-algebraic and all morphisms in the bottom row
  are morphisms of $\cC$-pairs, between $(X°,D°)$ and the natural
  $\cC$-structures on the quotient pairs.
\end{lem}
\begin{proof}
  Except for the morphism $\underline{c}°$, the diagram is a combination of
  Lemma~\ref{lem:5-5} and Proposition~\ref{prop:5-10} above.  In order to
  construct $\underline{c}°$, observe that the group $G_2$ is a quotient $q :
  G_1 \twoheadrightarrow G_2$, and that $G_1$ acts on $\what{X}°_2$ via this
  quotient map, in a manner that makes the morphism $γ_1|_{\what{X}°_1}$
  equivariant.  The universal properties of the two Albanese maps
  $\alb(X,D,γ_2◦γ_1)°$ and $\alb(X,D,γ_2)°$ will then guarantee that $G_1$ acts
  on the Albanese varieties $\Alb(X,D,γ_2◦γ_1)°$ and $\Alb(X,D,γ_2)°$ in a
  manner that makes the quotient morphism $c°$ equivariant.  The map
  $\underline{c}°$ is then the induced $\cC$-morphism between the quotients
  pairs, as given by the universal property of $\cC$-pair quotients,
  \cite[Def.~12.3 and Thm.~12.5]{orbiAlb1}.
\end{proof}

%
%
\svnid{$Id: 06-albaneseIrregularity.tex 947 2024-11-11 08:48:24Z kebekus $}
\selectlanguage{british}

\section{The Albanese irregularity}
\label{sec:6}
\subversionInfo

\approvals{Erwan & yes \\ Stefan & yes}

Given a $\cC$-pair $(X,D)$ and a cover $γ: \what{X} \twoheadrightarrow X$, the
dimension of the Albanese is an important invariant of the triple $(X,D,γ)$.

\begin{defn}[Albanese irregularity, augmented Albanese irregularity]
  Assume Setting~\ref{set:5-1}.  If an Albanese exists, then refer to the number
  \[
    q_{\Alb}(X,D,γ) := \dim \Alb(X,D,γ)°
  \]
  as the \emph{Albanese irregularity of $(X,D,γ)$}.  The number
  \[
    q⁺_{\Alb}(X,D) = \sup \bigl\{ q_{\Alb}(X,D,γ) \,|\, γ \text{ a cover}
    \bigr\} ∈ ℕ ∪ \{ ∞ \}
  \]
  is the \emph{augmented Albanese irregularity of the $\cC$-pair $(X, D)$}.
\end{defn}

We will show in Section~\ref{sec:8} that the augmented Albanese irregularity
$q⁺_{\Alb}(X,D)$ is finite if $X$ is Kähler and if the $\cC$-pair $(X,D)$ is
special.

\subsection{Inequalities between irregularities}
\approvals{Erwan & yes \\ Stefan & yes}

\CounterStep{}If $(X, D)$ is a $\cC$-pair where $X$ is compact Kähler and if $γ
: \what{X} \twoheadrightarrow X$ is a cover of $(X,D)$, we have seen in
Proposition~\ref{prop:5-4} that the Albanese irregularity is bounded by the
irregularity,
\begin{equation}\label{eq:6-2-1}
  q_{\Alb}(X,D,γ) ≤ q(X,D,γ).
\end{equation}
There are settings where Inequality~\eqref{eq:6-2-1} is strict and the natural
morphism
\[
  H⁰\Bigl( Ω¹_{\Alb(X,D,γ)}(\log Δ) \Bigr)
  \xrightarrow{\diff \alb(X,D,γ)}
  H⁰\Bigl( \what{X},\, Ω^{[1]}_{(X,D,γ)} \Bigr)
\]
is not surjective.  Equivalently said: there are settings where it is not true
that every adapted reflexive differential on $\what{X}$ comes from a logarithmic
differential on the Albanese.  A first example has already been discussed in the
previous section.

\begin{example}[Strict inequality between irregularities]\label{ex:6-3}%
  Let $X$ be the cone over the elliptic curve discussed in
  Example~\vref{ex:4-20}.  Let $\Id_X : X → X$ be the trivial covering.  We have
  seen in Example~\ref{ex:4-20} that
  \[
    q(X,0,\Id_X) = h⁰ \bigl( X,\, Ω¹_{(X,0,\Id_X)} \bigr)
    = h⁰ \bigl( X,\, Ω^{[1]}_X \bigr)
  \]
  is positive while $\Alb(X,0,\Id_X) = \Alb(X)$ is a point, so that
  $q_{\Alb}(X,0,\Id_X) = 0$.
\end{example}

Example~\ref{ex:6-3} might seem artificial, given that $X$ has an elliptic
singularity.  With more work, we will show that Inequality~\eqref{eq:6-2-1} can
be strict, even in the simplest case where $X$ is a smooth curve and the
boundary $D$ is empty.

\begin{prop}[Strict inequality between irregularities for curve covers]\label{prop:A-1}%
  There exists a curve $X$ of genus $g(X) = 5$ and a cover $γ : \what{X}
  \twoheadrightarrow X$ such that $q_{\Alb}(X,0,γ) < q(X,0,γ)$.
\end{prop}

We will prove Proposition~\ref{prop:A-1} in Section~\ref{sec:A-3} below.  We
thank Frédéric Campana for help and numerous discussions on the subject.

\subsection{Preparation for the proof of Proposition~\ref{prop:A-1}}
\approvals{Erwan & yes \\ Stefan & yes}

The proof of Proposition~\ref{prop:A-1} is motivated by the folklore principle
that Prym varieties of sufficiently general curve covers are simple.  For
brevity, we restrict ourselves to showing that \emph{one} cover with simple Prym
variety and suitable numerical invariants exists.

\begin{prop}[Curve covers with simple Prym variety]\label{prop:A-2}%
  There exists a curve $X$ of genus $g(X) = 5$ and a 2:1-cover $γ : \what{X}
  \twoheadrightarrow X$, branched in exactly two points, whose associated Prym
  variety $P_γ$ is a connected, simple group.
\end{prop}
\begin{proof}
  Consider the Prym map $\operatorname{Pr}_5(2, 2)$, as discussed in
  Fact~\vref{fact:A-6}.  Recall from \cite[Cor.~5.8]{MR3877472} that
  $\operatorname{Pr}_5(2, 2)$ is generically finite.  If
  \[
    P := \overline{\img \operatorname{Pr}_5(2, 2)} ⊂ {\mathcal A}_{5,D}
  \]
  denotes the Zariski closure of the image, this implies that
  \[
    \dim P = \dim \operatorname{Pr}_5(2, 2) \overset{\eqref{eq:A-1-1}}{=} 14
    \quad \text{while} \quad
    \dim {\mathcal A}_{5,D} \overset{\text{Fact~\ref{fact:A-7}}}{=} 15,
  \]
  so $P$ is a divisor in ${\mathcal A}_{5,D}$.  But then, it follows from
  \cite[Prop.~3.4.i]{MR967018} that the general element of $P$ corresponds to a
  simple Abelian variety.  In particular, general elements of the constructible
  set $\img \operatorname{Pr}_5(2, 2)$ correspond to covers with simple Prym
  varieties.
\end{proof}

\subsection{Proof of Proposition~\ref{prop:A-1}}
\approvals{Erwan & yes \\ Stefan & yes}
\label{sec:A-3}

\CounterStep{}Following Proposition~\ref{prop:A-2}, let $X$ be a curve of genus
$g(X) = 5$ and a 2:1-cover $γ : \what{X} \twoheadrightarrow X$, branched in
exactly two points, whose associated Prym variety $P_γ$ is a connected, simple
group.  We will show that $q_{\Alb}(X,0,γ) < q(X,0,γ)$.

\subsubsection*{Step 1: A lower bound for $q(X,0,γ)$}
\approvals{Erwan & yes \\ Stefan & yes}

As in Remark~\ref{rem:A-4} use the action of the Galois group to decompose the
push-forward of $𝒪_{\what{X}}$ into an invariant and an anti-invariant part,
\begin{equation}\label{eq:A-8-1}
  γ_* 𝒪_{\what{X}} ≅ 𝒪_X ⊕ ℒ^*.
\end{equation}
The assumption that $γ$ is branched in exactly two points shows that $ℒ$ is a
line bundle with $\deg ℒ = 1$.  We can then bound $q(X,0,γ)$ as follows,
\begin{align}
  q(X,0,γ) & = h⁰ \bigl( \what{X},\, Ω¹_{(X,0,γ)} \bigr) && \text{Definition} \nonumber \\
  & = h⁰ \bigl( \what{X},\, γ^* ω_X \bigr) = h⁰ \bigl( X,\, γ_* γ^* ω_X \bigr) && \text{\cite[Ex.~3.5(2)]{orbiAlb1}} \nonumber \\
  & = h⁰ \bigl( X,\, ω_X ⊗ γ_* 𝒪_{\what{X}} \bigr) && \text{proj.~formula} \nonumber \\
  & = h⁰ \bigl( X,\, ω_X ⊗ (𝒪_X ⊕ ℒ^*) \bigr) && \text{\eqref{eq:A-8-1}} \nonumber \\
  & = q(X) + h⁰ \bigl( X,\, ω_X ⊗ ℒ^* \bigr).  \nonumber \\
  \intertext{The difference term is estimated by Riemann-Roch,}
  h⁰ \bigl( X,\, ω_X ⊗ ℒ^* \bigr) & ≥ h⁰ \bigl( X,\, ω_X ⊗ ℒ^* \bigr) - h¹ \bigl( X,\, ω_X ⊗ ℒ^* \bigr) \nonumber \\
  & = \deg (ω_X ⊗ ℒ^*) - g(X) + 1 = 3.  && \text{Riemann-Roch} \nonumber
\end{align}
In summary, we find that
\begin{equation}\label{eq:2}
  q(X) < q(X,0,γ).
\end{equation}

\subsubsection*{Step 2: An upper bound for $q(X,0,γ)$}
\approvals{Erwan & yes \\ Stefan & yes}

Recall from \cite[Obs.~3.9]{orbiAlb1} that there exists a natural inclusion
\[
  H⁰ \bigl( \what{X},\, Ω¹_{(X,0,γ)} \bigr) = H⁰ \bigl( \what{X},\, γ^* ω_X \bigr) ⊂ H⁰ \bigl( \what{X},\, ω_{\what{X}} \bigr).
\]
We claim that this inclusion is strict.  For this, it suffices to recall that
$ω_{\what X}$ is basepoint free, while all differentials coming from $X$ will
vanish at each of the two ramification points.  In summary, we find that
\begin{equation}\label{eq:A}
  q(X,0,γ) < q(\what{X}).
\end{equation}

\subsubsection*{Step 3: Summary of Setup}
\approvals{Erwan & yes \\ Stefan & yes}

Choose a point $\what{x} ∈ \what{X}$ and use
\[
  0_{\Alb(\what X)} := \alb(\what{X})(\what{x}),
  \quad
  0_{\Alb(X,0,γ)} := \alb(X,0,γ)(\what{x})
  \quad\text{and}\quad
  0_{\Alb(X)} := \alb(X)\bigl(γ (\what{x})\bigr)
\]
to equip $\Alb(\what{X}, 0)$, $\Alb(X, 0, γ)$ and $\Alb(X,0)$ with Lie group
structures and natural Lie group isomorphisms,
\begin{equation}\label{eq:A-8-4}
  \Alb(\what X,0) ≅ \operatorname{Jac}(\what X)
  \quad\text{and}\quad
  \Alb(X,0) ≅ \operatorname{Jac}(X).
\end{equation}
We construct a diagram as follows,
\[
  \begin{tikzcd}[column sep=1.4cm, row sep=1cm]
    & P_γ \ar[d, hook, "ι_P"] \ar[r, two heads, "q^\top"] & Q_γ \ar[d, hook, "ι_Q"] \\
    \what{X} \ar[r, "\alb(\what{X}{,} 0)"] \ar[d, two heads, "γ"'] & \Alb(\what{X}, 0) \ar[d, two heads, "\alb(γ)"] \ar[r, two heads, "q^\bot"] & \Alb(X, 0, γ) \ar[d, two heads, "a"] \\
    X \ar[r, "\alb(X{,} 0)"'] & \Alb(X,0) \ar[r, phantom, "="] & \Alb(X,0).
  \end{tikzcd}
\]
\begin{itemize}
  \item The morphism $\alb(γ)$ is induced by the universal property of
  $\Alb(\what{X}, 0)$, given that $\alb(X{,} 0) ◦ γ$ is a morphism to a
  compact torus.  Our choice of Lie group structures on $\Alb(\what{X}, 0)$ and
  $\Alb(X,0)$ guarantees this is a Lie group morphism.  It is surjective because
  it contains $\alb(X,0)(X)$, which generates $\Alb(X,0)$ as a group.
  We set $P_γ := \ker \alb(γ)$ and let $ι_P$ be the inclusion.

  \item The morphism $a$ is induced by the universal property of $\Alb(X, 0,
  γ)$, given that $\alb(X{,} 0) ◦ γ$ is a morphism to a compact
  torus whose differential takes its image in $γ^* Ω¹_X =
  Ω¹_{(X,0,γ)}$.  Our choice of Lie group structures on $\Alb(\what{X},
  0)$ and $\Alb(X,0, γ)$ guarantees this is a Lie group morphism.  It is
  surjective because it contains $\alb(X,0)(X)$, which generates $\Alb(X,0)$ as
  a group.  We set $Q_γ := \ker a$ and let $ι_Q$ be the inclusion.

  \item The morphism $q^\bot$ is induced by the universal property of
  $\Alb(\what{X}, 0)$, given that $\alb(X{,} 0, γ)$ is a morphism to a
  compact torus.  Our choice of Lie group structures on $\Alb(\what{X}, 0)$ and
  $\Alb(X,0, γ)$ guarantees this is a Lie group morphism.  It is surjective
  because it contains $\alb(X,0, γ)(X)$, which generates $\Alb(X,0,
  γ)$ as a group.

  \item The morphism $q^\top$ is the induced surjection between the kernels.
\end{itemize}

\subsubsection*{Step 4: Conclusion}
\approvals{Erwan & yes \\ Stefan & yes}

Observe that the isomorphisms \eqref{eq:A-8-4} identify the morphism $\alb(γ)$
with the norm map $N(γ) : \operatorname{Jac}(\what X) → \operatorname{Jac}(X)$
between Jacobians.  Our setup therefore guarantees that $P_γ$ is isomorphic to
the Prym variety of $γ$, hence is a simple Abelian variety of dimension
$5$.

Inequality~\eqref{eq:A} guarantees that $\dim \Alb(X,0,γ) < \dim \Alb(\what{X},
0)$.  In particular, we find that $\dim Q_γ < \dim P_γ$.  But since $P_γ$ is
simple, this is possibly if and only if $Q_γ$ is trivial, so that $a$ is
isomorphic and
\[
  q_{\Alb}(X,0,γ) = \dim \Alb(X,0,γ) = \dim \Alb(X,0) = q(X).
\]
Inequality~\eqref{eq:2} then reads
\[
  q_{\Alb}(X,0,γ) < q(X,0,γ),
\]
which is what we wanted to show.  \qed

%
%
\svnid{$Id: 07-adaptedAlbaneseZuo.tex 940 2024-10-10 07:57:46Z kebekus $}
\selectlanguage{british}

\section{The Albanese for a subspace of differentials}
\subversionInfo
\approvals{Erwan & yes \\ Stefan & yes}

This section proves the existence of an Albanese of a cover as a special case of
the ``Albanese for a subspace of differentials''.  We refer the reader to
\cite[Sect.~4.2]{MR1738433} for a related construction in the smooth, proper
case.  Throughout the present section, we work in following setting.

\begin{setting}\label{set:7-1} %
  Let $(X,D)$ be a log pair where $X$ is compact.  Let $V ⊆ H⁰\Bigl(X,\,
  Ω^{[1]}_X(\log D)\Bigr)$ be a linear subspace.
\end{setting}

\begin{defn}[The Albanese for a subspace of differentials]\label{def:7-2}%
  Assume Setting~\ref{set:7-1}.  An Albanese of $(X,D,V)$ is a semitoric variety
  $\Alb(X,D,V)° ⊂ \Alb(X,D,V)$ together with a quasi-algebraic morphism
  \[
    \alb(X,D,V)° : X° → \Alb(X,D,V)°
  \]
  such that the following holds.
  \begin{enumerate}
  \item\label{il:7-2-1} The pull-back morphism of logarithmic differentials,
    \[
      \diff \alb(X,D,V) : H⁰\Bigl( Ω¹_{\Alb(X,D,V)}(\log Δ) \Bigr) → H⁰\Bigl( X,\, Ω^{[1]}_X(\log D) \Bigr)
    \]
    takes its image in $V$.

  \item\label{il:7-2-2} If $A° ⊂ A$ is any semitoric variety and if $a° : X° →
    A°$ is quasi-algebraic such that
    \[
      \diff a : H⁰\bigl( A,\, Ω¹_A(\log Δ) \bigr) → H⁰\Bigl( X,\, Ω^{[1]}_X(\log D) \Bigr)
    \]
    takes its image in $V$, then $a°$ factors uniquely as
    \[
      \begin{tikzcd}[column sep=2cm]
        X° \arrow[r, "\alb(X{,}D{,}V)°"'] \arrow[rr, "a°", bend left=10] & \Alb(X,D,V)° \arrow[r, "∃!b°"'] & A°,
      \end{tikzcd}
    \]
    where $b°$ is quasi-algebraic.
  \end{enumerate}
\end{defn}

\begin{warning}
  We do not claim or ask in Item~\ref{il:7-2-1} that the space $V$ is equal to
  the image of the morphism $\diff \alb(X,D,V)$.  See Section~\vref{sec:7-3} for
  a sobering example which shows that surjectivity is a delicate property of the
  subspace $V$.
\end{warning}

We will later consider Definition~\ref{def:7-2} in a setting where the space $V$
is of the form $V = H⁰\bigl( X,\, ℱ \bigr)$, for a subsheaf $ℱ ⊆ Ω¹_X(\log D)$.
The following notion will be used.

\begin{defn}[The Albanese for subsheaves of differentials]\label{def:7-4}%
  Assume Setting~\ref{set:7-1}.  If there exists a subsheaf $ℱ ⊆ Ω¹_X(\log D)$
  such that $V = H⁰\bigl( X,\, ℱ \bigr)$, then we denote the Albanese briefly as
  $\alb(X,D,ℱ)° : X° → \Alb(X,D,ℱ)°$.
\end{defn}

\subsection{Uniqueness and existence}
\label{sec:7-1}
\approvals{Erwan & yes \\ Stefan & yes}

As before, the universal property spelled out in Item~\ref{il:7-2-2} implies
that $\Alb(X,D,V)°$ is unique up to unique isomorphism and that $\Alb(X,D,V)$ is
bimeromorphically unique.  As before, we abuse notation and refer to any
Albanese as ``the Albanese'', with associated semitoric \emph{Albanese variety}
$\Alb(X,D,V)° ⊂ \Alb(X,D,V)$ and quasi-algebraic \emph{Albanese morphism}
$\alb(X,D,V)°$.

\begin{prop}\label{prop:7-5}%
  Assume Setting~\ref{set:7-1}.  If $X$ is Kähler, then an Albanese of $(X,D,V)$
  exists.  The dimension is bounded by
  \begin{equation}\label{eq:7-5-1}
    \dim \Alb(X,D,V)° ≤ \dim_ℂ V.
  \end{equation}
  If $x ∈ X°$ is any point and if we use
  \[
    0_{\Alb(X,D,V)°} := \alb(X,D,V)° (x) ∈ \Alb(X,D,V)°
  \]
  to equip $\Alb(X,D,V)°$ with the structure of a Lie group, then the image of
  $\alb(X,D,V)°$ generates $\Alb(X,D,V)°$ as an Abelian group.
\end{prop}

Example~\vref{ex:7-8} shows that Inequality~\eqref{eq:7-5-1} might be strict.  As
in Section~\ref{sec:4-2}, we give a direct construction of one Albanese.

\begin{construction}[Construction of the Albanese for a subspace of differentials]\label{cons:7-6}%
  In Setting~\ref{set:7-1}, consider the annihilator $V^\perp ⊆ H⁰\bigl( X,\,
  Ω^{[1]}_X(\log D) \bigr)^*$ and recall from Observation~\ref{obs:4-15} that
  the construction of $\Alb(X,D)°$ equips us with a canonical holomorphic Lie
  group morphism
  \begin{equation}\label{eq:7-6-1}
    H⁰\bigl( X,\, Ω^{[1]}_X(\log D) \bigr)^* \twoheadrightarrow \Alb(X,D)°.
  \end{equation}
  The image
  \[
    I_V := \img\bigl(V^\perp → \Alb(X,D)°\bigr)
  \]
  is then a subgroup of $\Alb(X,D)$ that may or may not be closed.  Either way,
  Fact~\vref{fact:3-26} allows taking the smallest quasi-algebraic subgroup $B ⊆
  \Alb(X,D)°$ that contains $I_V$.  We write
  \[
    \Alb(X,D,V)° := \factor{\Alb(X,D)°}{B}
  \]
  and obtain morphisms
  \[
    \begin{tikzcd}[column sep=2cm]
      X° \ar[r, "\alb(X{,}D)°"'] \ar[rr, "\alb(X{,}D{,}V)°", bend left=10] & \Alb(X,D)° \ar[r, two heads, "q°\text{, quotient}"'] &
      \Alb(X,D,V)°.
    \end{tikzcd}
  \]
  Recall from Facts~\ref{fact:3-23} and \ref{fact:3-27} that $B$ and
  $\Alb(X,D,V)$ are isomorphic to semitori, and that there exists a semitoric
  compactification $\Alb(X,D,V)° ⊆ \Alb(X,D,V)$ that renders the quotient
  morphism $q°$ quasi-algebraic.  With this choice of compactification,
  Lemma~\ref{lem:2-4} guarantees that the morphism $\alb(X,D,V)°$ is
  quasi-algebraic, as desired.
\end{construction}

\begin{proof}[Proof of Proposition~\ref*{prop:7-5}]
  We need to verify that Construction~\ref{cons:7-6} satisfies the properties
  spelled out in Proposition~\ref{prop:5-4}.  Once this is done,
  Proposition~\vref{prop:4-10} and surjectivity of the quotient morphism $q$
  guarantees that the image of $\alb(X,D,V)°$ generates $\Alb(X,D,V)°$ as an
  Abelian group, as claimed.
  
  \subsubsection*{Property~\ref*{il:7-2-1}}

  To prove Property~\ref{il:7-2-1}, write $W := \img \diff \alb(X,D,V)$ and
  recall that
  \[
    \img\bigl(V^\perp → \Alb(X,D)° \bigr) \overset{\text{constr.}}{⊆} \ker(q°)
    \overset{\text{Fact~\ref{fact:4-19}}}{=} \img\bigl(W^\perp → \Alb(X,D)°
    \bigr).
  \]
  Given that the Lie group morphism \eqref{eq:7-6-1} has maximal rank, we find
  that $V^\perp ⊆ W^\perp$ and hence that $V ⊇ W$, as desired.

  \subsubsection*{Property~\ref*{il:7-2-2}}
  
  Assume that a morphism $a°: X° → A°$ as in Property~\ref{il:7-2-2} is given.
  The universal property of $\Alb(X,D)$ will then yield a factorization
  \[
    \begin{tikzcd}[column sep=2.5cm]
      X° \arrow[r, "\alb(X{,}D)°"'] \arrow[rr, "a°", bend left=15] &
      \Alb(X,D)° \arrow[r, "β°\text{, quasi-algebraic}"'] & A°.
    \end{tikzcd}
  \]
  We claim that the quasi-algebraic morphism $β°$ factors via $q°$,
  \[
    \begin{tikzcd}[column sep=2.8cm]
      \Alb(X,D)° \arrow[r, two heads, "q°\text{, quasi-algebraic}"'] \arrow[rr, "β°\text{, quasi-algebraic}", bend left=15] & \factor{\Alb(X,D)°}{B} \arrow[r, "∃!b°"'] & A°.
    \end{tikzcd}
  \]
  Equivalently, we claim that $B ⊆ \ker(β°)$.  This follows easily: writing
  \[
    W := \img\Bigl(\diff a : H⁰\bigl( A,\, Ω¹_A(\log Δ) \bigr) → H⁰\bigl(
      X,\, Ω^{[1]}_X (\log D) \bigr)\Bigr),
  \]
  we know by assumption that $W ⊆ V$ or equivalently, that $W^\perp ⊇ V^\perp$.
  By Fact~\vref{fact:4-19}, this is in turn equivalent to $\ker(β) ⊇ I_V$.  The
  desired inclusion $\ker(β°) ⊇ B$ follows as soon as we recall from
  Fact~\vref{fact:3-28} that $\ker(β°)$ is quasi-algebraic.  Lemma~\vref{lem:2-4}
  guarantees that $b°$ is quasi-algebraic, as required.  The statement about the
  dimension is clear from the construction.
\end{proof}

\subsection{Proof of Proposition~\ref*{prop:5-4}}
\approvals{Erwan & yes \\ Stefan & yes}
\label{sec:7-2}

In the setting of Proposition~\ref{prop:5-4}, set
\[
  V := H⁰\bigl( \what{X},\, Ω^{[1]}_{(X,D,γ)} \bigr).
\]
Using the notation introduced in Definition~\ref{def:7-4},
Proposition~\ref{prop:7-5} equips us with a semitoric variety
\[
  \Alb\left(\what{X}, \what{D}, Ω^{[1]}_{(X,D,γ)}\right)° ⊂ \Alb\left(\what{X}, \what{D}, Ω^{[1]}_{(X,D,γ)}\right)
\]
and a quasi-algebraic morphism
\[
  \alb\left(\what{X}, \what{D}, Ω^{[1]}_{(X,D,γ)}\right)° \::\: \what{X}° → \Alb\left(\what{X}, \what{D}, Ω^{[1]}_{(X,D,γ)}\right)°
\]
that we take as the Albanese of the cover $γ$ of the $\cC$-pair $(X,D)$.  A
comparison of the Properties~\ref{il:7-2-1}--\ref{il:7-2-2} guaranteed by
Proposition~\ref{prop:7-5} with the Properties~\ref{il:5-2-1}--\ref{il:5-2-2}
required by Proposition~\ref{prop:5-4} concludes the proof.  \qed

\subsection{Examples}
\label{sec:7-3}
\approvals{Erwan & yes \\ Stefan & yes}

We end the present section with two simple examples.

\begin{example}
  In the setting of Definition~\ref{def:7-2}, if $V = \{0\}$, then $\Alb(X,D,V)$
  is a point.
\end{example}

\begin{example}\label{ex:7-8}%
  Let $E$ be an elliptic curve.  Set $X = E⨯E$ and take $D := 0 ∈ \Div(X)$.
  Pulling back differentials from the two factors gives natural morphisms
  \[
    \diff π_i : H⁰ \bigl( E,\, Ω¹_E \bigr) → H⁰ \bigl( X,\, Ω¹_X \bigr).
  \]
  Choose a number $τ ∈ ℂ$ and set $V := \img\bigl( (\diff π_1) + τ·(\diff π_2)
  \bigr)$, which is a one-dimensional linear subspace of $H⁰\bigl( X,\, Ω¹_X
  \bigr)$.  The following will hold.
  \begin{itemize}
  \item If $τ$ is non-rational, then $I_V$ is dense in $\Alb(X,0)°$ and
    $\Alb(X,0,V)° = \{ 0\}$.
    
  \item Towards the other extreme, if $τ = 0$, then $\Alb(X,0,V)° = \Alb(E)$.
  \end{itemize}
\end{example}

%
%
\svnid{$Id: 08-boundedIrregularity.tex 942 2024-10-10 16:13:56Z kebekus $}
\selectlanguage{british}

\section{Boundedness of the Albanese irregularity for special pairs}
\subversionInfo
\approvals{Erwan & yes \\ Stefan & yes}
\label{sec:8}

Following Ueno's work \cite{Ueno75}, Campana has remarked in
\cite[Sect.~5.2]{Cam04} that the Albanese morphism of a special manifold is
always surjective.  We extend Campana's observation to the Albanese of a cover.
For $\cC$-pairs that are special in the sense of \cite[Def.~6.11]{orbiAlb1}, the
following theorem implies that the dimension of the Albanese is bounded by the
dimension of $X$.  In particular, it cannot go to infinity as we consider higher
and higher covers.  Along these lines, we view the theorem as a boundedness
result.

\begin{thm}[$\cC$-pairs whose Albanese morphism is not dominant]\label{thm:8-1}
  In Setting~\ref{set:5-1}, assume that $X$ is Kähler.  If the Albanese morphism
  $\alb(X,D,γ)°$ is \emph{not} dominant, then there exists a number $1 ≤ p ≤
  \dim X$ and a coherent rank-one subsheaf $ℒ_1 ⊂ Ω^{[p]}_{(X,D,\Id_X)}$ with
  $\cC$-Kodaira-Iitaka dimension $κ_{\cC}(ℒ_1) ≥ p$.
\end{thm}

We refer the reader to \cite[Sect.~6]{orbiAlb1} for the definition of
``$\cC$-Kodaira-Iitaka dimension'' and for the related notions of ``special
pairs'' and ``Bogomolov sheaves''.

\begin{cor}[The Albanese for covers for special pairs]\label{cor:8-2} %
  In Setting~\ref{set:5-1}, assume that $X$ is Kähler.  If $(X,D)$ is special,
  then the Albanese morphism $\alb(X,D,γ)°$ is dominant.  \qed
\end{cor}

The proof of Theorem~\ref{thm:8-1} is given in Section~\ref{sec:8-2}, starting
from Page~\pageref{sec:8-2} below.

\begin{rem}
  Recall from Definition~\ref{def:5-2} that the Albanese morphism $\alb(X,D,γ)°$
  is quasi-algebraic, so that topological closure of its image,
  \[
    \overline{\img \alb(X,D,γ)°} ⊆ \Alb(X,D,γ)°,
  \]
  is always analytic.  The word ``dominant'' in Theorem~\ref{thm:8-1} and
  Corollary~\ref{cor:8-2} is therefore meaningful.
\end{rem}

\begin{rem}\label{rem:8-4}%
  Assume Setting~\ref{set:5-1}.  If the $\cC$-pair $(X,D)$ is special,
  Corollary~\ref{cor:8-2} implies in particular that $q⁺_{\Alb}(X,D,γ) ≤ \dim
  X$.
\end{rem}

Even for special pairs, one cannot expect that the Albanese morphism
$\alb(X,D,γ)°$ is surjective.  The following simple example shows what can go
wrong.

\begin{example}[Failure of surjectivity]
  Let $T$ be a compact torus and let $t ∈ T$ be any point.  Let $X$ be the
  blow-up of $T$ in $t$ and let $D ∈ \Div(X)$ be the exceptional divisor, with
  multiplicity one.  Then, the logarithmic pair $(X,D)$ is special, the Albanese
  for the identity morphism equals $\Alb(X,D,\Id_X)° = T$ and
  \[
    \img \alb(X,D,γ)° = T ∖ \{ t \}.
  \]
\end{example}

\subsection{Failure of dominance}
\approvals{Erwan & yes \\ Stefan & yes}

To prepare for the proof of Theorem~\ref{thm:8-1}, we analyse the setting where
the Albanese of a cover fails to be dominant.  The construction presented here
will also be used in the forthcoming paper \cite{orbiAlb3}, where we prove a
$\cC$-version of the Bloch-Ochiai theorem.

\begin{setting}[Failure of dominance]\label{set:8-6}%
  In Setting~\ref{set:5-1}, assume that $X$ is Kähler.  Assume that the cover
  $γ$ is Galois with group $G$, and use Corollary~\vref{cor:3-21} to choose an
  Albanese
  \[
    \Alb(X,D,γ)° ⊂ \Alb(X,D,γ), %
    \quad\text{written in short as } %
    \Alb° ⊂ \Alb,
  \]
  such that the $G$-action on $\Alb°$ extends to $\Alb$.  Recall that the
  Albanese morphism $\alb°$ is quasi-algebraic.  The topological closure of the
  image, $א := \overline{\img \alb°}$, is thus an analytic subset of $\Alb$.
  Set $א° := א ∩ \Alb°$ and assume that $א°$ is a proper subset, $א° ⊊ \Alb°$.
  Finally, choose an element $\what{x} ∈ \what{X}°$ and use its image point
  \[
    0_{\Alb°} := \alb°(\what{x}) ∈ \Alb°
  \]
  to equip $\Alb°$ with the structure of a Lie group.
\end{setting}

\begin{rem}[Stabilizer groups]
  In Setting~\ref{set:8-6}, recall from \cite[Prop.~5.3.16]{MR3156076} that the
  stabilizer subgroup
  \[
    \operatorname{St}_{\Alb°}(א°) = \bigl\{ a ∈ \Alb° \,|\, a+ א° = א° \bigr\} ⊂ \Alb°
  \]
  is closed and quasi-algebraic.  Recall from \cite[Prop.~5.3.13]{MR3156076}
  that its maximal connected subgroup $I ⊂ \operatorname{St}_{\Alb°}(א°)$ is
  then a semitorus.
\end{rem}

\begin{obs}[Properness of $I$ as a subgroup of $\Alb°$]\label{obs:8-8}%
  By construction, we have
  \[
    0_{\Alb°} = \alb°(\what{x}) ∈ \img \alb° ⊆ א°.
  \]
  It follows that $\operatorname{St}_{\Alb°}(א°) ⊆ א°$.  This equips us with
  inclusions
  \[
    I ⊆ \operatorname{St}_{\Alb°}(א°) ⊆ א° ⊊ \Alb°
  \]
  and shows that $I ⊊ \Alb°$ is a proper subgroup.  The quotient group $\Alb°/I$
  is not trivial.
\end{obs}

We have seen in Observation~\vref{obs:5-6} that the morphism $\alb°$ is
equivariant with respect to the $G$-action on $\Alb°$.  The action will then
stabilize the subset $א°$.  As the next lemma shows, it will also stabilize
$\operatorname{St}_{\Alb°}(א°)$ and $I$, at least up to translation.

\begin{lem}[Relation between $G$ and $I$]\label{lem:8-9}%
  In Setting~\ref{set:8-6}, if $g ∈ G$ is any element, then $g·I$ is a translate
  of $I$.  In particular, the $G$-action of $\Alb°$ maps $I$-orbits to
  $I$-orbits, for the additive action of $I$ on $\Alb°$.
\end{lem}
\begin{proof}
  Since all connected components of the group $\operatorname{St}_{\Alb°}(א°)$
  are translates of the identity component, it suffices to show that
  $g·\operatorname{St}_{\Alb°}(א°)$ is a translate of
  $\operatorname{St}_{\Alb°}(א°)$.  To this end, recall from
  Proposition~\vref{prop:3-18} that we may write $g : \Alb° → \Alb°$ in the form
  $g : a ↦ f°(a) + g(0_{\Alb°})$, where $f°: \Alb° → \Alb°$ is a group morphism.
  In particular, we find that
  \begin{equation}\label{eq:8-9-1}
    א° = g(א°) = f°(א°) + g(0_{\Alb°})
    \quad⇔\quad
    f°(א°) = א° - g(0_{\Alb°}).
  \end{equation}
  This gives
  \begin{align*}
    g\bigl(\operatorname{St}_{\Alb°}(א°)\bigr) & = f°\bigl(\operatorname{St}_{\Alb°}(א°)\bigr) + g(0_{\Alb°}) \\
    & = \operatorname{St}_{\Alb°}\bigl(f°(א°)\bigr) + g(0_{\Alb°}) && f°\text{ a group morphism} \\
    & = \operatorname{St}_{\Alb°}\bigl(א°-g(0_{\Alb°})\bigr) + g(0_{\Alb°}) && \text{\eqref{eq:8-9-1}} \\
    & = \operatorname{St}_{\Alb°}(א°) + g(0_{\Alb°}) && \text{Defn.\ of }\operatorname{St}_{\Alb°}(•) \qedhere
  \end{align*}
\end{proof}

\begin{construction}\label{cons:8-10}%
  Maintaining Setting~\ref{set:8-6}, we construct a non-trivial semitoric
  variety $B° ⊂ B$ with $G$-action and a diagram
  \[
    \begin{tikzcd}[column sep=2cm, row sep=1cm]
      \what{X}° \ar[r, "\alb°"'] \ar[d, two heads, "γ°\text{, quotient by }G"'] \ar[rr, bend left=15, "b°"] & \Alb° \ar[r, two heads, "β°\text{, quotient by }I"'] \ar[d, two heads, "γ_{\Alb°}\text{, quotient by }G"] & B° \ar[d, two heads, "γ_{B°}\text{, quotient by }G"] \\
      X° \ar[r, "δ°"'] & \factor{\Alb°}{G} \ar[r, two heads, "ε°"'] & \factor{B°}{G}
    \end{tikzcd}
  \]
  where (among other things) the following holds.
  \begin{itemize}
    \item All horizontal arrows are quasi-algebraic,

    \item all arrows in the top row are $G$-equivariant, and

    \item all arrows in the bottom row are $\cC$-morphisms for the $\cC$-pairs
      \[
        (X°, D°),\quad \factor{\bigl( \Alb°, 0\bigr)}{G}, \quad\text{and}\quad \factor{\bigl(B°, 0\bigr)}{G}.
      \]
  \end{itemize}
  The left rectangle of the diagram is given by Proposition~\vref{prop:5-10}.  As
  for the right rectangle, take $B°$ as the quotient $\Alb°/I$.  Recall from
  \cite[Thm.~5.3.13]{MR3156076} that $B°$ is a semitorus, and that there exists
  a semitoric compactification $B° ⊆ B$ that renders the quotient morphism $β°$
  quasi-algebraic.  Lemma~\ref{lem:8-9} gives a natural action $G ↺ B°$ that
  makes the morphism $β°$ equivariant, and Corollary~\vref{cor:3-21} allows
  assuming without loss of generality that the $G$ action extends from $B°$ to
  $B$.  The right rectangle of the diagram is now given by the universal
  property of $G$-quotients, \cite[Prop.~12.7]{orbiAlb1}.  Finish the
  construction by recalling from \cite[Prop.~12.7]{orbiAlb1} that $ε°$ is a
  morphism of $\cC$-pairs, from $\bigl(\Alb°, 0\bigr)/G$ to $\bigl(B°,
  0\bigr)/G$, as required.

  To conclude Construction~\ref{cons:8-10}, consider the topological closure $Z
  := \overline{\img β°}$, which is an analytic subset of $B$.  As before, write
  $Z° := Z ∩ B°$ and set $p := \dim Z$.
\end{construction}

The following observations summarize the main properties of the construction.

\begin{obs}\label{obs:8-11}%
  By construction, $Z°$ is not invariant under the action of any proper
  semitorus in $B°$.  In this setting, recall from Kawamata's proof of the Bloch
  conjecture, \cite{Kawa80}, or more specifically from
  \cite[Cor.~3.8.27]{KobayashiGrundlehren} that there exist $B°$-invariant
  differentials $τ°_0$, …, $τ°_p ∈ H⁰ \bigl( B°,\, Ω^p_{B°} \bigr)$ such that
  the restrictions $τ°_•|_{Z°_{\reg}}$ are linearly independent
  top-differentials on $Z°_{\reg}$, and therefore define a $(p+1)$-dimensional
  linear subspace
  \[
    V := \bigl\langle τ°_0|_{Z°_{\reg}}, …, τ°_p|_{Z°_{\reg}} \bigr\rangle ⊆ H⁰ \bigl(
    Z°_{\reg}, Ω^p_{Z°_{\reg}} \bigr) = H⁰ \bigl( Z°_{\reg}, ω_{Z°_{\reg}} \bigr).
  \]
  The associated meromorphic map $φ_V : Z°_{\reg} \dasharrow ℙ^p$ is generically
  finite.  Recall from Item~\ref{il:3-15-2} of Proposition~\vref{prop:3-15} that
  the $B°$-invariant differentials $τ°_• ∈ H⁰ \bigl( B°,\, Ω^p_{B°} \bigr)$
  automatically extend to differentials with logarithmic poles at infinity, say
  $τ_• ∈ H⁰ \bigl(B,\, Ω^p_B(\log Δ) \bigr)$.
\end{obs}

\begin{obs}\label{obs:8-12}%
  We have observed in \ref{obs:8-8} that $I ⊆ א°$.  There is more that we can
  say.  The assumption $א° ⊊ \Alb°$ and Item~\ref{il:5-2-2} of
  Definition~\ref{def:5-2} imply that $א°$ is not itself a semitorus.  In
  particular, we find that $I ⊊ א°$ is a proper subset and that the variety $Z°$
  is therefore positive-dimensional.  The inclusion $I ⊂ א°$ also implies that
  the morphisms
  \[
    β° : \Alb° \twoheadrightarrow B° %
    \quad\text{and}\quad %
    β°|_{א°} : א° → Z°
  \]
  are $G$-equivariant fibre bundles, both with typical fibre $I$.  The analytic
  variety $Z°$ is therefore a proper subset, $Z° ⊊ B°$.
\end{obs}

\subsection{Proof of Theorem~\ref*{thm:8-1}}
\approvals{Erwan & yes \\ Stefan & yes}
\label{sec:8-2}

We prove Theorem~\ref{thm:8-1} in the remainder of the present
Section~\ref{sec:8} and maintain Setting~\ref{set:5-1} throughout.  For
simplicity of notation, we prove the contrapositive: assuming that the Albanese
morphism $\alb(X,D,γ)°$ is \emph{not} dominant, we show that the $\cC$-pair
$(X,D)$ admits a Bogomolov sheaf and is hence \emph{not} special.

The proof follows classic arguments, with some additional complications because
of our use of adapted differentials and because of the singularities of the
varieties involved.

\subsection*{Step 1: Simplification}
\approvals{Erwan & yes \\ Stefan & yes}

Recall Lemma~\ref{lem:5-5}: Non-dominance of $\alb(X,D,γ)°$ is preserved when we
replace $γ$ by any cover that factors via $γ$.  We can therefore pass to the
Galois closure and assume that we are in Setting~\ref{set:8-6}.  We use the
notation introduced in Construction~\ref{cons:8-10} and
Observations~\ref{obs:8-11}--\ref{obs:8-12} in the remainder of the proof.

\subsection*{Step 2: A rank-one sheaf in $Ω^{[p]}_{(X,D,γ)}\bigl|_{\what{X}°}$}
\approvals{Erwan & yes \\ Stefan & yes}
\CounterStep{}

Consider the composed morphism of $G$-sheaves
\begin{equation}\label{eq:8-13-1}
  \begin{tikzcd}
    (b°)^* \: Ω^p_{B°} \ar[r, "\diff b°"] & Ω^p_{\what{X}°} \ar[r] &
    Ω^{[p]}_{\what{X}°},
  \end{tikzcd}
\end{equation}
and let $ℒ° ⊆ Ω^{[p]}_{\what{X}°}$ denote the image sheaf, which is then a
torsion free $G$-subsheaf of $Ω^{[p]}_{\what{X}}$.  We summarize its main
properties.

\begin{obs}\label{obs:8-14}%
  The sheaf $ℒ°$ is of rank one because $b°$ factors via the $p$-dimensional
  space $Z°$.  \qed~\mbox{(Observation~\ref{obs:8-14})}
\end{obs}

\begin{claim}\label{claim:8-15} %
  The sheaf $ℒ°$ is contained in the subsheaf
  $Ω^{[p]}_{(X,D,γ)}\bigr|_{\what{X}°} ⊆ Ω^{[p]}_{\what{X}°}$.
\end{claim}
\begin{proof}[Proof of Claim~\ref{claim:8-15}]
  The morphism $b°$ factors via $\alb°$.  Since pull-back of Kähler
  differentials is functorial, $\diff b°$ factors via $\diff \alb°$ and the
  image of the composed morphism \eqref{eq:8-13-1} is contained in the image of
  the composition
  \[
    \begin{tikzcd}
      (\alb°)^* \: Ω^p_{\Alb°} \ar[r, "\diff \alb°"] & Ω^p_{\what{X}°}
      \ar[r] & Ω^{[p]}_{\what{X}°}.
    \end{tikzcd}
  \]
  But then Remark~\ref{rem:5-3} gives the claim.
  \qedhere~\mbox{(Claim~\ref{claim:8-15})}
\end{proof}

\subsection*{Step 3: A rank-one sheaf in $Ω^{[p]}_{(X,D,γ)}$}
\approvals{Erwan & yes \\ Stefan & yes}

We extend the sheaf $ℒ°$ from $\what{X}°$ to a rank-one, reflexive sheaf that is
defined on all of $\what{X}$.  As in Section~\ref{sec:4}, the reader coming from
algebraic geometry might find the proof surprisingly complicated: in the
analytic setting, it is typically not possible to extend coherent sheaves across
codimension-two subsets.

\begin{claim}\label{claim:8-16}%
  There exists a rank-one, reflexive $G$-subsheaf $ℒ ⊆ Ω^{[p]}_{(X,D,γ)}$ whose
  restriction to $\what{X}°$ contains $ℒ°$.  There are sections $σ_0, …, σ_p ∈
  H⁰ \bigl( \what{X},\, ℒ \bigr)$ whose associated linear system defines a
  dominant meromorphic map $\what{X} \dasharrow ℙ^p$.
\end{claim}
\begin{proof}[Proof of Claim~\ref{claim:8-16}]
  The morphism $b° : \what{X}° → B°$ is quasi-algebraic and therefore extends to
  a $G$-equivariant meromorphic map $b : \what{X} \dasharrow B$.  Choose a
  $G$-equivariant log-resolution $(\wtilde{X}, \wtilde{D})$ of $(\what{X},
  \what{D})$ and the meromorphic map $b$ as follows:
  \[
    \begin{tikzcd}[column sep=large, row sep=1cm]
      \wtilde{X} \ar[d, two heads, "ρ\text{, resolution}"'] \ar[rd, bend left=10, "\wtilde{b}"] \\
      \what{X} \ar[r, dashed,"b"'] & B.
    \end{tikzcd}
  \]
  We can then consider $G$-subsheaves
  \[
    \img \Bigl( \diff \wtilde{b} : Ω^p_B(\log Δ) → Ω^p_{\wtilde{X}}(\log \wtilde{D}) \Bigr) ⊆ Ω^p_{\wtilde{X}}(\log \wtilde{D})
  \]
  and
  \[
    ℒ' := ρ_* \img \bigl( \diff \wtilde{b} \bigr) ⊆ ρ_* Ω^p_{\wtilde{X}}(\log \wtilde{D}) ⊆ Ω^{[p]}_{\what{X}}(\log \what{D})
  \]
  The construction guarantees that the sheaves $ℒ'$ and $ℒ°$ agree over the open
  set $\what{X}°$; in particular, we find that $ℒ'$ is of rank one.  Together
  with Claim~\ref{claim:8-15}, the construction shows that $ℒ'$ is contained in
  $Ω^{[p]}_{(X,D,γ)}$.  Finally, let $ℒ$ be the saturation of $ℒ'$ in
  $Ω^{[p]}_{(X,D,γ)}$.  The sheaf $ℒ$ is then automatically reflexive.  In
  summary, we obtain inclusions of $G$-sheaves as follows,
  \[
    ℒ' ⊆ ℒ ⊆ Ω^{[p]}_{(X,D,γ)} ⊆ Ω^{[p]}_{\what{X}}(\log \what{D}).
  \]
  In order to construct the sections $σ_•$, recall from
  Observation~\ref{obs:8-11} that the differentials $τ_•$ have logarithmic poles
  at infinity, and then so do their pull-backs.  To be more precise, consider
  the reflexive differentials
  \[
    σ_• ∈ H⁰ \bigl( \what{X},\, Ω^{[p]}_{\what{X}}(\log \what{D})\bigr)
  \]
  that generically agree with the pull-back of $τ_•$, and therefore restrict to
  sections
  \[
    σ_•|_{\what{X}°} ∈ H⁰\bigl( \what{X}°,\, ℒ° \bigr) ⊂ H⁰\bigl( \what{X}°,\,
    Ω^{[p]}_{(X,D,γ)} \bigr).
  \]
  But that already implies that the $σ_•$ are sections of $ℒ$.
  \qedhere~\mbox{(Claim~\ref{claim:8-16})}
\end{proof}

\subsection*{Step 4: Conclusion}
\approvals{Erwan & yes \\ Stefan & yes}

Given any number $i ∈ ℕ$, recall from \cite[Obs.~4.12]{orbiAlb1} that reflexive
symmetric multiplication of adapted reflexive tensors yields inclusions
\[
  ℒ^{[⊗i]} ⊆ \Sym^{[i]}_{\cC} Ω^{[p]}_{(X,D,γ)}.
\]
We consider the $G$-invariant push-forward sheaves,
\[
  ℒ_i := \left(γ_* ℒ^{[⊗ i]} \right)^G
  ⊆ \left(γ_* \Sym^{[i]}_{\cC} Ω^{[p]}_{(X,D,γ)} \right)^G
  \overset{\text{\cite[Cor.~4.20]{orbiAlb1}}}{=} \Sym^{[i]}_{\cC} Ω^{[p]}_{(X,D,\Id_X)}.
\]
Recall from \cite[Lem.~A.4]{GKKP11} that the sheaves $ℒ_i$ are reflexive.  By
construction, their rank is one.  We will show in this step that $κ_{\cC}(ℒ_1) ≥
p$.

\begin{obs}
  If $i ∈ ℕ$ is any number, then $ℒ_i$ equals the $i^{\text{th}}$ $\cC$-product
  sheaf
  \[
    ℒ_i = \Sym^{[i]}_{\cC} ℒ_1,
  \]
  as introduced in \cite[Def.~6.5]{orbiAlb1} \qed
\end{obs}

Recalling the definition of the $\cC$-Kodaira-Iitaka dimension from
\cite[Sect.~6.2]{orbiAlb1}, it remains to find one sheaf $ℒ_i$ with non-empty
linear system whose associated meromorphic map has an image of dimension $≥ p$.
For this, consider the linear systems
\[
  W_i := H⁰ \left( \what{X}, ℒ^{[⊗i]} \right)^G ⊆ H⁰ \left( \what{X}, ℒ^{[⊗i]} \right).
\]
If $i$ is sufficiently large and divisible, then $W_i$ is positive-dimensional
and the associated meromorphic map $φ_W : \what{X} \dasharrow ℙ^•$ has an image
of dimension
\[
  \dim \img φ_W ≥ \dim \img (φ_V◦b°) ≥ p.
\]
By construction, the meromorphic map $φ_W$ is constant on $G$-orbits and the
induced meromorphic map $φ : X \dasharrow ℙ^•$ equals the meromorphic map
associated with the reflexive sheaf $ℒ_i$.  We have seen above that this
finishes the proof of Theorem~\ref{thm:8-1}.  \qed


\phantomsection\addcontentsline{toc}{part}{Applications}

%
%
\svnid{$Id: 09-C-semitoric.tex 940 2024-10-10 07:57:46Z kebekus $}
\selectlanguage{british}

\section{\texorpdfstring{$\cC$}{𝒞}-semitoric pairs}
\label{sec:9}
\subversionInfo
\approvals{Erwan & yes \\ Stefan & yes}

We argue that quotients of semitoric varieties should be seen as
$\cC$-analogoues of the tori and semitoric varieties that appear in the classic
Albanese construction.  Before stating our main result on the existence of an
Albanese for a $\cC$-pair, we define and discuss the relevant notion precisely.

\begin{defn}[$\cC$-semitoric pairs] %
  A \emph{$\cC$-semitoric pair} is a $\cC$-pair $(X,D)$ such that there exists a
  semitoric variety $A° ⊂ A$, a finite group $G$ acting on $(A, Δ_A)$, and a
  $\cC$-isomorphism of the form
  \begin{equation}\label{eq:9-1-1}
    (X, D) ≅ \factor{\bigl( A, Δ_A \bigr)}{G}
  \end{equation}
  An isomorphism as in \eqref{eq:9-1-1} is called a \emph{presentation} of the
  $\cC$-semitoric pair.
\end{defn}

The choice of a presentation is not part of the data that defines a
$\cC$-semitoric pair.

\begin{rem}
  If a $\cC$-pair $(X,D)$ is $\cC$-semitoric, then it is $(X,D)$ is locally
  uniformizable, \cite[Def.~2.32]{orbiAlb1} and $X$ is compact Kähler,
  cf.~\cite[Prop.~5.3.5]{MR3156076}.
\end{rem}

\begin{example}
  The $\cC$-pair
  \[
    \textstyle\Bigl( ℙ¹, \frac{1}{2}·\{0\}+\frac{1}{2}·\{1\}+\frac{1}{2}·\{2\}+\frac{1}{2}·\{∞\} \Bigr)
  \]
  is $\cC$-semitoric.
\end{example}

It is perhaps not obvious from the outset that ``$\cC$-semitoric pair'' is a
meaningful notion.  In particular, it is probably not clear that morphisms
between the open parts of $\cC$-semitoric pairs have anything to do with the
groups that define the semitoric structures of domain and target.  Here, we
would like to make the point that \emph{quasi-algebraic} $\cC$-morphisms of
$\cC$-semitoric pairs do indeed come from group morphisms, and therefore respect
the structure in a meaningful way.  We see this as a strong indication that
$\cC$-semitoric pairs are relevant objects to consider.

\begin{thm}[Morphisms between $\cC$-semitoric pairs]\label{thm:9-4} %
  Let $(X_1, D_{X_1})$ and $(X_2, D_{X_2})$ be two $\cC$-semitoric pairs with
  presentations
  \[
    (X_1, D_{X_1}) ≅ \factor{\bigl( A_1, Δ_{A_1} \bigr)}{G_1}
    \quad\text{and}\quad
    (X_2, D_{X_2}) ≅ \factor{\bigl( A_2, Δ_{A_2} \bigr)}{G_2.}
  \]
  Given any quasi-algebraic $\cC$-morphism $φ° : (X°_1, D°_{X_1}) → (X°_2,
  D°_{X_2})$, there exists a semitoric variety $B° ⊂ B$ and a commutative
  diagram of the following form,
  \begin{equation}\label{eq:9-4-1}%
    \begin{tikzcd}[column sep=2.8cm]
      B° \ar[r, two heads, "ψ°\text{, quasi-algebraic}", "\text{étale cover}"'] \ar[d, "Φ°\text{, quasi-algebraic}"'] & A°_1 \ar[r, two heads, "q°_1\text{, quotient}"] & X°_1 \ar[d, "φ°"] \\
      A°_2 \ar[r, equal] & A°_2 \ar[r, two heads, "q°_2\text{, quotient}"'] & X°_2.
    \end{tikzcd}
  \end{equation}
\end{thm}

\begin{rem}[Quasi-algebraic maps between semitori]\label{rem:9-5}%
  Recall Proposition~\ref{prop:3-18}: if we choose points $0_{B°} ∈ B°$ and
  $0_{A°_2} ∈ A°_2$ to equip $B°$ and $A°_2$ with Lie group structures, then
  $Φ°$ can be written as a Lie group morphism composed with a translation.
\end{rem}

As an immediate corollary to Remark~\ref{rem:9-5}, we note that quasi-algebraic
morphisms of $\cC$-semitoric pairs enjoy many of the special properties known
for Lie group morphisms.  The following corollary lists a few of them.

\begin{cor}[Description of morphisms between $\cC$-semitoric pairs]\label{cor:9-6}%
  The following holds in the setting of Theorem~\ref{thm:9-4}.
  \begin{enumerate}
    \item The fibres of $φ°$ are of pure dimension.
    \item Any two non-empty fibres of $φ°$ are of the same dimension.
    \item If $φ°$ is quasi-finite, then it is finite.  \qed
  \end{enumerate}
\end{cor}

\subsection{Proof of Theorem~\ref*{thm:9-4}}
\approvals{Erwan & yes \\ Stefan & yes}

\CounterStep{}We maintain notation and assumptions of Theorem~\ref{thm:9-4} in
the present section.  To begin, choose a component
\[
  C° ⊆ \text{normalisation of } A°_1 ⨯_{X°_2} A°_2.
\]
The natural morphism $β° : C° \twoheadrightarrow A°_1$ is finite.  By the
analytic version of ``Zariski's main theorem in the form of Grothendieck'',
\cite[Thm.~3.4]{DethloffGrauert}, there exists a unique normal compactification
$C° ⊂ C$ where $β°$ extends to a finite morphism $β : C → A_1$.  An elementary
computation shows that the natural morphism $η° : C° → A°_2$ is quasi-algebraic
for this compactification, so that we obtain the following diagram,
\begin{equation}\label{eq:9-7-1}
  \begin{tikzcd}[column sep=2cm, row sep=1cm]
    C \ar[d, two heads, "β\text{, finite}"'] \ar[rrr, bend left=10, dashed, "η"] & C° \ar[r, "η°"'] \ar[d, two heads, "β°\text{, finite}"'] \ar[l, hook] & A°_2 \ar[d, equal] \ar[r, hook] & A_2 \ar[d, equal] \\
    A_1 \ar[d, two heads, "\text{quotient}"'] & A°_1 \ar[d, two heads, "q°_1\text{, quotient}"'] \ar[l, hook] & A°_2 \ar[d, two heads, "q°_2\text{, quotient}"] \ar[r, hook] & A_2 \ar[d, two heads, "\text{quotient}"] \\
    X_1 \ar[rrr, bend right=10, dashed, "φ"'] & X°_1 \ar[r, "φ°"] \ar[l, hook] & X°_2 \ar[r, hook] & X_2.
  \end{tikzcd}
\end{equation}

\subsection*{Step 1: Analysis of $β$}
\approvals{Erwan & yes \\ Stefan & yes}

The morphism $β$ is a cover for the logarithmic $\cC$-pair $(A_1, Δ_{A_1})$.
Recall from \cite[Obs.~3.16]{orbiAlb1} that the associated $\cC$-cotangent sheaf
equals
\begin{equation}\label{eq:9-7-2}
  Ω^{[1]}_{(A_1, Δ_{A_1}, β)} = β^* Ω¹_{A_1}\bigl(\log Δ_{A_1}\bigr).
\end{equation}
In particular, we find that the composed pull-back morphism
\[
  \diff β : H⁰\bigl( A_1,\, Ω¹_{A_1}(\log Δ_{A_1}) \bigr) → H⁰\bigl( C,\, Ω^{[1]}_C (\log Δ_C) \bigr)
\]
takes its image in $H⁰\bigl( C,\, Ω^{[1]}_{(A_1, Δ_{A_1}, β)} \bigr)$.  The
universal property of the adapted Albanese for the adapted cover $β$, as
specified in Item~\ref{il:5-2-2} of Definition~\ref{def:5-2}, will therefore
apply to give a factorization
\[
  \begin{tikzcd}[column sep=2.5cm]
    C° \ar[rr, two heads, bend left=10, "β°"] \ar[r, "\alb(A_1{,} Δ_{A_1}{,} β)°"'] & \Alb(A_1, Δ_{A_1}, β)° \ar[r, two heads, "ψ°"'] & A°_1,
  \end{tikzcd}
\]
where the morphisms $β°$ and $\alb(…)°$ are quasi-algebraic.  By
Lemma~\ref{lem:2-4}, then so is the morphism $ψ°$.  We claim that the surjection
$ψ°$ is also finite, and hence by Corollary~\ref{cor:3-19} an étale cover.
Equivalently, we claim that $\Alb(A_1, Δ_{A_1}, β)° ≤ \dim A°_1$.  But
\begin{align*}
  \dim \Alb(A_1, Δ_{A_1}, β)° & ≤ h⁰\bigl( C,\, Ω^{[1]}_{(A_1, Δ_{A_1}, β)} \bigr) & & \text{Proposition~\ref{prop:5-4}} \\
  & = h⁰\bigl( C,\, β^* Ω¹_{A_1}(\log Δ_{A_1}) \bigr) & & \text{\eqref{eq:9-7-2}} \\
  & = h⁰\bigl( C,\, 𝒪_C^{⊕ \dim A°_1} \bigr) = \dim A°_1 & & \text{Proposition~\ref{prop:3-15}}.
\end{align*}

\subsection*{Step 2: Analysis of $η$}
\approvals{Erwan & yes \\ Stefan & yes}

Recall \cite[Obs.~12.10]{orbiAlb1}, which implies that the morphisms $q°_•$ of
Diagram~\eqref{eq:9-7-1} are adapted for $(X°_•, D°_•)$ and that the
$\cC$-cotangent sheaves equal
\begin{equation}\label{eq:9-7-3}
  Ω^{[1]}_{(X°_•, D°_•, q°_•)} = Ω¹_{A°_•}.
\end{equation}
Along similar lines, \cite[Obs.~4.15]{orbiAlb1} implies that the morphism
$q°_1◦β°$ is adapted for the pair $(X°_1, D°_1)$, and that
\begin{equation}\label{eq:9-7-4}
  Ω^{[1]}_{(X°_1, D°_1, q°_1◦β°)}
  = (β°)^{[*]} Ω^{[1]}_{(X°_1, D°_1, q°_1)}
  \overset{\eqref{eq:9-7-3}}{=} (β°)^* Ω¹_{A°_1}.
\end{equation}
The assumption that $φ°$ is a $\cC$-morphism implies $η°$ admits pull-back of
adapted reflexive differentials,
\[
  d η° : (η°)^* Ω^{[1]}_{(X°_2, D°_2, q°_2)} → Ω^{[1]}_{(X°_1, D°_1, q°_1◦β°)},
\]
where
\[
  Ω^{[1]}_{(X°_2, D°_2, q°_2)} \overset{\eqref{eq:9-7-3}}{=} Ω¹_{A°_2} %
  \quad \text{and} \quad %
  Ω^{[1]}_{(X°_1, D°_1, q°_1◦β°)} \overset{\eqref{eq:9-7-4}}{=} (β°)^* Ω¹_{A°_1}.
\]
In particular, we find that the composed pull-back morphism
\[
  \diff η : H⁰\bigl( A_2,\, Ω¹_{A_2}(\log Δ_{A_2}) \bigr) → H⁰\bigl( C,\,
  Ω^{[1]}_C (\log Δ_C) \bigr)
\]
takes its image in
\[
  H⁰\bigl( C,\, Ω^{[1]}_C(\log Δ_C) \bigr) = H⁰\bigl( C,\, Ω^{[1]}_{(A_1, Δ_{A_1}, β)} \bigr).
\]
As above, the universal property of the adapted Albanese will therefore apply to
give a factorization
\[
  \begin{tikzcd}[column sep=2.2cm]
    C° \ar[rr, bend left=10, "η°"] \ar[r, "\alb(A_1{,} Δ_{A_1}{,} β)°"'] & \Alb(A_1, Δ_{A_1}, β)° \ar[r, "Φ°"'] & A°_2,
  \end{tikzcd}
\]
where $Φ°$ is quasi-algebraic.

\subsection*{Step 3: Summary}
\approvals{Erwan & yes \\ Stefan & yes}

We have seen in Steps~1 and 2 that $β°$ and $η°$ both factor via $\alb(A_1,
Δ_{A_1}, β)°$.  The following diagram summarizes the situation,
\[
  \begin{tikzcd}[column sep=2.5cm]
    C° \ar[r, "\alb(A_1{,} Δ_{A_1}{,} β)°"] & \Alb(A_1, Δ_{A_1}, β)° \ar[r, two heads, "ψ°\text{, étale}"] \ar[d, "Φ°\text{, quasi-algebraic}"'] & A°_1 \ar[r, two heads, "q°_1\text{, quotient}"] & X°_1 \ar[d, "φ°"] \\
     & A°_2 \ar[r, equals] & A°_2 \ar[r, two heads, "q°_2\text{, quotient}"] & X°_2.
  \end{tikzcd}
\]
The proof of Theorem~\ref{thm:9-4} is then finished once we set $B° := \Alb(A_1,
Δ_{A_1}, β)°$.  \qed

%
%
\svnid{$Id: 10-orbiAlbanese.tex 941 2024-10-10 16:10:48Z kebekus $}
\selectlanguage{british}

\section{The Albanese of a \texorpdfstring{$\cC$}{𝒞}-pair with bounded irregularity}
\subversionInfo

\subsection{Existence of the Albanese in case of bounded irregularity}
\approvals{Erwan & yes \\ Stefan & yes}

With all the necessary preparation at hand, the main result on the existence of
an Albanese of a $\cC$-pair is now formulated as follows.

\begin{defn}[The Albanese of a $\cC$-pair]\label{def:10-1}%
  Let $(X, D)$ be a $\cC$-pair where $X$ is compact.  An Albanese of $(X,D)$ is
  a $\cC$-semitoric pair $\bigl(\Alb(X,D), Δ_{\Alb(X,D)}\bigr)$ and a
  quasi-algebraic $\cC$-morphism
  \[
    \alb(X,D)° : (X°,D°) → \bigl(\Alb°(X,D), Δ°_{\Alb(X,D)}\bigr)
  \]
  such that the following holds: If $(S, Δ_S)$, $s ∈ S°$ is any other
  $\cC$-semitoric pair and if $s° : (X°,D°) → (S°, Δ°_S)$ is any quasi-algebraic
  $\cC$-morphism, then $s°$ factors uniquely as
  \[
    \begin{tikzcd}[column sep=2.4cm]
      (X°, D°) \ar[r, "\alb(X{,}D)°"'] \ar[rr, "s°", bend left=10] & \bigl(\Alb°_x(X,D), Δ°_{\Alb(X,D)}\bigr) \ar[r, "∃!c°\text{, quasi-algebraic}"'] & (S°, Δ°_S).
    \end{tikzcd}
  \]
\end{defn}

\begin{thm}[The Albanese of a $\cC$-pair with bounded irregularity]\label{thm:10-2} %
  Let $(X, D)$ be a $\cC$-pair where $X$ is compact Kähler.  If $q⁺_{\Alb}(X,D)
  < ∞$, then an Albanese of $(X,D)$ exists.
\end{thm}

Theorem~\ref{thm:10-2} will be shown in Section~\ref{sec:10-3} below.

\begin{rem}[Special pairs]
  Recall from Remark~\vref{rem:8-4} that the assumption $q⁺_{\Alb}(X,D) < ∞$ is
  always satisfied if the $\cC$-pair $(X,D)$ is special.
\end{rem}

\begin{rem}[Uniqueness]
  The universal property implies that $\Alb(X,D)°$ is unique up to unique
  isomorphism and that $\Alb(X,D)$ is bimeromorphically unique.  The universal
  property also implies that $\dim \Alb(X,D) = q⁺_{\Alb}(X,D)$.
\end{rem}

As before, we abuse notation and refer to any Albanese of $(X,D)$ as ``the
Albanese''.

\subsection{Non-existence of the Albanese in case of unbounded irregularity}
\approvals{Erwan & yes \\ Stefan & yes}

Before proving of Theorem~\ref{thm:10-2}, we remark that the assumption
$q⁺_{\Alb}(X,D) < ∞$ is necessary in the strongest possible sense.

\begin{prop}[Non-existence of the Albanese in case of unbounded irregularity]\label{prop:10-5}%
  Let $(X, D)$ be a $\cC$-pair where $X$ is compact Kähler.  If $q⁺_{\Alb}(X,D)
  = ∞$, then an Albanese of $(X,D)$ cannot possibly exist.
\end{prop}
\begin{proof}
  We argue by contradiction and assume that there exists a $\cC$-semitoric pair
  $\bigl(A, Δ_A\bigr)$ and a quasi-algebraic $\cC$-morphism
  \[
    a° : (X°,D°) → \bigl( A°,Δ°_A)
  \]
  that satisfies the universal property of Theorem~\ref{thm:10-2}.  By
  assumption, there exists a cover $γ : \what{X} \twoheadrightarrow X$ such that
  $q_{\Alb}(X,D,γ) > \dim A°$.  Lemma~\vref{lem:5-5} allows assuming without
  loss of generality that $γ$ is Galois with group $G$.
  Proposition~\vref{prop:5-10} yields a diagram
  \begin{equation}\label{eq:10-5-1}
    \begin{tikzcd}[column sep=3cm, row sep=1cm]
      \what{X}° \ar[r, "\alb°(X{,}D{,}γ)"] \ar[d, two heads, "γ"'] & \Alb°(X,D,γ) \ar[d, two heads, "\text{quotient}"] \\
      X° \ar[r, "s°"'] & \factor{\Alb°(X,D,γ)}{G}
    \end{tikzcd}
  \end{equation}
  where $s°$ is a quasi-algebraic morphism of $\cC$-pairs,
  \[
    s° : (X°,D°) → \factor{\bigl(\Alb°(X,D,γ), 0 \bigr)}{G}.
  \]
  By assumption, the $\cC$-morphism $s°$ factors via $a°$, and equips us with a
  quasi-algebraic morphism of $\cC$-pairs,
  \[
    c° : (A°,Δ°_A) → \factor{\bigl(\Alb°(X,D,γ), 0 \bigr)}{G}.
  \]
  Observing that domains and target of the $\cC$-morphism $c°$ are
  $\cC$-semitoric pairs, Theorem~\ref{thm:9-4} yields a semitoric variety
  $(\wcheck{A}, Δ_{\wcheck{A}})$ and an extension of Diagram~\eqref{eq:10-5-1} as
  follows,
  \[
    \begin{tikzcd}[column sep=2.5cm, row sep=1cm]
      \wcheck{X}° \ar[r] \ar[d, two heads, "\text{finite}"'] & \wcheck{A}° \ar[r, "Φ°\text{, quasi-algebraic}"] \ar[dd, two heads, near end, "\text{quotient$\:◦\:$étale}"] & \Alb°(X,D,γ) \ar[d, equal] \\
      \what{X}° \ar[rr, "\alb°(X{,}D{,}γ)", near end, crossing over] \ar[d, two heads, "γ"'] & & \Alb°(X,D,γ) \ar[d, two heads, "\text{quotient}"] \\
      X° \ar[r, "a°"] \ar[rr, bend right=15, "s°"'] & A° \ar[r, "c°"] & \factor{\Alb°(X,D,γ)}{G}
    \end{tikzcd}
  \]
  where $Φ°$ is quasi-algebraic.  Since
  \[
    \dim \wcheck{A}° = \dim A° < \dim \Alb(X,D,γ)
  \]
  by construction, it is clear that $Φ°$ cannot be surjective.  It follows that
  the image of the Albanese morphism $\alb°(X,D,γ)$ is contained in $\img Φ° ⊊
  \Alb°(X,D,γ)$, contradicting the assertion of Proposition~\ref{prop:7-5} that
  the image generates $\Alb(X,D,γ)°$ as an Abelian group, once appropriate Lie
  group structures are chosen.
\end{proof}

\subsection{Proof of Theorem~\ref*{thm:10-2}}
\approvals{Erwan & yes \\ Stefan & yes}
\label{sec:10-3}

We maintain notation and assumptions of Theorem~\ref{thm:10-2}.  The proof is
somewhat long, as it involves the discussion of a fair number of diagrams and
references to almost all results obtained so far.  For the reader's convenience,
we present the argument in four relatively independent steps.

\subsection*{Step 1: Choices and constructions}
\approvals{Erwan & yes \\ Stefan & yes}

We consider the set of Galois covers,
\[
  M := \bigl\{ δ : \what{X}_δ \twoheadrightarrow X \text{ a Galois cover of } (X, D) \bigr\}.
\]
For every $δ ∈ M$, write $G_δ$ for the associated Galois group, write
$\what{X}°_δ := δ^{-1}(X°) ⊆ \what{X}_δ$ and denote the restriction of $δ$ by
$δ° : \what{X}°_δ → X°$.

\begin{choice}[Comparison morphism]
  For every pair of two covers $δ_1, δ_2 ∈ M$ where $δ_1$ factors via $δ_2$,
  \begin{equation}\label{eq:10-6-1}
    \begin{tikzcd}[column sep=2cm]
      \what{X}_{δ_1} \ar[r, two heads, "∃\:δ_{12}"'] \ar[rr, two heads, bend left=15, "δ_1"] & \what{X}_{δ_2} \ar[r, two heads, "δ_2"'] & X,
    \end{tikzcd}
  \end{equation}
  choose one morphism $δ_{12} : \what{X}_{δ_1} \twoheadrightarrow
  \what{X}_{δ_2}$ that makes \eqref{eq:10-6-1} commute.
\end{choice}
  
\begin{choice}[Albanese varieties for the covers]\label{choice:10-7}%
  For every $δ ∈ M$, use Proposition~\ref{prop:5-4} and Corollary~\ref{cor:3-21}
  to choose an Albanese $\bigl( A_δ, Δ_{A_δ}\bigr)$ of the cover $δ$, where the
  $G_δ$-action extends from $A°_δ$ to $A_δ$.  Denote the associated
  quasi-algebraic Albanese morphism by
  \[
    \what{a}°_δ : \what{X}°_δ → A°_δ.
  \]
\end{choice}

\begin{notation}[$\cC$-semitoric pairs]
  With the choices above, consider the $\cC$-semitoric pairs
  \[
    (B_δ, Δ_{B_δ}) := \factor{(A_δ, Δ_{A_δ})}{G_δ}.
  \]
  Let $a°_δ : (X°, D°) → (B°_δ, Δ°_{B_δ})$ be the quasi-algebraic
  $\cC$-morphisms introduced in Proposition~\ref{prop:5-10}.
\end{notation}

In the sequel, we will need to compare the $\cC$-semitoric pairs induced by two
covers that factor one another.  The following reminder summarizes what we
already know.

\begin{reminder}[Comparing covers]\label{remi:10-9}%
  Given two covers $δ_1, δ_2 ∈ M$ where $δ_1$ factors via $δ_2$,
  Lemma~\ref{lem:5-11} equips us with a commutative diagram
  \begin{equation}\label{eq:10-9-1}
    \begin{tikzcd}[column sep=2.2cm, row sep=1cm]
      \what{X}°_{δ_1} \ar[d, two heads, "δ_{12}"] \ar[dd, two heads, bend right=25, "δ°_1"'] \ar[r, "\what{a}°_{δ_1}"] & A°_{δ_1} \ar[rd, two heads, bend left=10, "\what{q}^{\:◦}_{δ_1δ_2}\text{, quasi-algebraic}"] \ar[dd, two heads, near end, "q_{δ_1}\text{, finite quotient}"] \\
      \what{X}°_{δ_2} \ar[rr, near end, crossing over, "\what{a}°_{δ_2}"] \ar[d, two heads, "δ°_2"] && A°_{δ_2} \ar[d, two heads, "q_{δ_2}\text{, finite quotient}"] \\
      X° \ar[r, "a°_{δ_1}"] \ar[rr, bend right=15, "a°_{δ_2}"'] & B°_{δ_1} \ar[r, two heads, "q°_{δ_1δ_2}"] & B°_{δ_2},
    \end{tikzcd}
  \end{equation}
  where all morphisms are quasi-algebraic and all morphisms in the bottom row
  are morphisms of $\cC$-pairs, between $(X°,D°)$, $(B°_{δ_1}, Δ°_{B_{δ_1}})$
  and $(B°_{δ_2}, Δ°_{B_{δ_2}})$.
\end{reminder}

\begin{choice}[Albanese of $(X,D)$]\label{choice:10-10}%
  Consider the numbers
  \begin{align*}
    n_δ & := \# \text{components in the typical non-empty fibre of } a°_δ : X° → B°_δ, \\
    n_{\min} & := \min\bigl\{ n_δ \::\: δ ∈ M \text{ and } \dim B_δ = q⁺_{\Alb}(X,D) \bigr\}
  \end{align*}
  and choose one particular cover $γ ∈ M$ such that $\dim B_γ = q⁺_{\Alb}(X,D)$
  and $n_γ = n_{\min}$.  Once the choice is made, consider the associated
  $\cC$-semitoric pair
  \[
    \bigl(\Alb(X,D), Δ_{\Alb(X,D)}\bigr) := (B_γ, Δ_{B_γ})
  \]
  with the associated morphism $\alb(X,D)° := a°_γ$.  We will show that this is
  an Albanese of $(X,D)$.
\end{choice}

\subsection*{Step 2: First properties of the construction}
\approvals{Erwan & yes \\ Stefan & yes}

We need to show that our choice of an Albanese does indeed satisfy the universal
properties required by Definition~\ref{def:10-1}.  To prepare for the proof, we
study covers $δ ∈ M$ that factor via $γ$.  The following claims show that the
$\cC$-morphism $q°_{δγ} : B°_δ → B°_γ$ of Reminder~\ref{remi:10-9} is an
isomorphism of $\cC$-pairs.  The proof makes extensive use of the notation
introduced in Reminder~\ref{remi:10-9}.  The reader might wish to write down
Diagram~\eqref{eq:10-9-1} in our particular situation, where $δ_1$ is replaced by
$δ$ and $δ_2$ is replaced by $γ$.

\begin{claim}\label{claim:10-11}%
  Assume that a cover $δ ∈ M$ factors via $γ$.  Then, the morphism $q°_{δγ}$ of
  Reminder~\ref{remi:10-9} is finite as a morphism of analytic varieties.
\end{claim}
\begin{proof}[Proof of Claim~\ref{claim:10-11}]
  The choices made in Step~2 guarantee that the quasi-algebraic morphism
  $\what{q}^{\:◦}_{δγ}$ is surjective between semitori of the same dimension.
  It follows from Corollary~\ref{cor:3-19} that $\what{q}^{\:◦}_{δγ}$ is finite.
  As the induced morphism between (finite) Galois quotients, $q°_{δγ}$ is then
  likewise finite.  \qedhere~\mbox{(Claim~\ref{claim:10-11})}
\end{proof}

\begin{claim}\label{claim:10-12}%
  Assume that a cover $δ ∈ M$ factors via $γ$.  Then, the morphism $q°_{δγ}$ of
  Reminder~\ref{remi:10-9} is biholomorphic as a morphism of analytic varieties.
\end{claim}
\begin{proof}[Proof of Claim~\ref{claim:10-12}]
  If $z ∈ \img a°_γ ⊆ B°_γ$ is general, observe that the following two
  conditions hold.
  \begin{description}
    \item[The morphism $q°_{δγ}$ is étale over $z$] We have seen in
    Proposition~\ref{prop:5-10} that $z$ is not contained in the branch locus of
    the finite quotient map $q_γ$.  In other words, $q_γ$ is étale over $z$.
    Corollary~\ref{cor:3-19} guarantees that $\what{q}^{\:◦}_{δγ}$ is étale
    everywhere, so that
    \[
      q_γ ◦ \what{q}^{\:◦}_{δγ} = q°_{δγ} ◦ q_δ
    \]
    is étale over $z$.  But then $q°_{δγ}$ is étale over $z$.

    \item[The set-theoretic fibre $\bigl(q°_{δγ}\bigr)^{-1}(z) ⊂ B°_δ$ is
    connected] This is a direct consequence of Choice~\ref{choice:10-10}.
  \end{description}
  Given that the number of fibre components is constant in finite, étale
  morphisms, we find that the finite morphism $q°_{δγ}$ has connected fibres.
  It is hence a one-sheeted analytic covering in the sense of
  \cite[Sect.~14.2]{MR1326618}.  Together with normality,
  \cite[Prop.~14.7]{MR1326618} applies to show that it is indeed biholomorphic.
  \qedhere~\mbox{(Claim~\ref{claim:10-12})}
\end{proof}

\begin{claim}\label{claim:10-13}%
  Assume that a cover $δ ∈ M$ factors via $γ$.  Then, the morphism $q°_{δγ}$ of
  Reminder~\ref{remi:10-9} is isomorphic as a morphism of $\cC$-pairs.
\end{claim}
\begin{proof}
  Using the biholomorphic map $q°_{δγ}$ to identify the analytic varieties
  $B°_δ$ and $B°_γ$, we need to show that the boundary divisors induced by the
  quotient morphism $q_δ$ and $q_γ$ agree.  The construction of categorical
  $\cC$-pair quotients, \cite[Cons.~12.4]{orbiAlb1}, tells us what the
  boundaries are: if $H_δ$ is any prime divisor in $B°_δ$ and if we choose prime
  divisors in the preimages spaces,
  \begin{align*}
    H_γ & = \bigl(\bigl(q°_{δγ}\bigr)^{-1}\bigr)^* H_δ && \text{prime divisor in } B°_γ \\
    \what{H}_γ & ≤ \bigl(q_γ\bigr)^* H_γ && \text{prime divisor in } A°_γ \\
    \what{H}_δ & ≤ \bigl(\what{q}°_{δγ}\bigr)^* H_γ && \text{prime divisor in } A°_δ,
  \end{align*}
  then
  \[
    \mult_{\cC, H_δ} Δ°_{B_δ} = \mult_{\what{H}_δ} \bigl(q_δ\bigr)^* H_δ
    \quad\text{and}\quad
    \mult_{\cC, H_γ} Δ°_{B_γ} = \mult_{\what{H}_γ} \bigl(q_δ\bigr)^* H_γ.
  \]
  But these two numbers agree, given that $q°_{δγ}$ and $\what{q}°_{δγ}$ are
  étale.  \qedhere~\mbox{(Claim~\ref{claim:10-13})}
\end{proof}

\subsection*{Step 4: Universal property}
\approvals{Erwan & yes \\ Stefan & yes}

We will now show that the constructions of the previous steps satisfy the
universal property spelled out in Definition~\ref{def:10-1}.  We fix the setting
for the remainder of the present proof.

\begin{setting}[Universal property]\label{set:10-14}%
  Let $(B, Δ_B)$ be a $\cC$-semitoric pair and assume that a quasi-algebraic
  $\cC$-morphism $a° : (X°, D°) → (B°, Δ°_B)$ is given.  Let
  \[
    (B, Δ_B) ≅ \factor{(A, Δ_A)}{G}
  \]
  be a presentation of the $\cC$-semitoric pair, with quotient morphism $q : A
  \twoheadrightarrow B$.
\end{setting}

\CounterStep{}We need to show that the $\cC$-morphism $a°$ factors via $a°_γ$
uniquely.  In other words, we need to find a quasi-algebraic $\cC$-morphism $c°$
fitting into a commutative diagram
\begin{equation}\label{eq:10-15-1}
  \begin{tikzcd}[column sep=2.4cm]
    (X°, D°) \ar[r, "a°_γ"'] \ar[rr, "a°", bend left=10] & \bigl(B°_γ, Δ°_{B_γ}\bigr) \ar[r, "∃!c°"'] & (B°, Δ°_B)
  \end{tikzcd}
\end{equation}
and prove that $c°$ is unique with this property.

\subsection*{Step 4a: Existence of a factorization}
\approvals{Erwan & yes \\ Stefan & yes}

Maintaining Setting~\ref{set:10-14}, we show that there exists one
quasi-algebraic $\cC$-morphism $c° : \bigl(B°_γ, Δ°_{B_γ}\bigr) → (B°, Δ°_B)$
that makes Diagram~\eqref{eq:10-15-1} commute.

\begin{construction}[Fibre product]
  Choose a component of the normalized fibre product
  \[
    \what{X}_ρ ⊆ \text{normalization of } A ⨯_B X.
  \]
  Denote the projection morphisms and their restrictions as follows,
  \begin{equation}\label{eq:10-16-1}
    \begin{tikzcd}[row sep=1cm]
      \what{X}_ρ \ar[d, two heads, "ρ"'] \ar[r, phantom, "⊇"] & \what{X}°_ρ \ar[d, two heads, "ρ°"'] \ar[r, "\what{a}^{\:◦}"] & A° \ar[d, two heads, "q°"] \ar[r, phantom, "⊆"] & A \ar[d, two heads, "q\text{, quotient}"] \\
      X \ar[r, phantom, "⊇"] & X° \ar[r, "a°"'] & B° \ar[r, phantom, "⊆"] & B.
    \end{tikzcd}
  \end{equation}
\end{construction}

\begin{obs}[Group actions in \eqref{eq:10-16-1}]
  The group $G$ acts on the fibre product $A ⨯_B X$ and on its normalization.
  The stabilizer of the component $\what{X}_ρ$,
  \[
    H := \operatorname{Stab}\bigl(\what{X}_ρ\bigr) ⊆ G,
  \]
  acts on $\what{X}_ρ$, and $ρ : \what{X}_ρ \twoheadrightarrow X$ is the
  quotient map of this action.  The projection map $ρ$ is therefore Galois.  In
  other words, $ρ ∈ M$.  The Galois group is the quotient of $H$ by the
  ineffectivity,
  \[
    G_ρ = \factor{H}{\ker \bigl( H → \Aut \what{X}_ρ \bigr)}.
  \]
  The projection map $\what{a}^{\:◦}$ is equivariant with respect to the action
  of $H$ on $\what{X}°_ρ$ and on $A°$.
\end{obs}

\begin{obs}[Factorization via the Albanese of the cover]\label{obs:10-18}%
  We have seen in \cite[Obs.~12.10]{orbiAlb1} that the quotient morphism $q°$ is
  an adapted cover for the pair $(B°, Δ°_B)$ and that the adapted differentials
  are described as
  \[
    Ω^{[1]}_{(B°, Δ°_B, q°)} = Ω¹_{(B°, Δ°_B, q°)} = Ω¹_{A°}.
  \]
  The assumption that $b°$ is a $\cC$-morphism guarantees by definition that
  Diagram~\eqref{eq:10-16-1} admits pull-back of adapted reflexive differentials.
  In other words: The composed pull-back morphism
  \[
    (\what{a}°)^* Ω¹_{A°} = (\what{a}°)^* Ω¹_{(B°, B°_S, q°)} \xrightarrow{\diff \what{a}°} Ω¹_{\what{X}°_ρ} → Ω^{[1]}_{\what{X}°_ρ}
  \]
  takes its image in the subsheaf $Ω^{[1]}_{(X°, D°, ρ)} ⊆
  Ω^{[1]}_{\what{X}°_ρ}$.  The universal property of the Albanese for the cover
  $ρ$, Item~\ref{il:5-2-2} of Definition~\ref{def:5-2}, will then guarantee that
  $\what{a}°$ factors as
  \begin{equation}\label{eq:10-18-1}
    \begin{tikzcd}[column sep=2cm, row sep=1cm]
      \what{X}°_ρ \ar[rr, bend left=15, "\what{a}^{\:◦}"] \ar[r, "\what{a}^{\:◦}_ρ"'] & A°_ρ \ar[r, "\what{s}^{\:◦}"'] & A°,
    \end{tikzcd}
  \end{equation}
  where $\what{s}°$ is a quasi-algebraic morphism of semitori.  There is more
  that we can say: $G_ρ$ is the Galois group of the covering morphism $ρ$ and
  therefore acts on $A_ρ$, the Albanese of the cover $ρ$.  The group $G_ρ$ is a
  quotient of $H$, and the universal property of the Albanese $A_ρ$ guarantees
  that the map $\what{a}^{\:◦}$ is equivariant with respect to the $H$-action.
\end{obs}

\begin{claim}[Extension of \eqref{eq:10-16-1} and \eqref{eq:10-18-1}]\label{claim:10-19}%
  There exists a commutative diagram of morphisms between analytic varieties,
  \begin{equation}\label{eq:10-19-1}
    \begin{tikzcd}[row sep=1cm, column sep=2cm]
      \what{X}°_ρ \ar[d, two heads, "ρ°"'] \ar[r, "\what{a}^{\:◦}_ρ"'] \ar[rrr, bend left=15, "\what{a}^{\:◦}"] & A°_ρ \ar[r, "\what{s}^{\:◦}"'] \ar[d, two heads, "q_ρ\text{, quot.~by }G_ρ"] & A° \ar[r, equals] \ar[d, two heads, "q'_ρ\text{, quot.~by }G_ρ"] & A° \ar[d, two heads, "q°\text{, quot.~by }G"]\\
      X°\vphantom{\factor{A°}{G_ρ}} \ar[r, "a°_ρ"] \ar[rrr, bend right=15, "a°"'] & B°_ρ\vphantom{\factor{A°}{G_ρ}} \ar[r,"s°"] & \factor{A°}{H} \ar[r, "β°"] & \factor{A°}{G},
    \end{tikzcd}
  \end{equation}
  where all morphisms in the bottom row are quasi-algebraic $\cC$-morphisms
  between $(X°, D°)$ and the $\cC$-pairs
  \[
    (B°_ρ, Δ°_{B_ρ}) = \factor{(A°_ρ, 0)}{G_ρ} = \factor{(A°_ρ, 0)}H,\quad \factor{(A°, 0)}H \quad\text{and}\quad (B°, Δ°_B) = \factor{(A°, 0)}G.
  \]
\end{claim}
\begin{proof}[Proof of Claim~\ref{claim:10-19}]
  Given that $A°$ is a semitorus hence smooth, recall from
  \cite[Ex.~8.6]{orbiAlb1} that the $H$-equivariant morphism $\what{s}^{\:◦}$ is
  a $\cC$-morphism between the trivial pairs $(A°_ρ, 0)$ and $(A°, 0)$.  Let
  $s°$ be induced by the $\cC$-morphism between the quotient pairs
  \[
    (B°_ρ, Δ°_{B_ρ}) = \factor{(A°_ρ, 0)}H \quad\text{and}\quad \factor{(A°, 0)}{H},
  \]
  as discussed in \cite[Prop.~12.7]{orbiAlb1}.  This morphism makes the middle
  square in \eqref{eq:10-19-1} commute.

  Next, we need to define the morphism $β°$.  For that, review the definition of
  categorical quotients of $\cC$-pairs, \cite[Def.~12.3]{orbiAlb1}.  The
  definition guarantees on the one hand that $q'_ρ$ and $q°$ are $\cC$-morphisms
  between the $\cC$-pairs
  \[
    (A°, 0), \quad \factor{(A°, 0)}{H} \quad\text{and}\quad \factor{(A°, 0)}G.
  \]
  Given that $q°$ is constant on the fibres of $q'_ρ$, the definition also says
  that $q°$ factors via $q'_ρ$, as required.  The induced $\cC$-morphism $β°$
  makes the right square in \eqref{eq:10-19-1} commute.

  It remains to show that $a° = β° ◦ s° ◦ a°_ρ$.  That, however, follows from
  the equality $q° ◦ \what{a}^{\:◦} = a° ◦ ρ°$ given by \eqref{eq:10-16-1}, using
  that $ρ°$ is surjective.  \qedhere~\mbox{(Claim~\ref{claim:10-19})}
\end{proof}

\begin{asswlog}[Factorization of $γ$]\label{asswlog:10-20}%
  The $\cC$-pair $\bigl(B°_γ, Δ°_{B_γ})$ and the morphism $a°_γ$ have been
  defined above using the cover $γ$, but we have seen in Claim~\ref{claim:10-13}
  that they can equally be defined by any cover that factors via $γ$.  Replacing
  $\what{X}_γ$ by the Galois closure of a suitable normalized fibre product, we
  may therefore assume without loss of generality that $γ$ factors via $ρ$,
  \[
    \begin{tikzcd}[column sep=2cm]
      \what{X}_γ \ar[r, two heads] \ar[rr, two heads, bend left=15, "γ"] & \what{X}_ρ \ar[r, two heads, "ρ"'] & X.
    \end{tikzcd}
  \]
\end{asswlog}

With Assumption~\ref{asswlog:10-20} in place, the existence of a factorization is
now immediate.  Reminder~\ref{remi:10-9} decomposes the left square in
\eqref{eq:10-19-1} as follows,
\[
  \begin{tikzcd}[column sep=3cm, row sep=1cm]
    \what{X}°_γ \ar[d, two heads, "γ°"'] \ar[rr, bend left=10, "\what{a}^{\:◦}_ρ"] \ar[r, "\what{a}^{\:◦}_γ"'] & A°_γ \ar[d, two heads, "q°_γ"] \ar[r, two heads, "\what{q}^{\:◦}_{γρ}"'] & A°_ρ \ar[d, two heads, "q°_ρ"]\\
    X° \ar[r, "a°_γ"] \ar[rr, bend right=12, "a°_ρ"'] & B°_γ \ar[r, two heads, "q°_{γδ}"] & B°_ρ
  \end{tikzcd}
\]
where all morphisms are quasi-algebraic and all morphisms in the bottom row are
morphisms of $\cC$-pairs, between $(X°,D°)$, $(B°_γ, Δ°_{B_γ})$ and $(B°_ρ,
Δ°_{B_ρ})$.  We can then set
\[
  c° := β° ◦ s° ◦ q°_{γδ}.
\]
A factorization is thus found.

\subsection*{Step 4b: Uniqueness of the factorization}
\approvals{Erwan & yes \\ Stefan & yes}

Maintain Setting~\ref{set:10-14} and assume that there are two quasi-algebraic
$\cC$-morphisms that makes Diagram~\eqref{eq:10-15-1} commute,
\[
  \begin{tikzcd}[column sep=2.4cm]
    (X°, D°) \ar[r, "a°_γ"'] \ar[rr, "a°", bend left=10] & \bigl(B°_γ, Δ°_{B_γ}\bigr) \ar[r, "∃\:c°_1{,} \: c°_2"'] & (B°, Δ°_B).
  \end{tikzcd}
\]
We need to show that the two morphisms are equal, $c°_1 = c°_2$.

\begin{construction}[Lifting $c°_•$ to Lie group morphisms]\label{cons:10-21}%
  Theorem~\ref{thm:9-4} equips us with semitoric varieties $\what{A}°_{γ,•} ⊂
  \what{A}_{γ,•}$, quasi-algebraic isogenies $i°_• : \what{A}°_{γ,•}
  \twoheadrightarrow A°_γ$ and quasi-algebraic Lie group morphisms $Φ°_• :
  \what{A}°_{γ,•} → A°$ forming commutative diagrams as in \eqref{eq:9-4-1},
  \[
    \begin{tikzcd}[column sep=2.8cm]
      \what{A}°_{γ,•} \ar[d, two heads, "i°_•\text{, quasi-algebraic and étale}"'] \ar[r, "Φ°_•\text{, quasi-algebraic}"] & A° \ar[d, equals]\\
      A°_γ \ar[d, two heads, "q°_γ\text{, quotient}"'] & A° \ar[d, two heads, "q°\text{, quotient}"] \\
      B°_γ \ar[r, "c°_•"'] & B°.
    \end{tikzcd}
  \]
  Blowing up in a left-invariant manner, we may assume without loss of
  generality that the $i°_•$ extend to morphisms $i_• : \what{A}_{γ,•}
  \twoheadrightarrow A_γ$.

  Define a semitoric variety $\what{A}°_γ ⊂ \what{A}_γ$ by choosing a suitable
  strong resolution of a component of the fibre product $\what{A}_{γ,1} ⨯_{A_γ}
  \what{A}_{γ,2}$.  The natural maps
  \[
    \what{A}°_γ \twoheadrightarrow \what{A}°_{γ,•} \twoheadrightarrow A°_γ
  \]
  are then quasi-algebraic and étale.  Compose the projection maps $\what{A}°_γ
  → \what{A}°_{γ,•}$ with $Φ°_•$ to obtain two quasi-algebraic morphisms between
  semitori, $\what{φ}°_• : \what{A}°_γ → A°$, each making the following diagram
  commute,
  \begin{equation}\label{eq:10-21-1}
    \begin{tikzcd}[column sep=2.6cm, row sep=1cm]
      \what{A}°_γ \ar[d, two heads, "\txt{$i°$\scriptsize, quasi-algebraic\\\scriptsize{}étale}"'] \ar[r, "\what{φ}°_•\text{, quasi-algebraic}"] & A° \ar[d, equals]\\
      A°_γ \ar[d, two heads, "q°_γ\text{, quotient}"'] & A° \ar[d, two heads, "q°\text{, quotient}"] \\
      B°_γ \ar[r, "c°_•"'] & B°.
    \end{tikzcd}
  \end{equation}
\end{construction}

\begin{construction}[Dominating $γ$]
  Continuing Construction~\ref{cons:10-21}, choose a component of the normalized
  fibre product $\what{A}_γ ⨯_{A_γ} \what{X}_γ$ and let $\what{X}_δ$ be the
  Galois closure of that component over $X$.  We obtain a Galois cover $δ :
  \what{X}_δ \twoheadrightarrow X$ and a commutative diagram of quasi-algebraic
  morphisms as follows,
  \begin{equation}\label{eq:10-22-1}
    \begin{tikzcd}[column sep=2.8cm, row sep=1cm]
      \what{X}°_δ \ar[dd, two heads, bend right=25, "δ°\text{, Galois cover}"'] \ar[d, two heads] \ar[r, "μ°"] & \what{A}°_γ \ar[d, two heads, "\txt{$i°$\scriptsize, quasi-algebraic\\\scriptsize{}étale}"] \\
      \what{X}°_γ \ar[d, two heads, "γ°"] \ar[r, "\what{a}°_γ"] & A°_γ \ar[d, two heads, "q°_γ\text{, quotient}"] \\
      X° \ar[r, "a°_γ"'] & B°_γ.
    \end{tikzcd}
  \end{equation}
\end{construction}

\begin{obs}[Factorization via the Albanese of the cover]
  In analogy to Observation~\ref{obs:10-18}, recall from
  \cite[Obs.~12.10]{orbiAlb1} that the quotient morphism $q°_γ$ is an adapted
  cover for the pair $(B°_γ, Δ°_{B_γ})$.  Since $i°$ is étale, the adapted
  differentials are described as
  \[
    Ω^{[1]}_{(B°_γ, Δ°_{B_γ}, q°_γ◦i°)} = Ω¹_{\what{A}°_γ}.
  \]
  Since $a°_γ$ is a $\cC$-morphism, Diagram~\eqref{eq:10-22-1} admits pull-back
  of adapted reflexive differentials.  The composed pull-back morphism
  \[
    (μ°)^* Ω¹_{\what{A}°_γ} = (μ°)^* Ω¹_{(B°_γ, Δ°_{B_γ}, q°_γ◦i°)} \xrightarrow{\diff μ°} Ω¹_{\what{X}°_δ} → Ω^{[1]}_{\what{X}°_δ}
  \]
  therefore takes its image in the subsheaf $Ω^{[1]}_{(X°, D°, δ°)} ⊆
  Ω^{[1]}_{\what{X}°_δ}$.  The universal property of the Albanese for the cover
  $δ°$, Item~\ref{il:5-2-2} of Definition~\ref{def:5-2}, will then guarantee
  that $\what{a}°$ factors as follows,
  \begin{equation}\label{eq:10-23-1}
    \begin{tikzcd}[column sep=2.6cm, row sep=1cm]
      \what{X}°_ρ \ar[rr, bend left=15, "μ°"] \ar[r, "\what{a}°_δ"'] & A°_δ \ar[r, "ν°\text{, quasi-algebraic}"'] & \what{A}°_γ.
    \end{tikzcd}
  \end{equation}
\end{obs}

\begin{leansum}
  Combining \eqref{eq:10-21-1}, \eqref{eq:10-22-1} and \eqref{eq:10-23-1}, the
  following two diagrams summarize the constructions obtained so far,
  \[
    \begin{tikzcd}[column sep=2.8cm, row sep=1cm]
      \what{X}°_δ \ar[dd, two heads, bend right=25, "δ°"'] \ar[d, two heads] \ar[r, "\what{a}^{\:◦}_δ"'] \ar[rr, bend left=15, "μ°"] & A°_δ \ar[d, two heads, "q°_{δγ}"] \ar[r, "ν°\text{, quasi-algebraic}"'] & \what{A}°_γ \ar[d, two heads, "\txt{$i°$\scriptsize, quasi-algebraic\\\scriptsize{}étale}"] \ar[r, "\what{φ}^{\:◦}_•\text{, quasi-algebraic}"] & A° \ar[d, equals]\\
      \what{X}°_γ \ar[d, two heads, "γ°"] \ar[r, "\what{a}^{\:◦}_γ"] & A°_γ \ar[d, two heads, "q°_γ"] \ar[r, equal] & A°_γ \ar[d, two heads, "q°_γ"] & A° \ar[d, two heads, "q°"]\\
      X° \ar[r, "a°_γ"] \ar[rrr, bend right=15, "a°"'] & B°_γ \ar[r, equal] & B°_γ \ar[r, "c°_•"] & B°.
    \end{tikzcd}
  \]
  Setting $\what{c}^{\:◦}_• := \what{φ}^{\:◦}_• ◦ ν°$, we are interested in the
  subdiagrams
  \begin{equation}\label{eq:10-24-1}
    \begin{tikzcd}[column sep=3cm, row sep=1cm]
      \what{X}°_δ \ar[d, two heads, "δ°"'] \ar[r, "\what{a}^{\:◦}_δ"] & A°_δ \ar[d, two heads, "q°_γ◦ q°_{δγ}"] \ar[r, "\what{c}^{\:◦}_•\text{, quasi-algebraic}"] & A° \ar[d, two heads, "q°"]\\
      X° \ar[r, "a°_γ"] \ar[rr, bend right=12, "a°"'] & B°_γ \ar[r, "c°_•"] & B°.
    \end{tikzcd}
  \end{equation}
\end{leansum}

\begin{claim}[Aligning the $\what{c}^{\:◦}_•$]\label{claim:10-25}%
  Recalling that $G$ is the Galois group of the morphism $q$, there exists an
  element $g ∈ G$ such that $\what{c}^{\:◦}_1 ◦ \what{a}^{\:◦}_δ = g ◦
  \what{c}^{\:◦}_2 ◦ \what{a}^{\:◦}_δ$.
\end{claim}
\begin{proof}[Proof of Claim~\ref{claim:10-25}]
  For every point of $x ∈ \what{X}°_δ$, commutativity of \eqref{eq:10-24-1}
  guarantees that the image points $(\what{c}^{\:◦}_• ◦ \what{a}^{\:◦}_δ) (x)$
  are contained in the same fibre of the Galois morphism $q°$.  Accordingly,
  there exists an element $g_x ∈ G$ such that
  \begin{equation}\label{eq:10-25-1}
    \what{c}^{\:◦}_1 ◦ \what{a}^{\:◦}_δ (x) = g_x ◦ \what{c}^{\:◦}_2 ◦ \what{a}^{\:◦}_δ (x).
  \end{equation}
  But since $G$ is finite, there exists one $g ∈ G$ such that \eqref{eq:10-25-1}
  holds for every $x ∈ \what{X}°_δ$.  \qedhere~\mbox{(Claim~\ref{claim:10-25})}
\end{proof}

To conclude, choose one Galois element $g ∈ G$ as in Claim~\ref{claim:10-25}.
Choose a point $x ∈ \what{X}°_δ$ and use
\[
  0_{A°_δ} := \what{a}^{\:◦}_δ(x)
  \quad\text{and}\quad
  0_{A°} := \what{c}^{\:◦}_1(0_{A°_δ})
\]
to equip $A°_δ$ and $A°$ with Lie group structures.  With these structures,
Proposition~\ref{prop:3-18} guarantees that
\[
  \what{c}^{\:◦}_1 : A°_δ → A°, \quad
  g ◦ \what{c}^{\:◦}_2 : A°_δ → A°, \quad\text{and}\quad
  g : A°_δ → A°_δ
\]
are Lie group morphisms, so that
\[
  \img \what{a}^{\:◦}_δ ⊆ \ker \bigl(\what{c}_1 - g ◦ \what{c}_2\bigr) ⊆ A°_δ.
\]
Recalling from Proposition~\ref{prop:5-4} that $\img \what{a}^{\:◦}_δ$ generates
$A°_δ$ as a group, we find that $\what{c}_1 = g ◦ \what{c}_2$.  Commutativity of
\eqref{eq:10-21-1} and surjectivity of $q°_δ$ then show that $c°_1 = c°_2$, as
required to finish the proof of Theorem~\ref{thm:10-2}.  \qed

%
%
\svnid{$Id: 11-problems.tex 939 2024-10-10 07:50:05Z kebekus $}
\selectlanguage{british}

\section{Problems and open questions}
\subversionInfo

\subsection{A weak Albanese for arbitrary \texorpdfstring{$\cC$}{𝒞}-pairs}
\approvals{Erwan & yes \\ Stefan & yes}
\label{sec:11-1}

If $(X, D)$ is a $\cC$-pair $q⁺_{\Alb}(X,D) = ∞$, then we have seen in
Proposition~\ref{prop:10-5} that an Albanese in the sense of
Definition~\ref{def:10-1} cannot possibly exist.  While it might be possible to
define a meaningful Albanese as an ind-variety or as a (yet to be defined)
ind-pair, we are convinced that a weak version of the Albanese does exist within
the world of classical $\cC$-pairs.  For many practical purposes, the following
definition might be just as useful as Definition~\ref{def:10-1} above.

\begin{defn}[The weak Albanese of a $\cC$-pair]\label{def:11-1}%
  Let $(X, D)$ be a $\cC$-pair where $X$ is compact.  A weak Albanese of $(X,D)$
  is a normal analytic variety $Z°$ and a morphism
  \[
    \walb(X,D)° : X° → Z°
  \]
  such that the following universal property holds.  If $(S, Δ_S)$ is any
  $\cC$-semitoric pair and if $s° : (X°,D°) → (S°, Δ°_S)$ is any quasi-algebraic
  $\cC$-morphism, then there exists a $\cC$-semitoric pair $(A, Δ_A)$, $a ∈ A°$
  and a commutative diagram of morphisms between analytic varieties
  \[
    \begin{tikzcd}[column sep=2cm]
      X° \ar[r, "\walb(X{,}D)°"'] \ar[rrr, "s°", bend left=10] & Z° \ar[r, "ι"'] & A° \ar[r, "c"'] & S°,
    \end{tikzcd}
  \]
  such that the following holds.
  \begin{enumerate}
    \item The morphism $ι$ is a generically injective.

    \item The composed morphism $ι◦\walb(X,D)°$ is a quasi-algebraic
    $\cC$-morphism between the pairs $(X°, D°)$ and $(A°, Δ°_A)$.

    \item The morphism $c$ is a quasi-algebraic $\cC$-morphism between the pairs
    $(A°, Δ°_A)$ and $(S°, Δ°_S)$.
  \end{enumerate}
\end{defn}

Comparing Definitions~\ref{def:10-1} and \ref{def:11-1}, one sees that
Definition~\ref{def:10-1} does not define the Albanese as a $\cC$-semitoric pair.
Instead, it defines $Z°$ as the normalized image of a hypothetical Albanese
variety and leaves the precise embedding of $Z°$ into a $\cC$-semitoric pair
unspecified.  In this sense, Definition~\ref{def:10-1} does not define the
Albanese map.  It defines its image.

\begin{conj}[The weak Albanese of a $\cC$-pair]\label{conj:11-2} %
  Let $(X, D)$ be a $\cC$-pair where $X$ is compact Kähler.  Then, a weak
  Albanese of $(X,D)$ exists.  The variety $Z°$ and the morphism $\walb(X,D)°$
  are unique up to unique isomorphism.  The group $\Aut_𝒪(X)$ acts on $Z°$ in a
  way that makes the morphism $\walb(X,D)°$ equivariant.
\end{conj}

With the tools available, we believe that Conjecture~\ref{conj:11-2} can be
shown without much trouble, but we fear that a full proof would either add
another ten pages to this already long paper or needs to be integrated into the
proof of Theorem~\ref{thm:10-2}, extending that proof further and rendering it
potentially unfit for human consumption.  We will therefore restrict ourselves
to the short sketch below and publish details elsewhere.

\begin{proof}[Sketch of proof]
  The proof of Theorem~\ref{thm:10-2} carries over in large parts.  The main
  difference is Claim~\ref{claim:10-13}, which asserts that the quotient
  varieties $B°_δ$ stabilize as soon as we pass to sufficiently fine covers,
  that is, to covers that factorize via $γ$.  In absence of the hypothesis
  $q⁺_{\Alb}(X,D) < ∞$ this cannot possibly hold true.  Instead, we claim that
  it is not the $B°_δ$ that stabilize, but it is the image sets $\img a°_δ ⊆
  B°_δ$ that does.
  
  In order to make this precise, one needs to modify Choice~\ref{choice:10-10}.
  At the end of Step~1, consider the numbers
  \begin{align*}
    d_δ & := \# \text{dimension of } \overline{\img a°_δ} ⊆ B_δ, \\
    d_{\min} & := \min\bigl\{ d_δ \::\: δ ∈ M \bigr\} \\
    n_δ & := \# \text{components in the typical non-empty fibre of } a°_δ : X° → B°_δ, \\
    n_{\min} & := \min\bigl\{ n_δ \::\: δ ∈ M \text{ and } d_δ = d_{\min} \bigr\}
  \end{align*}
  and choose one particular cover $γ ∈ M$ such that $d_γ = d_{\min}$ and $n_γ =
  n_{\min}$.  Once the choice is made, take
  \[
    Z° := \text{normalization of } \overline{\img a°_γ}
  \]
  and let $\walb(X,D)° : X° → Z°$ be the morphism obtained by applying the
  universal property of the normalization map to the restricted morphism $a°_δ :
  X° → Z°$.  With the appropriate modifications, the required universal
  properties of $\walb(X,D)°$ will follow by arguments analogous to those in
  Step~4 of the proof of Theorem~\ref{thm:10-2}.
\end{proof}

\subsection{Inequalities between augmented irregularities}
\approvals{Erwan & yes \\ Stefan & yes}
\label{sec:11-2}

\CounterStep{}If $(X, D)$ is a $\cC$-pair where $X$ is compact Kähler, then
\eqref{eq:6-2-1} immediately gives an inequality between the augmented
irregularities
\begin{equation}\label{eq:11-3-1}
  q⁺_{\Alb}(X,D) ≤ q⁺(X,D).
\end{equation}
We do not understand the meaning of Inequality~\eqref{eq:11-3-1}.  We do not
know if \eqref{eq:11-3-1} can ever be strict.  If it is strict, this means that
adapted differentials come in two types.
\begin{itemize}
  \item A subset of the adapted differentials comes from the Albanese morphism,
  at least after passing to suitable high covers.

  \item The general adapted differential is not induced by any morphism to a
  $\cC$-semitoric pair.
\end{itemize}
We do not understand this distinction and wonder if there is a geometric
explanation, perhaps in Hodge-theoretic terms.


\appendix
\addcontentsline{toc}{part}{Appendix}

%
%
\svnid{$Id: 0A-moduli.tex 949 2024-11-11 12:46:09Z kebekus $}
\selectlanguage{british}

\section{Moduli spaces of covers and the Prym map}
\subversionInfo
\approvals{Erwan & yes \\ Stefan & yes}

For the reader's convenience, this appendix recalls several facts and
constructions from the moduli theory of branched curve covers.  The results are
used in Section~\ref{sec:6}, where we compare the irregularity of a curve cover
with its Albanese irregularity.  We refer the reader to the first sections of
Mumford's classic paper \cite{MR379510} and to the introductory chapters of
\cite{MR2956039, MR3877472} for details and references.

\begin{fact}[Coarse moduli space of 2:1-covers]
  Given numbers $g ∈ ℕ$ and $r ∈ 2·ℕ$, there exists a scheme $\mathcal R_g(2,r)$
  of dimension
  \begin{equation}\label{eq:A-1-1}
    \dim \mathcal R_g(2,r) = \dim \mathcal M_g + r = 3·g - 3 + r
  \end{equation}
  parametrizing isomorphism classes of triples $(X, B, ℒ)$, where
  \begin{itemize}
    \item $X$ is a smooth curve of genus $g$,
    \item $B ∈ \Div X$ is a reduced divisor of degree $r$, and
    \item $ℒ ∈ \Pic(X)$ is a line bundle where $ℒ^{⊗2} ≅ 𝒪_X(B)$.
  \end{itemize}
  We refer to $\mathcal R_g(2,r)$ as the \emph{coarse moduli space of 2:1-covers
  over curves of genus $g$, branched in exactly $r$ points}.
\end{fact}

The name ``moduli space of covers'' is justified by the following remark, which
relates triples $(X, B, ℒ)$ to actual covers of $X$.

\begin{rem}[The space $\mathcal R_g(2,r)$ as a moduli space of covers]\label{rem:A-4}
  Recall the classic fact that to give an element of $\mathcal R_{g,r}$, it is
  equivalent to give a genus-$g$ curve together with a 2:1-cover branched in
  exactly $r$ points.
  \begin{itemize}
    \item Given a triple $(X, B, ℒ)$, follow \cite[Sect.~1]{MR379510} and use
    $B$ to equip the coherent $𝒪_X$-module $𝒪_X ⊕ ℒ^*$ with an $𝒪_X$-algebra
    structure.  The natural morphism
    \[
      γ : \operatorname{Spec} \left( 𝒪_X ⊕ℒ^* \right) \twoheadrightarrow X
    \]
    is then a 2:1-cover between smooth curves, branched exactly over the support
    of $R$.

    \item Given a genus-$g$ curve $X$ and a 2:1-cover $γ : \what{X} → X$ with a
    branch divisor $B ∈ \Div X$ of degree $r$, use the action of the Galois
    group to decompose the push-forward of $𝒪_{\what{X}}$ into an invariant and
    an anti-invariant part,
    \[
      γ_* 𝒪_{\what{X}} ≅ 𝒪_X ⊕ ℒ^*.
    \]
    Observe that $ℒ$ is a line bundle with $ℒ^{⊗2} ≅ 𝒪_X(B)$.
  \end{itemize}
\end{rem}

\begin{fact}[The norm map and Prym variety of a branched cover]\label{fact:A-5}%
  Let $γ : \what{X} \twoheadrightarrow X$ be a 2:1-cover of curves and let $N(γ)
  : \operatorname{Jac}(\what X) → \operatorname{Jac}(X)$ be the associated norm
  map, which is induced by push-forward of divisors.  If $γ$ is branched, then
  $P_γ := \ker N(γ)$ is a connected Abelian variety of dimension
  \[
    \textstyle \dim P_γ = g(X) - 1 + \frac{1}{2}·\deg \Branch γ,
  \]
  and the theta-divisor on $\operatorname{Jac}(\what{X})$ restricts to a
  polarization of type $D := (1, …, 1, 2, …, 2)$, where '2' is repeated $g(X)$
  times.  \qed
\end{fact}

\begin{fact}[The Prym map]\label{fact:A-6}%
  Given numbers $g ∈ ℕ$ and $r ∈ 2·ℕ ∖ \{0\}$, Fact~\ref{fact:A-5} yields a
  morphism of algebraic varieties,
  \[
    \operatorname{Pr}_g(2, r) : {\mathcal R}_g (2,r) → {\mathcal A}_{p,D},
  \]
  where $p := g - 1 + \frac{1}{2}·r$ and ${\mathcal A}_{p,D}$ is the moduli
  space of $p$-dimensional polarized Abelian varieties with polarization of type
  $D := (1, …, 1, 2, …, 2)$.  \qed
\end{fact}

\begin{fact}[The moduli space ${\mathcal A}_{p,D}$]\label{fact:A-7}%
  The moduli space ${\mathcal A}_{p,D}$ of $p$-dimensional polarized Abelian
  varieties with polarization of type $D$ is of dimension $\frac{1}{2}·p·(p+1)$.
  \qed
\end{fact}


\end{document}